\documentclass[dvipdfmx]{amsart}

\makeatletter
\@addtoreset{equation}{section}

\makeatother

\usepackage{amsfonts, amsmath, amssymb, amsthm}
\usepackage{url}
\usepackage[dvipdfmx]{graphicx}
\usepackage{color}
\usepackage{amscd}
\usepackage[right]{lineno}
\usepackage{enumerate}

\newtheorem{theorem}{Theorem}[section]
\newtheorem{lemma}{Lemma}[section]
\newtheorem{proposition}{Proposition}[section]
\newtheorem{corollary}{Corollary}[section]
\theoremstyle{remark}

\newtheorem{remark}{Remark}[section]

\newtheorem*{convention}{Convention}
\newtheorem{question}{Question}
\newtheorem{problem}{Problem}
\newtheorem{claim}{Claim}

\newcommand{\teich}{\mathcal{T}}

\newcommand{\ext}{{\rm Ext}}
\newcommand{\extsymp}{\boldsymbol{\epsilon}}

\newcommand{\mcg}{{\rm Mod}}

\newcommand{\gromov}[3]{\langle #1\,|\,#2\rangle_{#3}}
\newcommand{\cl}[1]{{\rm cl}_{GM}(#1)}

\newcommand{\APT}{{\rm APT}}

\newcommand{\veccdot}{}

\newcommand{\Bers}[1]{\mathcal{T}^B_{#1}}

\newcommand{\harmonicmeasure}{{\boldsymbol \omega}}

\newcommand{\ThursM}{\mu_{Th}}
\newcommand{\PThursM}{\hat{\mu}_{Th}}
\newcommand{\HSmeas}{{\bf m}}
\newcommand{\extrecip}{{\bf u}}
\newcommand{\convgenus}{{\boldsymbol \xi}}
\newcommand{\proje}{\boldsymbol{\psi}}

\newcommand{\pushThursMBers}{{\boldsymbol \mu}^B}

\newcommand{\Levi}[3]{\mathcal{L}(#1)[#2,#3]}

\newcommand{\repro}{\eta}

\makeindex

\begin{document}

\title[Poisson integral formula]{Pluripotential theory on Teichm\"uller space II \\ -- Poisson integral formula --}
\author{Hideki Miyachi}

\date{\today}
\address{School of Mathematics and Physics,
College of Science and Engineering,
Kanazawa University,
Kakuma-machi, Kanazawa,
Ishikawa, 920-1192, Japan
}
\email{miyachi@se.kanazawa-u.ac.jp}
\thanks{
t
This work is partially supported by JSPS KAKENHI Grant Numbers
20H01800,
20K20519,
18KK0071,
and was partially supported by JSPS KAKENHI Grant Numbers
16K05202,
16H03933,
17H02843}

\begin{abstract}
This is the second paper in a series of investigations of the pluripotential theory on Teichm\"uller space. The main purpose of this paper is to establish the Poisson integral formula for pluriharmonic functions on Teichm\"uller space which are continuous on the Bers compactification. We also observe that the Schwarz type theorem on the boundary behavior of the Poisson integral. We will see a relationship between the pluriharmonic measures and the Patterson-Sullivan measures discussed by Athreya, Bufetov, Eskin and Mirzakhani.
\end{abstract}

\maketitle
\tableofcontents

\section{Introduction}
\label{sec:Introduction}
This is the second paper in a series of investigations of the pluripotential theory on Teichm\"uller space.
The first paper in the series is \cite{miyachi-pluripotentialtheory1} in which we discussed an alternative approach to the Krushkal formula (\cite{MR1142683}) of the pluricomplex Green function on the Teichm\"uller space (cf. \eqref{eq:Green-function}).
The main purpose of this paper is to establish the Poisson integral formula for pluriharmonic functions on Teichm\"uller space which are continuous on the Bers compactification.
This result is announced in \cite{ProceedingsThurstoni2018} and \cite{Oberwolfach-report-miyachi}.

\subsection{Classical Poisson integral formula and a dictionary}
It is well-known that any continuous function $u$ on the closed upper half-plane $\overline{\mathbb{H}}=\mathbb{H}\cup (\mathbb{R}\cup \{\infty\})$ which is harmonic on the upper-half plane $\mathbb{H}$ satisifies
\begin{equation}
\label{eq:Poisson_integral_formula_up}
u(z)=\int_{\partial \mathbb{H}}P(z,\xi)u(\xi)\dfrac{d\xi}{\pi},
\end{equation}
for $z=x+iy\in \mathbb{H}$, where $\partial \mathbb{H}=\mathbb{R}\cup \{\infty\}$ and
$$
P(z,\xi)=\dfrac{{\rm Im}(z)}{|z-\xi|^2}.
$$
The integral representation \eqref{eq:Poisson_integral_formula_up} is called the \emph{Poisson integral formula}. Meanwhile, it is also well-known that the Teichm\"uller space $\teich_1$ of tori is identified with the upper-half plane $\mathbb{H}$ via the period map. In recognizing the Poisson integral formula \eqref{eq:Poisson_integral_formula_up} as the formula for the (pluri)harmonic functions on the Teichm\"uller space $\teich_1$, we obtain a dictionary as Table \ref{table:dictionary} in the case of tori (i,e, $(g,m)=(1,0)$) (cf. \S\ref{subsec:Case_tori}). In this paper, we will justify the correspondence in Table \ref{table:dictionary} for arbitrary $(g,m)$, as we state in Theorem \ref{thm:Poisson-integral-formula}.
\begin{table}
\begin{tabular}{|c|c|}\hline
Upper half-plane $\mathbb{H}$ & Teichm\"uller space $\teich_{g,m}$ \\ \hline\hline
Harmonic function & Pluriharmonic function \\ \hline
Ideal boundary $\partial \mathbb{H}$ & Bers boundary \\ \hline
Green function & $\log \tanh$ of the Teichm\"uller distance \\ \hline
Horofunctions (Busemann functions) & log of extremal lengths \\ \hline
Poisson kernel & Ratio of extremal lengths \\ \hline
Harmonic measure on $\partial \mathbb{H}$ & 
\begin{tabular}{c}
Pushforward measure of \\
the normalized Thurston measure
\end{tabular}
\\ \hline
\end{tabular}
\caption{A dictionary}
\label{table:dictionary}
\end{table}

\subsection{Results}
\label{subsec:resultssec}

Let $\teich_{g,m}$ be the Teichm\"uller space
of type $(g,m)$.
Let $\Bers{x_0}$ be the Bers slice with base point $x_0 \in \teich_{g,m}$.
Krushkal \cite{MR1119946} showed that Teichm\"uller space is hyperconvex.
By the Nehari-Kraus theorem, the Bers slice $\Bers{x_0}$ is a bounded domain in a finite dimensional complex Banach space (cf. \cite{MR0130972}).
Demailly \cite{MR881709}  establishes fundamental results in the pluripotential theory, the existence of the pluricomplex Green functions and the pluriharmonic measures for bounded hyperconvex domains in the complex Euclidean space (see \S\ref{subsec:result-from-pluripotential-theory}).

\subsubsection{Results}
\label{subsubsec:results}
Let $\partial \Bers{x_0}$ be the Bers boundary of $\Bers{x_0}\cong \teich_{g,m}$ and $\partial^{ue}\Bers{x_0}$ the subset of $\partial \Bers{x_0}$ which consists of totally degenerate groups without APT whose ending laminations are the supports of minimal, filling and uniquely ergodic measured laminations.
We define  a function $\mathbb{P}$ on $\teich_{g,m}\times \teich_{g,m}\times \partial \Bers{x_0}$ by
\begin{equation}
\label{eq:Poisson-kernel}
\mathbb{P}(x,y,\varphi)=
\begin{cases}
{\displaystyle
\left(
\frac{\ext_x(F_\varphi)}{\ext_y(F_\varphi)}
\right)^{3g-3+m}}
&
(\varphi\in \partial^{ue}\Bers{x_0}),
\\
1
&
(\mbox{otherwise}),
\end{cases}
\end{equation}
where for $\varphi\in \partial^{ue} \Bers{x_0}$, $F_\varphi$ is the measured foliation corresponding to the measured lamination whose support is the ending lamination of the Kleinian manifold associated to $\varphi$, and $\ext_{x}(F)$ is the extremal length of a measured foliation $F$ on $x\in \teich_{g,m}$ (cf. \S\ref{subsec:extremal-length}).

The main result of this paper is as follows.

\begin{theorem}[Poisson integral formula]
\label{thm:Poisson-integral-formula}
Let $V$ be a continuous function on the Bers compactification $\overline{\Bers{x_0}}$ which is pluriharmonic on $\Bers{x_0}\cong \teich_{g,m}$.
Then
\begin{equation}
\label{eq:main-Poisson-integral-formula}
V(x)=\int_{\partial \Bers{x_0}}V(\varphi)
\mathbb{P}(x_0,x,\varphi)
d\pushThursMBers_{x_0}(\varphi),
\end{equation}
where $\pushThursMBers_{x_0}$ is the probability measure on $\partial\Bers{x_0}$ defined as the pushforward measure of the Thurston measure on the space $\mathcal{PMF}$ of projective measured foliations associated with $x_0$.
Especially, the function \eqref{eq:Poisson-kernel} is the Poisson kernel for pluriharmonic functions on Teichm\"uller space.
\end{theorem}
See \S\ref{sec:Thurston-measure} for the precise definition of the probability measure $\pushThursMBers_{x_0}$ on $\partial\Bers{x_0}$.
Theorem \ref{thm:Poisson-integral-formula} follows from the \emph{Green formula for plurisubharmonic functions} stated in Theorem \ref{thm:Poisson-integral-formula-sub}.
We see that the measure $d\pushThursMBers_{x}=\mathbb{P}(x_0,x,\cdot)d\pushThursMBers_{x_0}$ coincides with Demailly's pluriharmonic measure of $x\in \teich_{g,m}$ (cf. Theorem \ref{thm:PHmeasure-is-Thurston-measure}). The formula \eqref{eq:main-Poisson-integral-formula} is rephrased as the integrable representation of integral functions on $\mathcal{PMF}$ (cf. \S\ref{subsec:integral-representation-with-pmf}).

Following \cite[\S2.3.1]{MR2913101},
we define the \emph{cocycle function}
by
$$
\beta(x,y,\varphi)=
\begin{cases}
{\displaystyle
\frac{1}{2}(\log\ext_x(F_\varphi)-\log\ext_y(F_\varphi))
}
&
(\varphi\in \partial^{ue}\Bers{x_0}),
\\[2mm]
0
&
(\mbox{otherwise})
\end{cases}
$$
for $(x,y,\varphi)\in \teich_{g,m}\times \teich_{g,m}\times \partial \Bers{x_0}$.
The cocycle function $\beta$ is also understood as the \emph{horofunction} for the Teichm\"uller distance when $\varphi\in \partial^{ue}\Bers{x_0}$ (cf. \cite{MR3289706}. See also \cite{MR3289702} and \cite{MR3278905}).
The Poisson integral formula \eqref{eq:main-Poisson-integral-formula} is rewritten as
\begin{equation}
\label{eq:main-Poisson-integral-formula-Buseman}
V(x)=\int_{\partial \Bers{x_0}}V(\varphi)
e^{-(6g-6+2m)\beta(x,x_0,\varphi)}
d\pushThursMBers_{x_0}(\varphi).
\end{equation}
The formulation \eqref{eq:main-Poisson-integral-formula-Buseman} implies that Demailly's pluriharmonic measures $\{\pushThursMBers_{x}\}_{x\in \teich_{g,m}}$ are thought of as the conformal density of dimension $6g-6+2m$ on $\partial\Bers{x_0}$ (cf. \cite[\S2.3.1]{MR2913101} and \cite{MR2495764}. See also Remark \ref{remark:PS-measure}).
This observation is in complete analogy with that in the case of the hyperbolic spaces (cf. \S\ref{subsec:Case_tori} See also \cite{ProceedingsThurstoni2018} and \cite[Theorem B]{MR1689341}).


In the proof of Theorem \ref{thm:Poisson-integral-formula}, we realize Teichm\"uller space as a convex cone in the $6g-6+2m$ dimensional Euclidean space (cf. \S\ref{sec:holomorphic-coordinate-ext}).
We give an explicit presentation of the complex structure on the convex cone which makes the realization biholomorphic (cf. \eqref{eq:ComplexStructure} and Proposition \ref{prop:holomorphic-chart}). The $\partial$ and $\overline{\partial}$-derivatives and the Levi form of extremal lengths in Proposition \ref{prop:derivatives-extremallengthfunction} are calculated with this complex structure (cf.  \cite{Towards-complex-miyachi2017}). An advantage of this realization is that the Monge-Amp\`ere measure of the extremal length function is simply represented (cf. \eqref{ddcNExt}).
Riera \cite{MR3394254} described the complex structure on Teichm\"uller space in terms of the Fenchel-Nielsen coordinates. Our presentation is thought of as a counterpart of Riera's one.

Schwarz \cite{MR1579542} studied the behavior of the Poisson integral of integrable functions on the unit circle around points where given functions are continuous (see also \cite[Theorem IV.2]{MR0114894}).
We will observe an analogy with Schwarz's theorem as follows (cf. \S\ref{sec:boundary-behaviorPI}).

\begin{theorem}[Schwarz type theorem]
\label{thm:boundary-behavior}
Let $V$ be an integrable function on $\partial\Bers{x_0}$ with respect to $\pushThursMBers_{x_0}$.
When $V$ is continuous at $\varphi_0\in \partial^{ue}\Bers{x_0}$,
$$
\lim_{\teich_{g,m}\ni x\to \varphi_0}\int_{\partial \Bers{x_0}}V(\varphi)
\mathbb{P}(x_0,x,\varphi)
d\pushThursMBers_{x_0}(\varphi)=V(\varphi_0).
$$
\end{theorem}
As a corollary, we deduce
\begin{corollary}[Holomorphic extension]
\label{thm:boundary-value}
Let $V$ be a complex-valued integrable function on $\partial\Bers{x_0}$ with respect to $\pushThursMBers_{x_0}$.
Suppose that $V$ is continuous on $\partial^{ue}\Bers{x_0}$.
If
\begin{equation}
\label{eq:CR}
\int_{\partial \Bers{x_0}}V(\varphi)
\overline{\partial}\mathbb{P}(x_0,\,\cdot\,,\varphi)
d\pushThursMBers_{x_0}(\varphi)=0
\end{equation}
on $\teich_{g,m}$ as $(0,1)$-forms,
then the Poisson integral
$$
P[V](x)=\int_{\partial \Bers{x_0}}V(\varphi)
\mathbb{P}(x_0,x,\varphi)
d\pushThursMBers_{x_0}(\varphi)
$$
is a holomorphic function on $\teich_{g,m}$ and satisfies
$$
\lim_{\teich_{g,m}\ni x\to \varphi_0}P[V](x)=V(\varphi_0)
\quad (\varphi_0\in \partial^{ue}\Bers{x_0}).
$$
\end{corollary}
In the complex function theory on $\teich_{g,m}$, Equation \eqref{eq:CR} will play as the \emph{homogeneous tangential Cauchy-Riemann equation} (in the distribution sense) for the boundary functions of holomorphic functions on Teichm\"uller space, which characterizes the boundary functions of holomorphic functions (cf. \cite{MR1211412}). 

Notice from Theorem \ref{thm:Poisson-integral-formula} that continuous functions on $\partial\Bers{x_0}$ with holomorphic extensions on $\teich_{g,m}$ satisfy \eqref{eq:CR}.
Unlike in the case of the unit disk, it is not clear whether the Poisson integral 
$$
\int_{\partial \Bers{x_0}}V(\varphi)
\mathbb{P}(x_0,x,\varphi)
d\pushThursMBers_{x_0}(\varphi)
$$
is pluriharmonic on $\teich_{g,m}$ for any integrable function $V$ on $\partial\Bers{x_0}$ (cf. Remark \ref{remark:Levi-Poisson}).

\subsection{Applications}

Applying the Poisson integral formula \eqref{eq:main-Poisson-integral-formula}
to $V\equiv 1$, we deduce that the Hubbard-Masur function is constant.
Namely, the volume of the unit ball in $\mathcal{MF}$ with respect to the extremal  length depends only on the topological type of $\Sigma_{g,m}$ (cf. Corollary \ref{coro:Hubbard-Masur-constant}).
This is first proved by Mirzakhani and Dumas (cf. \cite[Theorem 5.10]{MR3413977}).

The Poisson integral formula is thought of  as a generalization of the mean value theorem. Namely, the value of a (pluri)harmonic function in the domain is the average of the boundary value with harmonic measures. Applying the Poisson integral formula \eqref{eq:main-Poisson-integral-formula}, we will give the vector-valued (quadratic differential-valued) measures on $\mathcal{PMF}$ which describe the $\partial$ and $\overline{\partial}$-differentials of pluriharmonic functions on $\teich_{g,m}$ which is continuous on the Bers closure (cf. \S\ref{subsubsec:presentation-differential}). Applying this description to the trace functions of boundary groups of the Bers slice, we will represent Wolpert's quadratic differentials $\Theta_{\gamma,x}$ which corresponds to the differentials of the hyperbolic lengths of closed geodesics $\gamma$ in terms of the Hubbard-Masur differentials $q_{F,x}$ ($F\in \mathcal{MF}$) by the averaging procedure as follows.
\begin{theorem}[Representation of Wolpert's differentials]
\label{thm:Wolpert-Hubbard-Masur}
For $x\in \teich_{g,m}$ and $\gamma\in \pi_1(\Sigma_{g,m})$, we have
\begin{equation}
\label{eq:Wolpert-Hubbard-Masur}
\Theta_{\gamma,x}=\frac{\convgenus }{2\sinh(\ell_\gamma(x))}\int_{\mathcal{PMF}^{mf}}{\rm tr}^2\left(\rho_{\varphi_{[F],x}}(\gamma)\right) \,\frac{q_{F,x}}{\|q_{F,x}\|}
d\PThursM^{x}([F]).
\end{equation}
\end{theorem}
The definitions of the symbols in the formula \eqref{eq:Wolpert-Hubbard-Masur} and the proof of Theorem \ref{thm:Wolpert-Hubbard-Masur} can be found in \S\ref{subsubsec:Wolpert_differentials}.
The representation \eqref{eq:Wolpert-Hubbard-Masur} gives an interaction between the $L^2$-geometry on Teichm\"uller space (Weil-Petersson Riemannian-K\"ahlerian geometry) and the $L^1$ or $L^\infty$-geometry (Teichm\"uller Finsler geometry) on Teichm\"uller space.

\subsection{History, Motivation and Future}
\label{subsec:History}
The complex analytic structure on Teichm\"uller space was described by Ahlfors with the variational formula of the period matrix (cf. \cite{MR0124486}.
See also \cite{MR0213543}).
Bers \cite{MR0130972} realized Teichm\"uller space as a bounded domain,
called the \emph{Bers slice}, in a finite dimensional complex Banach space.
Teichm\"uller space has rich  and interesting properties in the complex analytical aspect.
For instance, Teichm\"uller space is Stein (Bers-Ehrenpreis \cite{MR0168800});
the holomorphic automorphism group is (essentially) isomorphic to the mapping class group (Royden \cite{MR0288254});
the moduli space is K\"ahler hyperbolic (McMullen \cite{MR1745010});
the Kobayashi distance coincides with the Kobayashi distance (Royden \cite{MR0288254} and Earle-Kra \cite{MR0430319}) but it does not coincide with the Carath\'eodory distance (Markovic \cite{MR3761105}).

\subsubsection*{A naive problem behind our research}
Any holomorphic invariant of (marked) Riemann surfaces or Kleinian groups is thought of as a holomorphic function on Teichm\"uller space, and the algebra of holomorphic functions characterizes Teichm\"uller space as a complex manifold up to complex conjugation (cf. \cite{MR0185143}). A fundamental problem behind this research is:
\begin{problem}
\label{problem:main}
What are reasonable geometric objects which represent holomorphic functions on Teichm\"uller spaces?
\end{problem}
The Teichm\"uller space is known to be the universal space of holomorphic families in the sense that holomorphic mappings \emph{into} Teichm\"uller space admit geometric interpretations,
holomorphic families of Riemann surfaces
(cf. \cite{MR0018762}. See also \cite{MR3288665} for a commentary and an English translation of \cite{MR0018762}).

\subsubsection*{Why Bers slices?}
%
There are many realizations of the Teichm\"uller space as domains in the Euclidean space which are base points free (e.g. \cite{MR627684}, \cite{MR2964074}, \cite{MR346149}). One may ask why we think the Bers slice. 

The Bers slice is mysterious: The Bers slice is defined by a transcendental manner, and is deeply related to the theory of univalent functions (cf.\S\ref{subsec:Bers-slice}). The Bers slices depend the base point (cf. \cite{MR1037141}). The Bers boundary is conjectured to be fractal and self-similar at the fixed point with respect to the pseudo-Anosov mapping class action (See Figure \ref{fig:bers-slice-sqr-00.25}. See \cite{MR2667553}, \cite{MR2229385} and \cite[Problem 7 in Chapter 10]{MR1401347}). To approach these conjectures, it seems necessary to understand a detailed relation between the holomorphic (geometric) structure and the topological aspect of the Teichm\"uller theory around the boundary. Despite such interesting problems are posed, to the author's knowledge, there is less mathematical tools for investigating the holomorphic structure around the Bers boundary, and it is expected to develop the complex analytical aspect of the Teichm\"uller theory to clarify the relation. Actually, Problem \ref{problem:main} is motivated from these conjectures. 
\begin{figure}[t]
\includegraphics[height = 5cm]{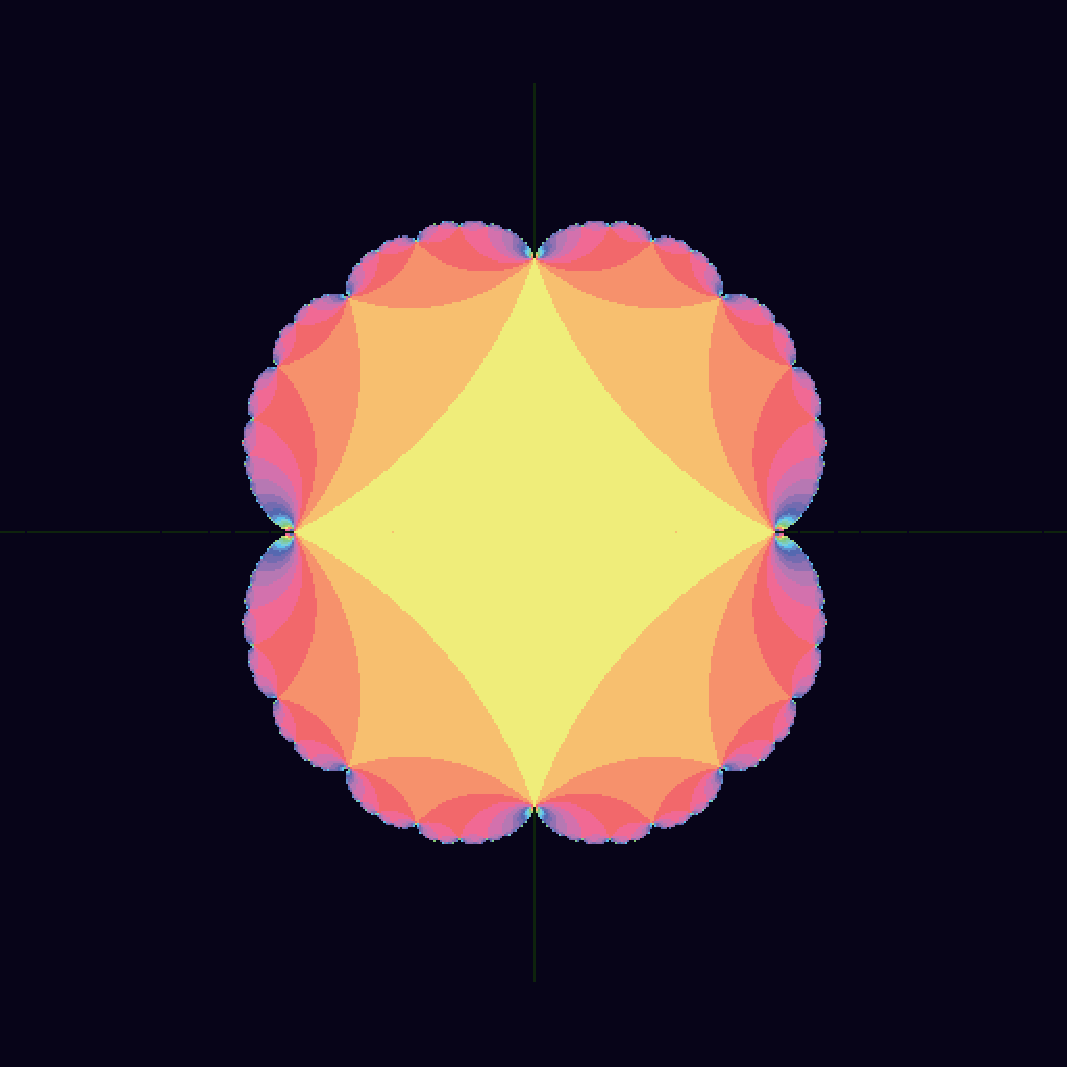}
\caption{The Bers slice of once punctured tori whose base point is the square torus. Courtesy of Professor Yasushi Yamashita (cf. \cite{MR2229385}).}
\label{fig:bers-slice-sqr-00.25}
\end{figure}
%
%
%
%
%

The Bers slice is itself an interesting bounded domain in view of Complex analysis: It is hyperconvex  (Krushkal \cite{MR1119946}), and its closure is polynomially convex (Shiga \cite{MR766636}, Deroin-Dujardan \cite{MR3707286}). Indeed, due to the polynomially convexity, \emph{almost all} of holomorphic functions on $\teich_{g,m}$ are represented by the Poisson integral formula \eqref{eq:main-Poisson-integral-formula} in the sense that any holomorphic function on $\teich_{g,m}$ is approximated by holomorphic functions on the ambient space of the Bers slice.

%
%

\subsubsection*{Complex analysis encounters the Thurston theory}
In a celebrated paper  \cite{MR2925381},
Brock, Canary and Minsky settled the ending lamination theorem.
The ending lamination theorem enables us to parametrize
the Bers boundary by topological invariants called the \emph{end-invariants},
and makes a strong connection 
between the complex analytical aspect in Teichm\"uller theory
and the Thurston theory (the topological aspect in Teichm\"uller theory)
(cf. \S\ref{sec:Kleinian-surface-groups-and-the-Bers-slice}).
Our research is based on sophisticated results in the theory of
Kleinian groups as well as Teichm\"uller theory.


To approach Problem \ref{problem:main}, we attempt to realize holomorphic functions as functions on spaces coming out from the topological aspect. The Bers boundary and the space of projective measured foliations are thought of as being essentially assembled from topological invariants by the ending lamination theorem. 
%

The extremal length functions, which appear in the Poisson kernel \eqref{eq:Poisson-kernel}, are thought of as the \emph{intersection number} between marked Riemann surfaces and measured foliations in Extremal length geometry (cf. \cite[Lemma 5.1]{ MR1231968}, \cite{MR1099913}, \cite{MR3278905} and \S\ref{subsec:ELG}). 
Thus,
the Poisson integral formula \eqref{eq:main-Poisson-integral-formula} and the homogeneous Cauchy-Riemann equation
\eqref{eq:CR}
are expected
to strengthen the connection 
between
the complex-analytical aspect
and
the topological aspect in Teichm\"uller theory,
and to develop
\emph{Complex analysis on Teichm\"uller space with Thurston theory}.


\subsection{About this paper}
This paper is organized as follows.
In \S\ref{sec:Toy_model}, we discuss the case of tori for a model case of our main theorem.  
From \S\ref{sec:Teichmuller-theory} to \S\ref{sec:Thurston-measure},
we recall basics and known results in Teichm\"uller theory.
In \S\ref{sec:holomorphic-coordinate-ext} and \S\ref{sec:measures-on-horosphere-monge-ampere-measure}, we recall and discuss the holomorphic coordinates associated to the extremal length functions of essentially complete measured foliations developed in \cite{Towards-complex-miyachi2017}, and the presentation of the Levi forms of the extremal length functions. 

The proof of Theorem \ref{thm:Poisson-integral-formula} is accomplished in the discussion from \S\ref{sec:pluriharmonic-measure-Thurston-measure} to \S\ref{sec:intrinsic-representation}.
In the proof, we will compare the Thurston measure on the unit sphere in $\mathcal{MF}$ in terms of extremal length functions with the measures defined on the pluricomplex Green function on the level set (cf. Proposition \ref{prop:local-comparison}). For the comparison, we adopt the reciprocals of extremal length functions as mediators (cf. \eqref{eq:comparison-green-dT}). We will prove Theorem \ref{thm:boundary-behavior} in \S\ref{sec:boundary-behaviorPI}.

\S\ref{sec:averaging-pmf}, we rephrase the integral representation \eqref{eq:main-Poisson-integral-formula} in terms of the integration on $\mathcal{PMF}$ and discuss the integral representation of the $\partial$ and $\overline{\partial}$-differentials of pluriharmonic functions on $\teich_{g,m}$ (cf. Corollary \ref{coro:integral-representation-PMF} and \eqref{eq:Diff-holo-functions}). The holomorphic quadratic differentials associated to the differentials of hyperbolic lengths of closed geodesics are represented by averaging the Hubbard-Masur differentials by the Thurston measure (cf. \eqref{eq:Wolpert} and Theorem \ref{thm:Wolpert-Hubbard-Masur}).

\section*{Acknowledgement}
The author would like to express his gratitude to anonymous referees for valuable comments. He also thanks Professor Ken'ichi Ohshika and Professor Athanase Papadopoulous for their constant encouragements.

%

\section{One dimensional cases}
\label{sec:Toy_model}
\subsection{Case of tori}
\label{subsec:Case_tori}
We check the correspondence in Table \ref{table:dictionary} for tori. We start with recalling the horofunction compactification of the Teichm\"uller space of tori. The horofunctions are presented with the extremal length. We will see the same conclusion holds for arbitrary $(g,m)$ in \S\ref{sec:ELG}.

Let $\Sigma_1$ be a (topolotical) torus and $A$ and $B$ are generators of the fundmanetal group $\pi_1(\Sigma_1)\cong H_1(\Sigma_1)$ such that the algebraic intersection number $A\cdot B$ is $+1$. As discussed above, the Teichm\"uller space $\teich_1$ is identified with $\mathbb{H}$ in the sense that any marked torus is biholomorphically equivalent to $X_\tau=\mathbb{C}/\Gamma_\tau$ where $\Gamma_\tau$ is the lattice on $\mathbb{C}$ generated by $1$ and $\tau\in \mathbb{H}$ and the marking $\Sigma_1\to X_\tau$ sends $A$ to $1$ and $B$ to $\tau$. We denote by $x_\tau\in \teich_1$ the marked tori associated to $\tau\in \mathbb{H}$.

The free homotopy classes of simple closed curves on a (topological) torus $\Sigma_1$ is enumerated by $\hat{\mathbb{Q}}=\mathbb{Q}\cup \{\infty\}$ ($\infty=\pm 1/0$). In our convention, the $p/q$-curve on $\Sigma_1$ corresponds to $-pA+qB$, where $p$ and $q$ are taken to be relatively prime when $pq\ne 0$, $p=1$ when $q=0$, and $q=1$ when $p=0$. The space of measured foliation $\mathcal{MF}(\Sigma_1)$ on $\Sigma_1$ is canonically identified with the quotient space $\mathbb{R}^2/(x,y)\sim (-x,-y)$, so that the $p/q$-curve corresponds to the equivalence class $[-p,q]\in  \mathbb{R}^2/(x,y)\sim (-x,-y)$ of $(-p,q)\in \mathbb{R}^2$. The space $\mathcal{PMF}(\Sigma_1)=(\mathcal{MF}(\Sigma_1)-\{0\})/\mathbb{R}_{>0}$ of projective measured foliations is identified with $\partial\mathbb{H}=\hat{\mathbb{R}}=\mathbb{R}\cup\{\infty\}$, and the identification is induced by the map
\begin{equation}
\label{eq:identification_PMF}
\Pi_1\colon \mathcal{MF}(\Sigma_1)=\mathbb{R}^2/(x,y)\sim (-x,-y)\ni [a,b]\mapsto -a/b\in \hat{\mathbb{R}}.
\end{equation}

The \emph{extremal length} $\ext_\tau(F_[a,b])$ of the measured foliation $F_{[a,b]}$ is equal to
$$
\ext_\tau([a,b])=\dfrac{|a+b\tau|^2}{{\rm Im}(\tau)}.
$$
The Teichm\"uller distance on $\teich_1$ coincides with the hyperbolic distance $d_\mathbb{H}$ of curvature $-4$, and \emph{Kerckhoff's formula} holds:
$$
d_T(x_{\tau_1},x_{\tau_2})=\dfrac{1}{2}\log\sup_{p/q\in \hat{\mathcal{Q}}}\dfrac{\ext_{x_{\tau_1}}(C_{p/q})}{\ext_{x_{\tau_2}}(C_{p/q})}
=
\dfrac{1}{2}\log\dfrac{|\tau_1-\overline{\tau_2}|+|\tau_1-\tau_2|}{|\tau_1-\overline{\tau_2}|-|\tau_1-\tau_2|}.
$$
The \emph{horofunction} appears in the horofunction compactification (introduced by Gromov \cite{MR624814}) which is defined by the closure of the embedding (defined with basepoint $\tau=\tau_0$)
$$
\teich_1\ni x_{\tau_1}\mapsto [\teich_1\ni \tau \mapsto d_T(x_\tau,x_{\tau_1})-d_T(x_{\tau_1},x_{\tau_0})]\in C^0(\teich_1)
$$
into the space of continuous functions on $\teich_1$ (endowed with the topology of uniform convergence on any compact sets). In the case of $(\teich_1,d_T)\cong (\mathbb{H},d_\mathbb{H})$, the horofunction compactification canonically coincides with the closure $\overline{\mathbb{H}}=\mathbb{H}\cup \hat{\mathbb{R}}$. The horofunction associated to $\Pi_1([a,b])=-a/b\in \hat{\mathbb{R}}$ is 
\begin{align*}
\mathfrak{h}(x_\tau,x_{\tau_0},-a/b)
&=\lim_{\tau_1\to -a/b}(d_T(\tau ,x_{\tau_1})-d_T(x_{\tau_1},x_{\tau_0})) \\
&=
\frac{1}{2}\log\dfrac{{\rm Im}(\tau_0)}{{\rm Im}(\tau)}\dfrac{|a+b\tau|^2}{|a+b\tau_0|^2}
\\
&=
\frac{1}{2}(\log\ext_{x_\tau}(F_{[a,b]})-\log\ext_{x_{\tau_0}}(F_{[a,b]})).
\end{align*}

The Thurston measure $\ThursM$ on $\mathcal{MF}(\Sigma_1)$ is induced by the Euclidean measure on $\mathbb{R}^2$ up to multiplying positive constants. The measure $\ThursM$ is a mapping class group-invariant and ergodic measure supported on the filling measured foliations (e.g. \cite{MR1503464} and \cite{MR1953296}. See also \cite{MR2424174}). In the case of tori, filling measured foliations correspond to points in $\mathcal{MF}(\Sigma_1)=\mathbb{R}^2/(x,y)\sim (-x,-y)$ with irrational slopes. Let $F_{[a,b]}$ be the measured foliation labeled $[a,b]\in \mathcal{MF}(\Sigma_1)$.
The unit sphere in terms of the extremal length function $\mathcal{SMF}_{x_{\tau}}=\{F\in \mathcal{MF}\mid \ext_{x_{\tau}}(F)=1\}$
is (the quotient of) an ellipse $\{[a,b]\mid |a+b\tau|^2={\rm Im}(\tau)\}$.
The (normalized) Thurston measure $\PThursM^\tau$ associated to $\tau\in \mathbb{H}\cong \teich_1$ is a Borel measure on the ellipse $\mathcal{SMF}_{x_{\tau}}$, which is defined by the cone extension (cf. \eqref{eq:Thurston-measure-SMF}).
The pushforward measure ${\boldsymbol \mu}_\tau=(\Pi_1)_*(\PThursM^\tau)$ on $\hat{\mathbb{R}}=\partial \mathbb{H}=\partial \teich_1$ is 
\begin{equation}
\label{eq:harmonic_measure_H}
{\boldsymbol \mu}_\tau(A)=\int_A\dfrac{{\rm Im}(\tau)}{|\tau-\xi|^2}\dfrac{d\xi}{\pi}
\end{equation}
for $A\subset \hat{\mathbb{R}}$, which is nothing but the harmonic measure on $\mathbb{H}$ at $\tau$ (cf. \cite[\S I.1]{MR2150803}). The unit sphere $\mathcal{SMF}_{x_{\tau}}$ is thought of as the infinitesimal circle in the tangent space and the harmonic measure (the pushforward measure) is the visual angle measure (e.g. \cite[\S1]{MR2150803}).
Thus, the classical Poisson integral formula \eqref{eq:Poisson_integral_formula_up} is written as
$$
u(\tau)=\int_{\partial \teich_1}u(\xi)e^{-2\mathfrak{h}(x_\tau,x_{\tau_0},\xi)}d{\boldsymbol \mu}_{\tau_0}(\xi)
$$
with fixed $\tau_0\in \teich_1$, as discussed around \eqref{eq:main-Poisson-integral-formula-Buseman}.

\subsection{Cases of once punctured tori and fourth punctured spheres}
\label{subsec:Case_once_tori}
The Teichm\"uller space $\teich_{1,1}$ of once punctured tori is canonically identified with the Teichm\"uller space $\teich_1$ of tori and hence with the upper-half plane $\mathbb{H}$. Hence we write $x_\tau$ a point in $\teich_{1,1}$ corresponding to $\tau\in \mathbb{H}$ as \S\ref{subsec:Case_tori}. From the commensurability, the Bers embedded the Teichm\"uller space $\teich_{0,4}$ of fourth punctured spheres canonically coincides with that of $\teich_{1,1}$ (cf. \cite[Lemma 3.1]{MR2060380}). 

Fix $x_0\in \teich_{1,1}$. Let $\Bers{x_0}$ be the Bers slice of once punctured tori with base point $x_0$ (cf. \S\ref{subsec:Bers-slice}. See also Figure \ref{fig:bers-slice-sqr-00.25}). The identification $\mathbb{H}\cong \Bers{x_0}$ is nothing but the Riemann mapping. Since the Bers slice is a Jordan domain, the identification extends the closures $\overline{\mathbb{H}}\to \overline{\Bers{x_0}}$ (cf. \cite{MR1689341}). The map $\partial \mathbb{H}\to \partial \Bers{x_0}$ induced from the Riemann mapping coinsides with the mapping defined from the ending lamination theorem after identifying $\partial \mathbb{H}\cong \mathcal{PMF}$ by \eqref{eq:identification_PMF}.

The harmonic measure $\pushThursMBers_{x_\tau}$ at $x_\tau\in \teich_{1,1}\cong \Bers{x_0}$ on the Bers boundary $\partial \Bers{x_0}$ is the pushforward measure of \eqref{eq:harmonic_measure_H}, and hence, the harmonic measure $\pushThursMBers_{x_\tau}$ is the pushforward measure of the normalized Thurston measure $\PThursM^\tau$ on the unit sphere $\mathcal{SMF}_{x_\tau}$ of the extremal length function via $\mathcal{SMF}_{x_\tau}\to \partial\Bers{x_0}\cong \partial \mathbb{H}$ as discussed in \S\ref{subsec:Case_tori}. Therefore, our main theorem, Theorem \ref{eq:main-Poisson-integral-formula}, in this case follows from these observations (cf. \cite{MR2150803}).



\section{Teichm\"uller theory}
\label{sec:Teichmuller-theory}
Let $\Sigma_{g,m}$ be a closed orientable surface 
of genus $g$ with $m$-marked points with $2g-2+m>0$ (possibly $m=0$).
We define the complexity of $\Sigma_{g,m}$
by $\convgenus=\convgenus(\Sigma_{g,m})=3g-3+m$.
In this section,
we recall basics in Teichm\"uller theory.
For reference,
see \cite{MR590044},
\cite{MR568308},
\cite{MR903027} ,
\cite{MR2245223},
\cite{MR1215481},
and \cite{MR927291}
for instance.

\subsection{Teichm\"uller space}
\emph{Teichm\"uller space} $\teich_{g,m}$ of type $(g,m)$
is the equvalence classes
of marked Riemann surfaces of type $(g,m)$. 
A \emph{marked Riemann surface} $(M,f)$ of type $(g,m)$
is a pair of a Riemann surface $M$ of analytically finite type $(g,m)$
and an orientation preserving homeomorphism
$f\colon \Sigma_{g,m}\to M$.
Two marked Riemann surfaces $(M_{1},f_{1})$ and $(M_{2},f_{2})$
of type $(g,m)$ are
\emph{(Teichm\"uller) equivalent} if there is a conformal mapping
$h\colon M_{1}\to M_{2}$ such that $h\circ f_{1}$ is homotopic to $f_{2}$.

The \emph{Teichm\"uller distance} $d_T$ is a complete distance
on $\teich_{g,m}$ defined by
$$
d_T(x_1,x_2)=\frac{1}{2}\log \inf_hK(h)
$$
for $x_i=(M_i,f_i)$ ($i=1,2$),
where the infimum runs over all quasiconformal mapping
$h\colon M_1\to M_2$ homotopic to $f_2\circ f_1^{-1}$
and $K(h)$ is the maximal dilatation of a quasiconformal mapping $h$.

The \emph{mapping class group} $\mcg_{g,m}$ is 
the group of homotopy classes of orientation preserving homeomorphisms 
on $\Sigma_{g,m}$.
Any element $[\omega]\in \mcg_{g,m}$ acts on $\teich_{g,m}$
by $[\omega](M,f)=(M,f\circ \omega^{-1})$.

\subsection{Quadratic differentials}
For $x=(M,f)\in \teich_{g,m}$,
we denote by $\mathcal{Q}_{x}$ the complex Banach space
of holomorphic quadratic differentials $q=q(z)dz^{2}$
on $M$ with
$$
\|q\|=\int_{M}|q(z)|\frac{\sqrt{-1}}{2}dz\wedge d\overline{z}
<\infty.
$$
The space $\mathcal{Q}_{x}$ is isomorphic to $\mathbb{C}^{\convgenus}$.
The union $\mathcal{Q}_{g,m}=\cup_{x\in \teich_{g,m}}\mathcal{Q}_{x}$
is recognized as the holomorphic cotangent bundle of $\teich_{g,m}$ via the pairing \eqref{eq:pairing-T-coT} given later.
A differential $q\in \mathcal{Q}_{g,m}$ is said to be \emph{generic}
if all zeros are simple and all marked points of the underlying surface
are simple poles of $q$.
Generic differentials are open and dense subset in $\mathcal{Q}_{g,m}$
and in each fiber $\mathcal{Q}_{x}$ for $x\in \teich_{g,m}$.

\subsection{Infinitesimal complex structure on $\teich_{g,m}$}
\label{subsec:infinitesimal-Teich}
Teichm\"uller space $\teich_{g,m}$ is a complex manifold of dimension
$\convgenus$.
The infinitesimal complex structure is described as follows.

Let $x=(M,f)\in \teich_{g,m}$.
Let $L^{\infty}(M)$ be the Banach space of measurable $(-1,1)$-forms
$\mu=\mu(z)d\overline{z}/dz$ on $M$
with the essential supremum norm
$$
\|\mu\|_{\infty}={\rm ess.sup}_{p\in M}|\mu(p)|<\infty.
$$
Then,
the holomorphic tangent space $T_{x}\teich_{g,m}$ of $\teich_{g,m}$ at $x$
is described as the quotient space
$$
L^{\infty}(M)/\{\mu\in L^{\infty}(M)\mid
\langle \mu,\varphi\rangle=0,\forall \varphi\in \mathcal{Q}_{x}\},
$$
where
\begin{equation*}
\langle \mu,\varphi\rangle=
\int_{M}\mu(z)\varphi(z)\frac{\sqrt{-1}}{2}dz\wedge d\overline{z}.
\end{equation*}
Any element of $L^\infty(M)$ is called an \emph{infinitesimal Beltrami differential} in this context.
For $v=[\mu]\in T_x\teich_{g,m}$
and $\varphi\in \mathcal{Q}_x$,
a \emph{canonical pairing} between $T_x\teich_{g,m}$
and $\mathcal{Q}_x\cong T_x^*\teich_{g,m}$ is defined by
\begin{equation}
\label{eq:pairing-T-coT}
\langle v,\varphi\rangle=\langle \mu,\varphi\rangle.
\end{equation}

Let $q_0\in \mathcal{Q}_x$ be a generic differential and $v\in T_x\teich_{g,m}$.
The \emph{$q_0$-realization of $v$}
is a quadratic differential
$\eta_v\in \mathcal{Q}_x$ which satisfy
\begin{equation}
\label{eq:q0-realization}
\langle v,\varphi\rangle=\int_M\frac{\overline{\eta_v}}{|q_0|}\varphi
\end{equation}
for all $\varphi\in \mathcal{Q}_x$.
The correspondence
\begin{equation}
\label{eq:qo-realization-1}
T_x\teich_{g,m}\ni v\mapsto \eta_v\in \mathcal{Q}_x
\end{equation}
is an anti-complex linear isomorphism
(cf. \cite[Theorem 5.3]{MR3413977} and \cite[\S4.2]{MR3715450}).

\subsection{Measured foliations and laminations}
Let $\mathcal{S}$ be the set of homotopy classes of 
essential simple closed curves on $\Sigma_{g,m}$.
By a \emph{multi-curve} we mean an unordered finite sequences 
$(\alpha_i)_i$ in $\mathcal{S}$
such that $\alpha_i\ne \alpha_j$ and $i(\alpha_i,\alpha_j)=0$ for all $i\ne j$.
Let $i(\alpha,\beta)$ denote the \emph{geometric intersection number}
for simple closed curves $\alpha,\beta\in \mathcal{S}$.
Let $\mathcal{WS}=\{t\alpha\mid t\ge 0, \alpha\in \mathcal{S}\}$
be the set of weighted simple closed curves.
The intersection number on $\mathcal{WS}$
is defined by
\begin{equation}
\label{eq:intersection-number-WS}
i(t\alpha,s\beta)=ts\,i(\alpha,\beta)\quad
(t\alpha,s\beta\in \mathcal{WS}).
\end{equation}

\subsubsection{Measured foliations}
We consider an embedding
\begin{equation*}
\mathcal{WS}\ni t\alpha\mapsto [\mathcal{S}\ni \beta
\mapsto i(t\alpha,\beta)]\in \mathbb{R}_{\ge 0}^{\mathcal{S}}.
\end{equation*}
We topologize the function space
$\mathbb{R}_{\ge 0}^{\mathcal{S}}$
with the topology of pointwise convergence.
The closure $\mathcal{MF}$ of the image of the embedding is called
the \emph{space of measured foliations} on $\Sigma_{g,m}$.
Let
$$
{\rm proj}\colon \mathbb{R}_{\ge 0}^{\mathcal{S}}-\{0\}\to \mathbb{PR}_{\ge 0}^{\mathcal{S}}=(\mathbb{R}_{\ge 0}^{\mathcal{S}}-\{0\})/\mathbb{R}_{>0}
$$
be the projection.
The image $\mathcal{PMF}={\rm proj}(\mathcal{MF}-\{0\})$
is called the space of \emph{projective measured foliations}  on $\Sigma_{g,m}$.
We write $[F]$ the projective class of $F\in \mathcal{MF}-\{0\}$.
$\mathcal{MF}$ and $\mathcal{PMF}$ are homeomorphic 
to $\mathbb{R}^{2\convgenus}$ and $\mathbb{S}^{2\convgenus-1}$,
respectively.

By definition,
$\mathcal{MF}$
contains $\mathcal{WS}$
as a dense subset.
The intersection number extends continuously 
as a non-negative function $i(\,\cdot\,,\,\cdot\,)$
on $\mathcal{MF}\times \mathcal{MF}$
satisfying $i(F,F)=0$ and $F(\alpha)=i(F,\alpha)$
for $F\in \mathcal{MF}\subset \mathbb{R}_{\ge 0}^{\mathcal{S}}$
and $\alpha\in \mathcal{S}$.
The mapping class group $\mcg_{g,m}$
acts on $\mathcal{MF}$ by
$$
i([\omega](F),\alpha)=i(F,\omega^{-1}(\alpha))\quad
(F\in \mathcal{MF},
\alpha\in \mathcal{S})
$$
and $[\omega](tF)=t[\omega](F)$ for $t\ge 0$ and $F\in \mathcal{MF}$.
We say that two measured foliations $F$ and $G$ are \emph{transverse}
if no nonzero measured foliation $H$ satisfies
$i(H,F)=i(H,G)=0$
(cf. \cite{MR1099913}).

\subsubsection{Measured laminations}
Fix a hyperbolic structure of finite area on $\Sigma_{g,m}$.
A \emph{geodesic lamination} $L$ on $\Sigma_{g,m}$ is a non-empty closed set
which is a disjoint union of complete simple geodesics,
where a geodesic is said to be \emph{complete}
if it is either closed or has infinite length in both of its ends.
The geodesics in $L$ are called the \emph{leaves} of $L$.
A \emph{transverse measure} for a geodesic lamination $L$
means an assignment a Borel measure to each arc transverse to $L$,
subject to the following two conditions:
If the arc $k'$ is contained in the transverse arc $k$,
the measure assigned to $k'$ is the restriction of the measure assigned to $k$;
and
if the two arcs $k$ and $k'$ are homotopic through a family of arcs transverse to $L$,
the homotopy sends the measure assigned to $k$ to the measure assigned to $k'$.
A transverse measure to a geodesic lamination $L$
is said to have \emph{full support} if
the support of the measure assigned to each transverse arc $k$ is exactly $k\cap L$.
A \emph{measured lamination} $L$ is a pair  consisting of a geodesic lamination
called the \emph{support} of $L$, 
and full support transverse measures to the support. 
Let $\mathcal{ML}$ be the set of measured laminations on $\Sigma_{g,m}$
(with fixing a complete hyperbolic structure).
A weighted simple closed curve $t\alpha$ is identified with
a measured lamination consisting of a simple closed geodesic homotopic to $\alpha$
and an assignment $t$-times the Dirac measures
whose support consists of the intersection to transverse arcs.
The intersection number \eqref{eq:intersection-number-WS}
on weighted simple closed curves
extends continuously to $\mathcal{ML}\times \mathcal{ML}$.

It is known that there is a canonical identification
$\mathcal{MF}\cong \mathcal{ML}$ such that
$F\in \mathcal{MF}$ corresponds to $L$ if and only if
\begin{equation*}
i(F,\alpha)=i(L,\alpha)\quad (\alpha\in \mathcal{S})
\end{equation*}
(e.g. \cite{MR1810534}, \cite{MR1144770} and \cite{Thuston-LectureNote}).

\begin{convention}
Henceforth,
we will frequency use the canonical correspondence between measured laminations
and measured foliations.
\end{convention}

For $F\in \mathcal{MF}$,
we denote by $L(F)$ the support of the corresponding measured lamination.
For simplicity,
we call $L(F)$ the \emph{support lamination} of $F$.
For a geodesic lamination $L$,
we define
$$
\mathcal{MF}_L=\{F\in \mathcal{MF}\mid L(F)\subset L\}.
$$
It is known that
$\mathcal{MF}_L$ is a non-empty convex closed cone in $\mathcal{MF}$.

An $F\in \mathcal{MF}$ is called \emph{minimal}
if any leaf of $L(F)$ is dense in $L(F)$
(with respect to the induced topology from $\Sigma_{g,m}$).
An $F\in \mathcal{MF}$ is called \emph{filling}
if any complementary region of $L(F)$ is either an ideal polygon or 
a once punctured ideal polygon, which is equivalent to say that $i(F,\alpha)>0$ for all $\alpha\in \mathcal{S}$.
In this paper,
a measured lamination $L$ is said to be \emph{uniquely ergodic} if it is minimal and filling
and if $L'\in \mathcal{ML}$ satisfies $i(L,L')=0$,
then $L'=tL$ for some $t\ge 0$.
A measured foliation is said to be \emph{uniquely ergodic} if
so is the corresponding measured lamination.

A measured foliation $F$ is said to be \emph{essentially complete} if 
each component of the complement of $L(F)$ is either an ideal triangle
or a once punctured ideal monogon
if $(g,m)\ne (1,1)$
and a once punctured bigon otherwise
(cf. \cite[Definition 9.5.1, Propositions 9.5.2 and 9.5.4]{Thuston-LectureNote}).
Essentially complete measured foliations are generic in $\mathcal{MF}$.

\section{Extremal length geometry}
\label{sec:ELG}
\subsection{Hubbard-Masur theorem}
\label{subsec:Hubbard-masur-theorem}
For $x=(M,f)\in \teich_{g,m}$ and $q\in \mathcal{Q}_{x}$,
we define the \emph{vertical foliation} $v(q)\in \mathcal{MF}$
of $q=q(z)dz^2$
by
$$
i(v(q),\alpha)=\inf_{\alpha'\in f(\alpha)}\int_{\alpha'}|{\rm Re}(\sqrt{q(z)}dz)|
\quad (\alpha\in \mathcal{S}).
$$
We call $h(q)=v(-q)$ the \emph{horizontal foliation} of $q$.
Hubbard and Masur \cite{MR523212}
observed that the mapping
\begin{equation}
\label{eq:Hubbard-Masur-homeo}
\mathcal{Q}_{x}\ni q\mapsto v(q)\in \mathcal{MF}
\end{equation}
is homeomorphic for all $x\in \teich_{g,m}$.
From \eqref{eq:Hubbard-Masur-homeo},
for any $x\in \teich_{g,m}$ and $F\in \mathcal{MF}$,
there is a unique $q_{F,x}\in \mathcal{Q}_{x}$
with $v(q_{F,x})=F$.
We call $q_{F,x}$ the \emph{Hubbard-Masur differential}
for $F$ on $x$.
When $F$ is essentially complete,
$q_{F,x}$ is generic for all $x\in \teich_{g,m}$.

The horizontal and vertical foliations of $q\in \mathcal{Q}_{g,m}$ are transverse.
Namely,
\begin{equation}
\label{eq:fillingup}
i(h(q),H)+i(v(q),H)>0
\end{equation}
holds for all $H\in \mathcal{MF}-\{0\}$.
Futhermore,
we have an embedding
\begin{equation*}
\mathcal{Q}_{g,m}\ni q\to (h(q),v(q))\in \mathcal{MF}\times \mathcal{MF}.
\end{equation*}
The image is characterized by \eqref{eq:fillingup}
(cf. \cite{MR1099913}).

\subsection{Extremal length}
\label{subsec:extremal-length}
The \emph{extremal length} of $F\in \mathcal{MF}$ on $x=(M,f)\in \teich_{g,m}$ is defined by
$$
\ext_{x}(F)=\|q_{F,x}\|.
$$
The extremal length is a conformal quasi-invariant
in the sense that
\begin{equation}
\label{eq:comparison-K}
e^{-2d_T(x,y)}\ext_y(F)
\le \ext_x(F)
\le e^{2d_T(x,y)}\ext_y(F)
\end{equation}
for $x,y\in \teich_{g,m}$ and $F\in \mathcal{MF}$.
The extremal length function is continuous on $\teich_{g,m}\times \mathcal{MF}$.
Furthermore,
\eqref{eq:comparison-K}
is known to be sharp 
by \emph{Kerckhoff's formula}
\begin{equation*}
d_T(x,y)=\frac{1}{2}\log\sup_{F\in \mathcal{MF}-\{0\}}
\frac{\ext_x(F)}{\ext_{y}(F)}
\end{equation*}
(cf. \cite{MR559474}).
The extremal length of $\alpha\in \mathcal{S}$ is characterized by
\begin{equation}
\label{eq:analytic-definition-extremal-length}
\ext_x(\alpha)=\sup_{\rho}
\left\{
\left(\inf_{\alpha'\in f(\alpha)}\int_{\alpha'}\rho(z)|dz|\right)^2/\int_{M}\rho(z)^2dxdy
\right\},
\end{equation}
where the supremum runs over all conformal metric $\rho=\rho(z)|dz|$ on $M$.
Substituting the hyperbolic metric to $\rho$ in \eqref{eq:analytic-definition-extremal-length},
we have a comparison
\begin{equation}
\label{eq:comparison-hyp-ext}
{\rm leng}_x(\alpha)\le \sqrt{2\pi(2g-2+m)}\ext_x(\alpha)^{1/2},
\end{equation}
where ${\rm leng}_x(\alpha)$ is the hyperbolic length of the geodesic representative of $f(\alpha)$
on $M$.
After setting ${\rm leng}_x(t\alpha)=t\,{\rm leng}_x(\alpha)$
for $t\alpha\in \mathcal{WS}$,
we see that the comparison \eqref{eq:comparison-hyp-ext}
also holds for measured foliations (laminations).

Minsky observed the following
inequality,
called \emph{Minsky's inequality}
\begin{equation}
\label{eq:minsky-inequality}
i(F,G)^2\le \ext_x(F)\ext_x(G)
\end{equation}
for $F,G\in \mathcal{MF}$
and $x\in \teich_{g,m}$
(cf. \cite[Lemma 5.1]{MR1231968}).

\subsection{Extremal length geometry}
\label{subsec:ELG}
The closure $\cl{\teich_{g,m}}$ of the embedding
\begin{equation}
\label{eq:GMdef}
\teich_{g,m}\ni x\mapsto 
{\rm proj}([\mathcal{S}\ni \beta
\mapsto \ext_{x}(\alpha)^{1/2}])\in \mathbb{PR}_{\ge 0}^{\mathcal{S}}
\end{equation}
is called the \emph{Gardiner-Masur compactification} of $\teich_{g,m}$.
We identify $\teich_{g,m}$ with the image of \eqref{eq:GMdef}.
The \emph{Gardiner-Masur boundary} $\partial_{GM}\teich_{g,m}$ is, by definition,
the complement of $\teich_{g,m}$ from the Gardiner-Masur compactification. The Gardiner-Masur compactification coincides with the horofunction compactification (cf. \cite{MR3289706}).
The Gardiner-Masur boundary contains $\mathcal{PMF}$ (cf. \cite{MR1099913}).

Let $\mathcal{C}_{GM}={\rm proj}^{-1}(\cl{\teich_{g,m}})$.
Since $\mathcal{PMF}\subset  \partial_{GM}\teich_{g,m}$,
$\mathcal{C}_{GM}$ contains $\mathcal{MF}$.
The intersection number $i(\,\cdot,\cdot\,)$ and
the extremal length function $\ext_{x}$ ($x\in \teich_{g,m}$)
on $\mathcal{MF}$
extend continuously to $\mathcal{C}_{GM}$
(cf. \cite[Theorems 1 and 3]{MR3278905}).

For any $x_{0}\in \teich_{g,m}$,
there is a continuous function
$i_{x_{0}}$
on $\cl{\teich_{g,m}}\times \cl{\teich_{g,m}}$
such that
\begin{align}
i_{x_{0}}(x,y)
&=\exp(-2\gromov{x}{y}{x_{0}}),
\label{eq:ex-geo1}
\\
i_{x_{0}}(x,[\mathfrak{p}]) &=
\exp(-d_{T}(x_{0},x))\frac{\ext_{x}(\mathfrak{p})^{1/2}}{\ext_{x_{0}}(\mathfrak{p})^{1/2}},
\label{eq:ex-geo2}
\\
i_{x_{0}}([\mathfrak{p}],[\mathfrak{q}]) &=\frac{i(\mathfrak{p},\mathfrak{q})}{\ext_{x_0}(\mathfrak{p})^{1/2}\ext_{x_0}(\mathfrak{q})^{1/2}}
\label{eq:ex-geo3}
\end{align}
for $x,y\in \teich_{g,m}$ and $\mathfrak{p},\mathfrak{q}\in \mathcal{C}_{GM}$,
$[\mathfrak{p}]={\rm proj}(\mathfrak{p})$
and
$[\mathfrak{q}]={\rm proj}(\mathfrak{q})$,
where $\gromov{x}{y}{y_{0}}$ is the \emph{Gromov product}
$$
\gromov{x}{y}{y_{0}}=\frac{1}{2}(d_{T}(y_{0},x)+d_{T}(y_{0},y)-d_{T}(x,y))
$$
(cf. \cite{MR3278905}). Notice from \eqref{eq:ex-geo1} and \eqref{eq:ex-geo2} that the holofunction $\mathfrak{h}(x,\tau,x_0,[\mathfrak{p}])$ at $[\mathfrak{p}]\in \partial_{GM}\teich_{g,m}$ is 
$$
\mathfrak{h}(x,x_0,[\mathfrak{p}])=\dfrac{1}{2}\left(\log \ext_{x}(\mathfrak{p})-\log \ext_{x_0}(\mathfrak{p})\right).
$$

\section{Kleinian surface groups and the Bers slice}
\label{sec:Kleinian-surface-groups-and-the-Bers-slice}
\subsection{Kleinian surface groups}
\label{subsec:Kleinian-groups}
A \emph{(marked) Kleinian surface group} is, by definition,
a Kleinian group with an isomorphism from $\pi_1(\Sigma_{g,m})$
which sends peripheral curves to parabolic elements.
Let $\rho\colon \pi_1(\Sigma_{g,m})\to {\rm PSL}_2(\mathbb{C})$
be a Kleinian surface group.
Then, there is a homeomorphism from $\Sigma_{g,m}\times \mathbb{R}$ to the quotient manifold $\mathbb{H}^{3}/\rho(\pi_1(\Sigma_{g,m}))$ which induces $\rho$ (cf. \cite{MR847953} and \cite{Thuston-LectureNote}).
For $\alpha\in \mathcal{S}$,
the hyperbolic length ${\rm leng}_{\rho}(\alpha)$ of $\alpha$ on the quotient manifold
$\mathbb{H}^{3}/\rho(\pi_1(\Sigma_{g,m}))$ is the translation length of the corresponding
element in $\rho(\pi_1(\Sigma_{g,m}))$.
For a measured lamination (foliation) $L$ which is realizable
in the quotient manifold of $\rho$,
we define the hyperbolic length ${\rm leng}_\rho(L)$
as the hyperbolic length with respect to the induced hyperbolic metric
from the pleated surface realizing $L$.
By taking the lim-inf of the infima of length of measured laminations which are realizable in 
the quotient manifold of $\rho$,
the (hyperbolic) length function ${\rm leng}_\rho$ is well-defined on $\mathcal{MF}$.
The length function is known to be continuous on the product of the space of conjugacy classes
of Kleinian surface groups for $\Sigma_{g,m}$
and $\mathcal{MF}$ (cf. \cite{MR1791139} and \cite{MR1029395}).

When a Kleinian surface group $\rho$ admits a simply connected invariant domain $\Omega\subset \hat{\mathbb{C}}$,
the quotient $\Omega/\rho(\pi_1(\Sigma_{g,m}))$ is a Riemann surface
homeomorphic to $\Sigma_{g,m}$.
The representation $\rho$ determines a marking on $\Omega/\rho(\pi_1(\Sigma_{g,m}))$.
The hyperbolic length of any measured foliation $H$ on $\Omega/\rho(\pi_1(\Sigma_{g,m}))$
is at least ${\rm leng}_\rho(H)$
(cf. \cite[Theorem 3]{MR0297992} and \cite[Proposition 2.1]{MR1029395}).

\subsection{Bers slice}
\label{subsec:Bers-slice}
Fix $x_{0}=(M_{0},f_{0})\in \teich_{g,m}$
and let $\Gamma_{0}$ be the marked Fuchsian group acting on $\mathbb{H}$
uniformizing $M_{0}$ with the marking $\pi_1(\Sigma_{g,m})\cong \Gamma_0$
induced by $f_0$.
Let $A_{2}(\mathbb{H}^{*},\Gamma_{0})$ be the Banach space of automorphic forms
on $\mathbb{H}^{*}=\hat{\mathbb{C}}-\overline{\mathbb{H}}$
of weight $-4$ with the hyperbolic supremum norm.
For each $\varphi\in A_{2}(\mathbb{H}^{*},\Gamma_{0})$,
we can define a locally univalent meromorphic mapping $W_{\varphi}$
on $\mathbb{H}^{*}$ and the monodromy homomorphism 
$\rho_{\varphi}\colon \Gamma_{0}\to {\rm PSL}_{2}(\mathbb{C})$
such that 
the Schwarzian derivative of
$W_{\varphi}$ 
is equal to $\varphi$
and $\rho_{\varphi}(\gamma)\circ W_{\varphi}=W_{\varphi}\circ \gamma$
for all $\gamma\in \Gamma_{0}$.
Let $\Gamma_{\varphi}=\rho_{\varphi}(\Gamma_{0})$.

The \emph{Bers slice} $\Bers{x_{0}}$ with base point
$x_{0}\in \teich_{g,m}$
is a domain in $A_{2}(\mathbb{H}^{*},\Gamma_{0})$
which consists of $\varphi\in A_{2}(\mathbb{H}^{*},\Gamma_{0})$
such that $W_{\varphi}$ admits a quasiconformal extension
to $\hat{\mathbb{C}}$.
The Bers slice $\Bers{x_{0}}$ is bounded and
identified biholomorphically
with $\teich_{g,m}$.
Indeed,
any $x\in \teich_{g,m}$ corresponds to
$\varphi$ such that $\Gamma_{\varphi}$ is
the marked quasifuchsian group uniformizing $x_{0}$ and $x$
(cf. \cite{MR0111834}).
The closure $\overline{\Bers{x_{0}}}$ of $\Bers{x_{0}}$ in $A_{2}(\mathbb{H}^{*},\Gamma_{0})$
is called the \emph{Bers compactification}
of $\teich_{g,m}$.
The boundary $\partial\Bers{x_{0}}$ is called the \emph{Bers boundary}.
For $\varphi\in \overline{\Bers{x_{0}}}$,
$\Gamma_\varphi$ is a Kleinian surface group with isomorphism
$\rho_\varphi\colon \pi_1(\Sigma_{g,m})\cong \Gamma_0\to \Gamma_\varphi$.

\subsection{Boundary groups without APTs}
\label{subsec:BoundarygroupswithoutAPT}
A boundary point $\varphi\in \partial\Bers{x_0}$ is called a \emph{cusp} if
there is a non-parabolic element $\gamma\in \Gamma_{0}$
such that $\rho_{\varphi}(\gamma)$ is parabolic
(cf. \cite{MR0297992}).
Such $\gamma$ or $\rho_{\varphi}(\gamma)$
is called an \emph{accidental parabolic transformation} (APT)
of $\varphi$ or $\Gamma_{\varphi}$.
Let $\partial^{cusp}\Bers{x_{0}}$ be the set of cusps in $\partial\Bers{x_{0}}$
and set
$\partial^{mf}\Bers{x_{0}}
=\partial\Bers{x_{0}}-\partial^{cusp}\Bers{x_{0}}$.

For $\varphi\in \partial^{mf}\Bers{x_{0}}$,
the quotient manifold $\mathbb{H}^{3}/\Gamma_{\varphi}$
has two (non-cuspidal) ends
corresponding to $\Sigma_{g,m}\times (0,\infty)$
and $\Sigma_{g,m}\times (-\infty,0)$.
The negative end is geometrically finite and the surface at infinity
is conformally equivalent to $M_{0}$ (with orientation reversed).
To another end,
we assign a unique minimal and filling geodesic lamination,
called the \emph{ending lamination} for $\varphi$
(cf. \cite{MR847953} and \cite{Thuston-LectureNote}).

Let $x_0\in \teich_{g,m}$.
Let $\mathcal{PMF}^{mf}$ be the set of projective classes
of minimal and filling measured foliations.
By virtue of the ending lamination theorem and the Thurston double limit theorem,
we have the closed continuous surjective mapping
\begin{equation}
\label{eq:curve-complex-mininal-bers}
\Xi_{x_0}\colon \mathcal{PMF}^{mf}\to 
\partial^{mf} \Bers{x_0}
\end{equation}
which assigns $[F]\in \mathcal{PMF}^{mf}$
to the boundary group whose ending lamination is equal to $L(F)$.
The preimage of any point in $\partial^{mf} \Bers{x_0}$ is compact
(cf. \cite{MR2582104}).
$\mathcal{PMF}^{mf}$ contains
a subset $\mathcal{PMF}^{ue}$ consisting of uniquely ergodic measured foliations.
Let $\partial^{ue} \Bers{x_0}$ be the image of
$\mathcal{PMF}^{ue}$
under the identification \eqref{eq:curve-complex-mininal-bers}. 

For $x_1,x_2\in \teich_{g,m}$,
the change of the base points $\beta_{x_1,x_2}\colon
\Bers{x_1}\to \Bers{x_2}$
extends continuously to
$\Bers{x_1}\cup \partial^{mf}\Bers{x_1}\to \Bers{x_2}\cup \partial^{mf}\Bers{x_2}$
(cf. \cite{MR3330544}).
We denote by $\beta_{x_1,x_2}$ the extension for the simplicity.
In particular, the action of the mapping class group extends continuously on $\Bers{x_0}\cup \partial^{mf} \Bers{x_0}$ (cf. \cite{MR624803}). However, the action does not extend as homeomorphisms on the whole Bers compactification (cf. \cite{MR1037141}).

\subsection{Limits of Teichm\"uller rays in the Bers slice}
For $[H]\in \mathcal{PMF}$ and $x=(M,f)\in \teich_{g,m}$
the \emph{Teichm\"uller (geodesic) ray} $R_{[H],x}\colon [0,\infty)\to \teich_{g,m}$
for $[H]$ emanating from $x$ is defined as follows:
For $t\ge 0$, let $h_t\colon M\to h_t(M)$ is the quasiconformal mapping
with the Beltrami differential $\tanh(t)|q_{H,x}|/q_{H,x}$.
We set $R_{[H],x}(t)=(h_t(M),h_t\circ f)$.

The following proposition might be well-known.
However, 
we shall give a brief proof for confirmation.

\begin{proposition}
\label{prop:Teichmuller-limit}
Let $x_0\in \teich_{g,m}$.
For $x\in \teich_{g,m}$ and  $[H]\in \mathcal{PMF}^{mf}$,
the Teichm\"uller ray $R_{[H],x}$ converges to the totally degenerate
group without APT in $\partial^{mf}\Bers{x_{0}}$
whose ending lamination is $L(H)$.
\end{proposition}

\begin{proof}
Let $z_t=R_{[H],x}(t)$.
Let $\varphi \in \partial\Bers{x_{0}}$ be an accumulation point
of $\{z_t\}_{t\ge 0}$ as $t\to \infty$.
Let $\varphi_t\in \Bers{x_{0}}$ be the corresponding point to $z_t$ via the Bers embedding.
From \eqref{eq:comparison-hyp-ext},
for any $t\ge 0$,
the hyperbolic length ${\rm leng}_{\rho_{\varphi_t}}(H)$
of $H$ on the quotient manifold of $\rho_{\varphi_t}$ satisfies
\begin{align*}
{\rm leng}_{\rho_{\varphi_t}}(H)
\le
{\rm leng}_{z_t}(H)
&\le
\sqrt{2\pi(2g-2+m)}
\ext_{z_t}(H)^{1/2}
\\
&=
\sqrt{2\pi(2g-2+m)}
e^{-t}\ext_{x}(H)^{1/2}.
\end{align*}
By the continuity of the length function,
$H$ is not realizable in the marked Kleinian manifold associated to $\varphi$.
Hence,
the ending lamination associated to $\varphi$ is equal to $L(H)$.
From the ending lamination theorem,
such a Kleinian surface group is unique.
\end{proof}

Masur \cite{MR0385173} observed the same conclusion as Proposition \ref{prop:Teichmuller-limit}
for multi-curves.

\section{Thurston measure}
\label{sec:Thurston-measure}

\subsection{Thurston measure on $\mathcal{MF}$}
\label{subsec:Thurston-measure}
%
The \emph{Thurston measure} $\ThursM$ on $\mathcal{MF}$ is a unique locally finite $\mcg_{g,m}$-invariant ergodic measure, supported on the locus of filling measured laminations (cf. \cite{MR2424174}. See also \cite{MR1144770} and \cite{MR2415399}).
The Thurston measure satisfies that
for any measurable set
$E\subset \mathcal{MF}$
and $t>0$,
\begin{equation}
\label{eq:multple-TH-measure}
\ThursM(\{tF\mid F\in E\})=t^{2\convgenus}\ThursM(E).
\end{equation}
Let $\mathcal{BMF}_x=\{F\in \mathcal{MF}\mid \ext_x(F)\le 1\}$ for $x\in \teich_{g,m}$.
\begin{equation}
\label{eq:HM-fns}
{\rm Vol}_{Th}(x)=\mu_{Th}(\mathcal{BMF}_x)
\end{equation}
is a continuous function on $\teich_{g,m}$ with ${\rm Vol}_{Th}([\omega](x))={\rm Vol}_{Th}(x)$ for $x\in \teich_{g,m}$ and $[\omega]\in \mcg_{g,m}$ since $\mathcal{BMF}_{[\omega](x)}=[\omega](\mathcal{BMF}_x)$. The function \eqref{eq:HM-fns} is called the \emph{Hubbard-Masur function} on $\teich_{g,m}$ (cf. \cite[\S5.7]{MR3413977}).

\subsection{Thurston measure on $\mathcal{PMF}$}
\label{subsec:Thurstonmeasure}
For $x\in \teich_{g,m}$,
we define the \emph{unit sphere in terms of the extremal length function}
by
$$
\mathcal{SMF}_x=\partial \mathcal{BMF}_x=\{H\in \mathcal{MF}\mid \ext_{x}(H)=1\}.
$$
The projection $\mathcal{MF}-\{0\}\to \mathcal{PMF}$
induces a homeomorphism
$\proje_x\colon \mathcal{SMF}_x\to \mathcal{PMF}$.
We define a probability measure $\PThursM^x$ on
$\mathcal{SMF}_{x}$ by the cone construction
\begin{equation}
\label{eq:Thurston-measure-SMF}
\PThursM^{x}(E)=
\frac{
\ThursM(\{tG\mid G\in E,\ 0\le t\le 1\})}{{\rm Vol}_{Th}(x)}
\quad (E\subset \mathcal{SMF}_{x}).
\end{equation}
In this paper, we also call $\PThursM^{x}$ the \emph{Thurston measure} associated with $x\in \teich_{g,m}$ (cf. \cite[\S2.3.1]{MR2913101}). 
For $x,y\in \teich_{g,m}$,  a homeomorphism
\begin{equation*}
\proje_{x,y}\colon
\mathcal{SMF}_{x}
\ni G\to \frac{G}{\ext_{y}(G)^{1/2}}
\in \mathcal{SMF}_{y}.
\end{equation*}
induces
\begin{equation}
\label{eq:push-forward-marking_change}
(\proje_{x,y}^{-1})_*(\PThursM^{y})(E)=\PThursM^{y}(\proje_{x,y}(E))=\frac{{\rm Vol}_{Th}(x)}{{\rm Vol}_{Th}(y)}\int_E\frac{1}{\ext_y(F)^{\convgenus}}\PThursM^{x}(F)
\end{equation}
for a measurable set $E\subset \mathcal{SMF}_x$ from \eqref{eq:multple-TH-measure}.
Via the identification $\proje_x\colon \mathcal{SMF}_x\to \mathcal{PMF}$, we also regard $\PThursM^{x}$ as a probability measure on $\mathcal{PMF}$.

Notice that recently, the factor ${\rm Vol}_{Th}(x)/{\rm Vol}_{Th}(y)$ in \eqref{eq:push-forward-marking_change} is known to be equal to one by Dumas \cite{MR3413977}. However, our Poisson integral formula is proved without assuming Dumas' result and also gives another approach to it. Hence, we put the factor in \eqref{eq:push-forward-marking_change} (cf. Corollary \ref{coro:Hubbard-Masur-constant}).

%

\subsection{Push-forward measure on the Bers boundary}
\label{subsec:pushforwardmeasure-Bers-boundary}
For $x\in \teich_{g,m}$,
we set $\mathcal{SMF}^{mf}_{x}= \proje_x^{-1}(\mathcal{PMF}^{mf})$
and $\mathcal{SMF}^{ue}_{x}=\proje_x^{-1}(\mathcal{PMF}^{ue})$.
We define a probability (Borel) measure $\pushThursMBers_x$ on $\partial \Bers{x_0}$
as the pushforward measure of the Thurston measure $\PThursM^x$ via $\Xi_{x_0}\circ \proje_x\colon \mathcal{SMF}_x\to \partial\Bers{x_0}$:
\begin{equation}
\label{eq:defintion-pushforward-Th}
\int_{\partial\Bers{x_0}}f\,d\pushThursMBers_x
=\int_{\mathcal{SMF}_x}f\circ (\Xi_{x_0}\circ \proje_x)d
\PThursM^x
\end{equation}
for continuous functions $f$ on $\partial\Bers{x_0}$.
The superscript ``B" stands for the initial letter of ``Bers".
Masur \cite{MR644018} showed that
$\mathcal{SMF}_x^{ue}$ is of full measure in $\mathcal{SMF}_x$
with respect to $\PThursM^x$.
Hence,
the composition $f\circ (\Xi_{x_0}\circ \proje_x)$ is defined 
almost everywhere on $\mathcal{SMF}_x$.
Masur's observation also implies that $\partial^{ue}\Bers{x_0}$ is a set of full measure in $\partial\Bers{x_0}$
with respect to the pushforward measure
$\pushThursMBers_x$.

When we specify the base point $x_0$ of the Bers slice,
we denote by ${\boldsymbol \mu}^{B,x_0}_x$ instead of $\pushThursMBers_x$
(we only use this notation here).
The measure $\pushThursMBers_x$ is independent of the base point
of the Bers slice
in the sense that
\begin{equation*}
\int_{\partial\Bers{x_1}}f\circ \beta_{x_1,x_2}\,d{\boldsymbol \mu}^{B,x_1}_x
=
\int_{\partial\Bers{x_2}}f\,d{\boldsymbol \mu}^{B,x_2}_x
\end{equation*}
for any continuous function $f$ on $\partial \Bers{x_2}$
because $\beta_{x_1,x_2}\circ \Xi_{x_2}=\Xi_{x_1}$
(cf. \S\ref{subsec:BoundarygroupswithoutAPT}).

\section{Transverse measures and currents}
\subsection{Currents}
Let $M$ be a closed Riemann surface of genus $g(M)$ with finite marked points $P$.
Recall that a $1$-\emph{dimensional current} $T$ on $M$ is an element of the dual of the space of smooth $1$-forms (see e.g. \cite[\S3.1]{MR760450}).
By convention, we denote by $T\wedge\gamma$ the value of a form $\gamma$ by $T$. We set $\gamma\wedge T$ to be $-T\wedge \gamma$. Any closed one form $\xi$ on $M$ is thought of as a current such that
$$
\xi\wedge \gamma =\iint_M\xi\wedge \gamma.
$$

Let $\omega=\alpha+i\beta$ be a holomorphic $1$-form on $M$. Let $Z(\omega)$ be the union of zeros of $\omega$ and $P$. Let $v(\omega^2)$ be the vertical foliation of $\omega$ (cf. \S\ref{subsec:Hubbard-masur-theorem}). The leaves of $v(\omega^2)$ is oriented so that $\beta>0$.

Following McMullen \cite{MR4168686}, we define a current $\xi_H$ for $H\in \mathcal{MF}_{L(v(\omega^2))}$ as follows. Let $\{R_i\}_{i\in I}$ be a decomposition of $M$ by rectangles with respect to the flat stucture of $\omega$ such that the interior of each $R_i$ is contained in the complement of $Z(\omega)$. In the affine coordinates, $R_i$ is assumed to be represented as $[0,a_i]\times [0,b_i]$ in $\mathbb{R}^2\cong \mathbb{C}$.
The transverse measure $\mu=\mu_H$ of $H\in \mathcal{MF}_{v(\omega^2)}$ defines a measure on $\mu_{\tau_i}$ on $\tau_i=[0,a_i]\times \{0\}\subset \partial R_i$. For $x\in [0,a_i]$, we set $\ell_x$ the oriented vertical segment in $R_i$ emanating $(x,0)$. We define a current $T_H$ by 
\begin{equation}
\label{eq:def_current}
T_H\wedge\gamma=\sum_{i\in I}\int_0^{a_i}\left(\int_{\ell_x}\gamma\right)d\mu_{\tau_i}(x)
\end{equation}
for a smooth one form $\gamma$ on $M$. When we specify the transverse measure, we write $T_\mu$ instead of $T_H$. Since $\mu$ is a transverse measure, we can check that $T_H$ is defined independently of the choice of the rectangle decompositions. We also see that 
$$
dT_H=0,\quad (f\alpha)\wedge T_H=0,\quad T_H\wedge (g\beta)\ge 0
$$
for any smooth functions $f$ and $g$ with $g\ge 0$ on $X$ (cf. \cite[Proposition 3.1]{MR4168686}).
In particular,
\begin{equation}
\label{eq:intersection_omega}
T_H\wedge \beta =\sum_{i\in I}b_i \mu_{\tau_i}(\tau_i)= i(H, h(\omega^2))>0.
\end{equation}
where $h(\omega^2)$ is the horizontal foliation of $\omega$.

\subsection{Periods of currents}
Let $C\in H_1(M;\mathbb{R})$, let $\repro_C$ be the reproducing differential of the homology class $C$. Namely, 
$$
\int_C\alpha=-\iint_M\gamma\wedge \repro_C
$$
for all $C^1$-closed form $\gamma$ on $M$ (cf. \cite[\S II.3]{MR1139765}). Notice that
$$
C_1\cdot C_2=\iint_M\repro_{C_1}\wedge\repro_{C_2}=\int_{C_1}\repro_{C_2}
$$
for $C_1,C_2\in H_1(M;\mathbb{R})$, where the dot means the algebraic intersection number.
We define the period of the closed current $T$ along the homology class $C$ by
$$
\int_CT=-T\wedge \repro_C.
$$
In particular, for a canonical homology basis $\{A_j,B_j\}_{j=1}^{g(M)}$ on $H_1(M;\mathbb{R})$,
\begin{equation}
\label{eq:Twedgegamma}
T\wedge \gamma=\sum_{j=1}^{g(M)}\left\{\int_{A_j}T\int_{B_j}\gamma - \int_{B_j}T\int_{A_j}\gamma\right\}
\end{equation}
for a closed current $T$ and a smooth closed one-form $\gamma$ on $M$ (see the proof of \cite[III.2.3 Proposition]{MR1139765}). Equation \eqref{eq:Twedgegamma} induces the wedge product $T\wedge \xi$ between current $T$ and $\xi\in H^1(M,\mathbb{R})$.

A closed curve $\gamma\colon \mathbb{S}^1\to M$ (with marked point $P$) is said to be \emph{quasi-transversal} to the vertical foliation $v(\omega^2)$ if at every point $t\in \mathbb{S}^1$, either $\gamma(t)$ is in $Z(\omega)$, or $\gamma$ is locally near $t$ transversal to the underlying foliation of $v(\omega^2)$, or an inclusion into a leaf of the underlying foliation.
An oriented $C^1$-curve $c$ on $M$ is said to be \emph{decreasing} with respect to $\omega$ if $\alpha<0$ along $c$.
A quasi-transversal closed curve $c$ said to be \emph{decreasing} if all of its transversal parts are decreasing. The following is essentially due to Hubbard and Masur \cite{MR523212}.
\begin{proposition}
\label{prop:increasing_path}
Suppose a homology class $C$ is represented by a decreasing quasi-transvesal simple closed curve $c$.
Then,
$$
T_H\wedge \eta_{C}=i(c,H)
$$
for $H\in \mathcal{MF}_{L(v(\omega^2))}$.
\end{proposition}

\begin{proof}
Let $\mu=\mu_H$ be the transverse measure to $H$.
We cover $c$ by closed rectangles $\{R_i\}_{i=1}^N$ in terms of $\omega$ such that (1) ${\rm Int}(R_i)\cap {\rm Int}(R_j)=\emptyset$ for $i\ne j$ and the interior of each $R_i$ does not contains critical points and points of $P$; (2) the vertical segment parts of $c$ is contained in $\partial R_i$ for some $i\in I$, and; (3) $c$ intersects $\partial R_i$ only at the interiors of the vertical edges of $R_i$ for all $i=1$, $\cdots$, $N$. Let $N$ be a sufficiently narrow annular neighborhood of $c$ such that $N$ is contained in the interior of the union $R=\cup_{i=1}^N R_i$ and $N$ intersects $\partial R_i$ only at the interior of the vertical edges of $R_i$. We may assume that for $i\in I$, any component of $N\cap R_i$ is either one which connects two vertical sides or a half-disk neighborhood of a critical point or a point in $P$ in $\partial R_i$ (cf. Figure \ref{fig:intersection_rectangle}).
\begin{figure}[t]
\includegraphics[height = 3cm]{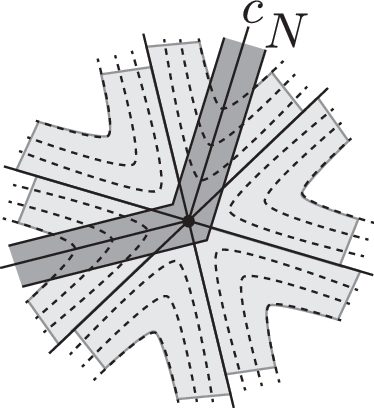}
\caption{Intersection between the neighborhood $N$ and the rectangles around a zero of $\omega$.}
\label{fig:intersection_rectangle}
\end{figure}

Let $U_0\cup U_1=R\setminus N$ such that $U_0$ lies the left of $c$. Under this convention, an oriented vertical segment in an rectangle $R_i$ intersects $c$ goes from $U_0$ to $U_1$ if $c$ is decreasing. Take a $C^\infty$ function $\varphi$ on $R$ such that $0\le \varphi \le 1$, $\varphi=0$ on $U_0$ and $\varphi=1$ on $U_1$. Then, $\eta_{C}$ is cohomologous to the closed form on $M$ defined by setting $d\varphi$ on $R$ and $0$ otherwise (cf. \cite[\S II.3]{MR1139765}).

Suppose $R_i$ is represented as $[0,a_i]\times [0,b_i]$ on $\mathbb{R}^2\cong \mathbb{C}$ under the flat structure of $\omega$. Let $\tau_i$ be the bottom horizontal edge of $R_i$. Let $\ell_x$ be the vertical segment in $R_i$ emanating $(x,0)$. By definition, 
\begin{align*}
T_H\wedge \eta_{C}
&=\sum_{i=1}^N\int_0^{a_i}\left(\int_{\ell_x}\eta_{C}\right)d\mu_{\tau_i}(x)=\sum_{i=1}^N\int_0^{a_i}\left(\int_{\ell_x}d\varphi\right)d\mu_{\tau_i}(x) \\
&= \sum_{i}\int_0^{a_i}d\mu_{\tau_i}(x) =\mu_H(c)=i(c,H),
\end{align*}
where the first term of the second line is given by summing all component of $N\cap R_i$ which connects vertical sides of $R_i$.
Notice in the above calculation that a component of $N\cap R_i$ which is a half-disk neighborhood of a point in $Z(\omega)$ in $\partial R_i$ does not contribute in the integration.
\end{proof}

\begin{proposition}
\label{prop:case-F}
Under the above notation, $T_{v(\omega^2)}={\rm Re}(\omega)$.
\end{proposition}
\begin{proof}
We use the notation around \eqref{eq:def_current}.
By definition,, $\mu_{\tau_i}=|dx|$ for each $\tau_i$. For any smooth one form $\gamma=Adx+Bdy$ on $M$,
\begin{align*}
T_{v(\omega^2)}\wedge \gamma
&=\sum_{i\in I}\int_0^{a_i}\left(\int_{\ell_x}\gamma\right)|dx| =\sum_{i\in I}\int_0^{a_i}\left(\int_{\ell_x}B dy\right)|dx| \\
&=\sum_{i\in I}\int_0^{a_i}\int_0^{b_i}B dxdy 
=\sum_{i\in I}\iint_{R_i} {\rm Re}(\omega)\wedge \gamma =\iint_M {\rm Re}(\omega)\wedge \gamma,
\end{align*}
which implies what we wanted.
\end{proof}

\subsection{Double branched covering spaces}
\label{subsec:double-branched-covering-space}
Let $x_{0}=(M_0,f_0)\in \teich_{g,m}$ and $q_0\in \mathcal{Q}_{x_0}$.
Let $\pi_{q_{0}}\colon \tilde{M}_{q_{0}}\to M_{0}$ be the double covering space associated to $\sqrt{q_{0}}$,
and $i_{q_{0}}\colon \tilde{M}_{q_{0}}\to \tilde{M}_{q_{0}}$ the covering transformation.
The lift of $\sqrt{q_{0}}$
defines a holomorphic $1$-form $\omega_{q_{0}}$ on $\tilde{M}_{q_{0}}$.
For $\mathbb{K}=\mathbb{Z}$, $\mathbb{R}$ or $\mathbb{C}$, we denote by $H_1(\tilde{M}_{q_0};\mathbb{K})^\pm$ the eigen space of the action of $i_{q_0}$ for the eigen value $\pm 1$.
%

Suppose that $q_0$ is not a square of an Abelian differential. For $H\in \mathcal{MF}_{L(v(q_0))}$, the transverse measure $\mu$ to the underlying foliation of $v(q_0)$ for $H$  is lifted to that for the vertical foliation $\tilde{\mu}$ of $\omega_{q_0}$ which is equivariant to the action of the involtion $i_{q_0}$. Namely, Let $\tau$ be a transverse arc to the vertical foliation of $\omega_{q_0}$. Then, for any measurable set $E\subset \tau$, $\tilde{\mu}_{i_{q_0}(\tau)}(i_{q_0}(E))=\tilde{\mu}_{\tau}(E)$. Therefore,
the current $T_{\tilde{\mu}}$ defined by the lift $\tilde{\mu}$ satisfies
$$
T_{\tilde{\mu}}(i_{q_0}^*(\gamma))=-T_{\tilde{\mu}}(\gamma)
$$
for any smooth one form $\gamma$ on $\tilde{M}_{q_0}$ since the the action of $i_{q_0}$ reverses the orientation of the leaves of the vertical foliation of $\omega_{q_0}$. This means that $T_{\tilde{\mu}}$ vanishes on $H_1(\tilde{M}_{q_0};\mathbb{K})^+$.

\section{Holomorphic coordinates associated to extremal lengths}
\label{sec:holomorphic-coordinate-ext}

\subsection{Double coverings for essentially complete measured foliations}
Henceforth,
we assume that $F\in \mathcal{MF}$ is essentially complete. In this case,
the section
$$
\teich_{g,m}\ni x \mapsto q_{F,x}\in\mathcal{Q}_{g,m}
$$
is smooth.

Fix $x_0=(M_0,f_0)\in \teich_{g,m}$ and set $q_0=q_{F,x_0}$.
For all $x=(M,f)\in \teich_{g,m}$,
the differential $q_{F,x}$ is generic and has $4g-4+m$ simple zeros and $m$ simple poles at marked points. 
The genus of $\tilde{M}_{q_0}$ is equal $\convgenus+g$ where $\convgenus=3g-3+m$.
Let $\tilde{Z}(q_0)\subset \tilde{M}_{q_0}$ be the set consisting of zeros of $\omega_{q_0}$ and the preimages of marked points of $M$. Notice from the discussion in \cite[Proposition 2.6]{MR523212} that $H_1(\tilde{M}_{q_0};\mathbb{K})^-$ and $H_1(\tilde{M}_{q_0},\tilde{Z}(q_0);\mathbb{K})^-$ are naturally isomorphic because of the exact sequence
\begin{equation*}
\label{eq:exact_sequence}
H_1(\tilde{Z}(q_0);\mathbb{K})^-
\to H_1(\tilde{M}_{q_0};\mathbb{K})^-
\to H_1(\tilde{M}_{q_0},\tilde{Z}(q_0);\mathbb{K})^-
\to H_0(\tilde{Z}(q_0);\mathbb{K})^-,
\end{equation*}
and $H_k(\tilde{Z}(q_0);\mathbb{K})^-=\{0\}$ for $k=0$, $1$,
since the involution $i_{q_0}$ fixes every points of $\tilde{Z}(q_0)$.

For any $x\in \teich_{g,m}$, there is a natural bijection between the set of transverse measures to the underlying foliation of $F$ and that to the underlying foliation of the vertical foliation of $q_{F,x}$. Hence, for any $x=(M,f)\in \teich_{g,m}$ and $H\in \mathcal{MF}_{L(F)}$, we define a current $T_{H,x}$ on the set of smooth one forms on $\tilde{M}_{q_{F,x}}$ associated to the lift $\tilde{\mu}$ of the transverse measure  $\mu$ for $H$.

\subsection{Periods of currents revisited}
 Since the Teichm\"uller space is contractible, the surface bundle $\cup_{(M,f)\in \teich_{g,m}}M$ is trivial (cf. \cite{MR0430318}). 
 
For $x=(M,f)\in \teich_{g,m}$, $M\setminus Z(q_{F,x})$ is the Riemann surface of type $(g,4g-4+2m)$ where $Z(q_{F,x})$ is the zeros of $q_{F,x}$ since each $q_{F,x}$ is generic. The zeros of the differentials $q_{F,x}$ ($x\in \teich_{g,m}$) define mutually disjoint $(4g-4+2m)$-smooth sections of the surface bundle $\cup_{(M,f)\in \teich_{g,m}}M$. This means that the zeros and the poles of $q_{F,x}$ is marked (labeled). By taking the double branched covering space along the sections, we obtain the surface bundle $\cup_{x\in \teich_{g,m}}\tilde{M}_{q_{F,x}}\to \teich_{g,m}$ which is a trivial bundle. 
Let  $\tilde{\Sigma}_F\to \Sigma_{g,m}$ be the double branched covering associated to $F$, which is branched at the singularities of $F$ (see \cite[\S4]{MR644018}).
Then, there is a homeomorphism $\tilde{f}_x\colon \tilde{\Sigma}_F\to \tilde{M}_x$ respecting the trivialization commutes the diagram
\begin{equation}
\label{eq:commute_double_covering}
\begin{CD}
\tilde{\Sigma}_F@>{\tilde{f}_x}>> \tilde{M}_{q_{F,x}} \\
@VVV @VVV \\
\Sigma_{g,m}@>{f}>> M
\end{CD}
\end{equation}
after an appropriate choice of the marking $f$ of $x=(M,f)$ so that $f$ maps the (marked) singularities of $F$ to those of $q_{F,x}$. The homeomorphism $\tilde{f}_x$ induces the identification
\begin{equation}
\label{eq:identification_homology_cohomology}
H_1(\tilde{M}_{q_{F,x}};\mathbb{K})^\pm\cong H_1(\tilde{\Sigma}_F;\mathbb{K})^\pm,\quad
H^1(\tilde{M}_{q_{F,x}};\mathbb{K})^\pm\cong H^1(\tilde{\Sigma}_F;\mathbb{K})^\pm
\end{equation}
for $x\in \teich_{g,m}$ and $\mathbb{K}=\mathbb{Z}$, $\mathbb{R}$ and $\mathbb{C}$, which are compatible with the duality. Namely, we denote by the pairing
$$
\int_C\xi
$$
for $C\in H_1(\tilde{M}_{q_{F,x}};\mathbb{K})$ and $\xi\in H^1(\tilde{M}_{q_{F,x}};\mathbb{K})$.
After identifying $H_1(\tilde{M}_{q_{F,x}};\mathbb{K})^\pm\cong H_1(\tilde{\Sigma}_F;\mathbb{K})^\pm$ (via the trivialization of the surface bundle), 
$$
\int_C\xi=\int_C\xi'
$$
when $\xi\in H^1(\tilde{M}_{q_{F,x}};\mathbb{K})$ and $\xi'\in H^1(\tilde{\Sigma}_F;\mathbb{K})$ are identified.

\begin{proposition}
\label{prop:locally_constant}
For any $C\in H_1(\tilde{\Sigma}_F;\mathbb{K})^-$ and $H\in \mathcal{MF}_{L(F)}$, 
the period of $T_{H,x}$ along $C$ is dependent only on $H$ and $C$.
\end{proposition}

\begin{proof}
Indeed, this proposition is proved by using the same argument as that by Hubbard and Masur in \cite[Proposition 4.3]{MR523212}. We shall sketch the proof.

It suffices to show that the period of $C$ is locally constant because the Teichm\"uller space is connected.

Let $x=(M,f)\in \teich_{g,m}$.
For a simple closed curve $\alpha\in \mathcal{S}$ and $x=(M,f)\in \teich_{g,m}$, we define a decreasing quasi-transverse curve $\tilde{\alpha}=\tilde{\alpha}_{F,x}$ on $\tilde{M}_{q_{F,x}}$ as follows. We represent $\alpha$ as a quasi-transversal curve to the vertical foliation $v(q_{F,x})$ (we may take the geodesic representative of $\alpha$ with respect to the flat metric $|q_{F,x}|$. See \cite[Theorem 17.4]{MR743423}). Then, $\tilde{\alpha}=\tilde{\alpha}_{F,x}$  is defined by taking a lift of $\alpha$ and by orienting transverse segments in the lift so that they are decreasing.
By applying the similar argument in \cite[Chapter II, S4]{MR523212}, we can see that 
$$
H_1(\tilde{M}_{q_{F,x}};\mathbb{R})^-\cong H_1(\tilde{M}_{q_{F,x}},\tilde{Z}(q_{F,x});\mathbb{R})^-
$$
is generated by $\{\tilde{\alpha}_{F,x}\mid \alpha\in \mathcal{S}\}$. We modify $\tilde{\alpha}_{F,x}$ at the self intersection points in $\tilde{M}_{q_{F,x}}\setminus \tilde{Z}(q_{F,x})$ as Figure \ref{fig:modification}, so that $\tilde{\alpha}_{F,x}$ can be thought of as a union of decreasing quasi-transversal simple closed curves since the modification does not change the homology class (see the proof of \cite[Proposition 2.2]{MR523212}). From Proposition \ref{prop:increasing_path}, the period of the lift $\tilde{\alpha}=\tilde{\alpha}_{F,x}$ of $\alpha$ is equal to $2i(F,\alpha)$. 
\begin{figure}[t]
\includegraphics[height=3cm]{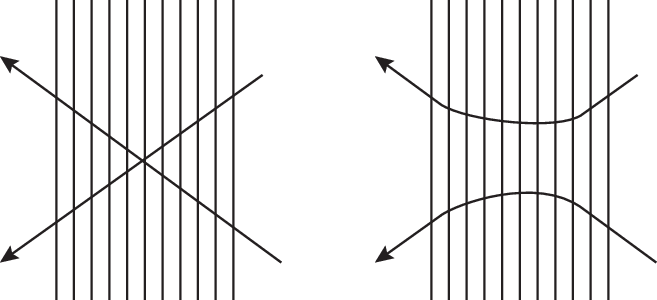}
\caption{Modification at self intersection points of two decreasing paths.}
\label{fig:modification}
\end{figure}
Since 
the decreasingness of (simple) closed curves is invariant under a (small) deformation, the period is a constant function around $x$.
\end{proof}

\subsection{Embedding into $H^1(\tilde{M}_{q_0};\mathbb{R})^-$}
In this section, we identify two branched coverings $\tilde{M}_{q_0}\to M_0$ and $\tilde{\Sigma}_F\to \Sigma_{g,m}$ via \eqref{eq:commute_double_covering}.
Under the identifications \eqref{eq:identification_homology_cohomology}, we define $\mathcal{H}_F\colon \teich_{g,m}\to H^1(\tilde{M}_{q_{0}};\mathbb{R})^-$
by
\begin{equation}
\label{eq:embedding-HF}
\mathcal{H}_F(x)=[{\rm Im}(\omega_{q_{F,x}})]
\end{equation}
where $[{\rm Im}(\omega_{q_{F,x}})]$ is the cohomology class in $H^1(\tilde{M}_{q_{0}};\mathbb{R})^-$ corresponding to the closed one form ${\rm Im}(\omega_{q_{F,x}})$. From Proposition \ref{prop:locally_constant} and \eqref{eq:Twedgegamma},
$$
T_{H,x_0}\wedge \mathcal{H}_F(x)=T_{H,x}\wedge {\rm Im}(\omega_{q_{F,x}})
$$
for $H\in \mathcal{MF}_{L(F)}$ and $x\in \teich_{g,m}$.

The following theorem is recognized as a counter part to Bonahon's theory
on the shearing coordinates on Teichm\"uller space
(cf. \cite[Theorem 20]{MR1413855}.
See also \cite[Th\'eo\`eme 6.1]{MR3695474}).

\begin{theorem}[Embedding]
\label{thm:embedding-H1}
The mapping $\mathcal{H}_F\colon \teich_{g,m}\to
H^1(\tilde{M}_{q_0};\mathbb{R})^-\cong H^1(\tilde{\Sigma}_{F};\mathbb{R})^-$ is a smooth embedding.
The image of $\mathcal{H}_F$ is the open cone 
$$
\mathcal{C}_{F}=\{\xi\in H^1(\tilde{M}_{q_0};\mathbb{R})^-
\mid
T_{H,x_0}\wedge \xi>0\
\mbox{for $H\in \mathcal{MF}_{L(F)}$}\}.
$$
\end{theorem}

\begin{proof}
The proof of the non-singularity of the differential of $\mathcal{H}_F$ is postponed to \S\ref{subsec:complex_structure} for the sake of readability. 
From \eqref{eq:intersection_omega}, the image of $\mathcal{H}_F$ is contained in the cone $\mathcal{C}_{F}$.
We only check here the properness of the mapping $\mathcal{H}_F$.

Let $K\subset \mathcal{C}_{F}$ be a compact set.
Since intersection number is continuous and $\mathcal{C}_F$ is an open cone with vertex $0$,
from \eqref{eq:intersection_omega},
there are $\epsilon_0=\epsilon_0(K)>0$ and $\epsilon_1=\epsilon_1(K)>0$ such that 
\begin{equation}
\label{eq:filling-limit}
\epsilon_0\le T_{H,x_0}\wedge \mathcal{H}_F(x)=2i(h(q_{F,x}),H)\le \epsilon_1
\end{equation}
for $x\in \teich_{g,m}$ with $\mathcal{H}_F(x)\in K$,
and $H\in \mathcal{MF}_{L(F)}\cap \mathcal{SMF}_{x_0}$.

Let $\{x_n=(M_n,f_n)\}_{n=1}^\infty\subset\teich_{g,m}$ with $\mathcal{H}_F(x_n)\in K$
for $n\in \mathbb{N}$. 
Let $H_n=h(q_{F,x_n})\in \mathcal{MF}$ be the horizontal foliation of $q_{F,x_n}$. Then, there is $t_n>0$ such that $t_nH_n$ converges to $G\in \mathcal{MF}-\{0\}$. From \eqref{eq:filling-limit}, $t_n$ does not diverge.
We claim
\begin{claim}
\label{claim:1-0}
The sequence $\{t_n\}_{n=1}^\infty$ is bounded from below.
\end{claim}
We postpone the proof of Claim \ref{claim:1-0} after finishing the proof of this theorem.
From Claim \ref{claim:1-0}, We may assume that $H_n$ converges to $G'=t_0G$ for some $t_0>0$.
From \eqref{eq:filling-limit}, $i(G',F)\ne 0$ and there is $x_\infty\in \teich_{g,m}$ such that $h(q_{F,x_\infty})=G'$. Since $h(q_{F,x_n})=H_n\to G'$, $q_{F,x_n}\to q_{F,x_\infty}$ as $n\to \infty$. Threfore $x_n\to x_\infty$. Thus, we conclude that $\mathcal{H}_F$ is proper.
\end{proof}

\begin{proof}[Proof of Claim \ref{claim:1-0}]
Suppose to the contrary that $t_n\to 0$.
From \eqref{eq:filling-limit},
$$
i(G,F)=\lim_{n\to \infty}t_ni(H_n,F)=0.
$$
Therefore, $F$ and $G$ are topologically equivalent, that is, the underlying foliations of $F$ and $G$ are isotopic  (cf. \cite{MR662738}).

Let $\tau$ be a complete train track on $\Sigma_{g,m}$ carrying $F$. Since $F$ is essentially complete, $F$ is thought of as an interior point of the transverse measures on $\tau$ (cf. \cite[Lemma 3.1.2]{MR1144770}). We may assume that each component of $\Sigma_{g,m}\setminus \tau$ contains a unique point of $P$.
Let $V(\tau)\subset \mathcal{MF}$ be the measured foliation carried by $\tau$. From the above discussion, $G$ is thought of as an interior point of $V(\tau)$. Hence, $H_n$ is carried by $\tau$ for sufficiently large $n$. Let $\mu_n\in V(\tau)$ be the transverse measure associated to $H_n$. The pair $(\tau,\mu_n)$ is regarded as a weighted graph on $\Sigma_{g,m}$. The branched covering $\tilde{\Sigma}_F\to \Sigma_{g,m}$ induces an orientation covering of $\tilde{\tau}\to \tau$ (cf. \cite[\S3.2]{MR1144770}) and the weighted graph $(\tau, \mu_n)$ defines a cycle $\tilde{\mu}_n$ on $\tilde{\Sigma}_F$ (via an orientation of $\tilde{\tau}$). After choosing an orientation of $\tilde{\tau}$ appropriately, the homology class $[\tilde{\mu}_n]$ in $H_1(\tilde{M}_{q_0};\mathbb{R})$ is thought of as the dual to $\mathcal{H}_F(x_n)$ in the sense that
\begin{equation}
\label{eq:intersection_with_C}
\int_C\mathcal{H}_F(x_n)=[\tilde{\mu}_n]\cdot C
\end{equation}
for any $C\in H_1(\tilde{M}_{q_0};\mathbb{R})$, where the dot in the right-hand side means the algebraic intersection number.

Let $\beta$ be a simple closed curve on $\Sigma_{g,m}$ which hits $\tau$ efficiently (cf. \cite[p.19]{MR1144770}). 
%
Since $\beta$ intersects $\tau$ only at branches of $\tau$,
$\beta$ is presented by the union of paths connecting singular points of $F$ which intersects $\tau$ only at branches between components containing the singular points (in the presentation, $\beta$ may not be simple). Let $\tilde{\beta}\subset \tilde{\Sigma}_{F}\cong \tilde{M}_{q_0}$ be the lift of $\beta$. We orient each branch of $\tilde{\beta}$ appropriately (as $\tilde{\alpha}_{F,x}$ in the proof of Proposition \ref{prop:locally_constant}) , the homology classs $[\tilde{\beta}]$ is in $H_1(\tilde{M}_{q_0},\tilde{Z}(q_0); \mathbb{R})^-\cong H_1(\tilde{M}_{q_0}; \mathbb{R})^-$ and satisfies $[\tilde{\mu}_n]\cdot [\tilde{\beta}]=2i(\beta,H_n)$. Since $\mathcal{H}_F(x_n)$ is contained in a compact set $K$, from \eqref{eq:intersection_with_C}, $i(\beta,H_n)$ is bounded from above.
Therefore,
$$
i(G,\beta)=\lim_{n\to \infty}i(t_n H_n,\beta)=0,
$$
which is a contradiction since $G$ is topologically equivalent to $F$.
\end{proof}

\subsection{Coordinates to $\mathbb{R}^{6g-6+2m}$}
Let $\{a_{i},b_{i}\}_{i=1}^{\convgenus+g}$
be a canonical homology basis of $H_{1}(\tilde{M}_{q_{0}};\mathbb{R})$
in the sense that $a_{j}\cdot a_{k}=b_{j}\cdot b_{k}=0$ and $a_{j}\cdot b_{k}=\delta_{jk}$
for $1\le j,k\le \convgenus+g$.
Suppose thet $\{a_{i},b_{i}\}_{i=1}^{\convgenus+g}$
satisfies
\begin{enumerate}
\item
$\{(\pi_{q_{0}})_{*}(a_{j}),(\pi_{q_{0}})_{*}(b_{j})\}_{j=1}^{g}$ is a canonical homology basis 
on $H_1(M_{q_{0}};\mathbb{Z})$; and
\item
$(i_{q_{0}})_{*}(a_{j})-a_{g+j}=
(i_{q_{0}})_{*}(b_{j})-b_{g+j}=0$, $(i_{q_{0}})_{*}(a_{2g+k})+a_{2g+k}=
(i_{q_{0}})_{*}(b_{2g+k})+b_{2g+k}=0$
for $1\le j\le g$ and $1\le k\le \convgenus-g$.
\end{enumerate}
We define a canonical homology basis $\{A_k,B_k\}_{k=1}^{\convgenus}$
of $H_1(\tilde{M}_{q_0};\mathbb{R})^-$
by
$$
A_{j}=\begin{cases}
{\displaystyle
\frac{1}{\sqrt{2}}(a_{j}-a_{g+j})} & (1\le j\le g), \\
a_{g+j} & (g+1\le j\le \convgenus),
\end{cases}
\quad
B_{j}=\begin{cases}
{\displaystyle
\frac{1}{\sqrt{2}}(b_{j}-b_{g+j})} & (1\le j\le g), \\
b_{g+j} & (g+1\le j\le \convgenus) 
\end{cases}
$$
for $j=1,\cdots,\convgenus$
(cf. Figure \ref{fig:branchedcovering}).
\begin{figure}
\includegraphics[width = 4.5cm]{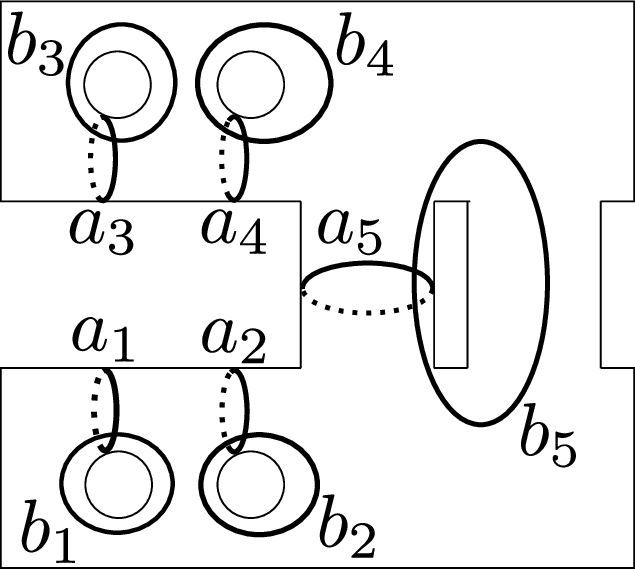}
\caption{A branched covering space and a canonical homology basis of the homology group
in the case of $g=2$ and $m=0$.}
\label{fig:branchedcovering}
\end{figure}
Consider mappings
$\Phi^{A}$, $\Phi^{B}\colon \teich_{g,m}\to \mathbb{R}^{\convgenus}$
and $\Phi\colon \teich_{g,m}\to \mathbb{R}^{2\convgenus}$
defined by
\begin{align*}
\Phi^{A}(x)
&=\left(
{\rm Im}\int_{A_{1}}\omega_{q_{F,x}},
\cdots,
{\rm Im}\int_{A_{\convgenus}}\omega_{q_{F,x}}
\right),
\\
\Phi^{B}(x)
&=\left(
{\rm Im}\int_{B_{1}}\omega_{q_{F,x}},
\cdots,
{\rm Im}\int_{B_{\convgenus}}\omega_{q_{F,x}}
\right),
\\
\Phi(x)&=(\Phi^{A}(x),\Phi^{B}(x)).
\end{align*}
The mapping $\Phi$ factors through $H_1(\tilde{M}_{q_0};\mathbb{R})^-$
with the embedding \eqref{eq:embedding-HF}.
We denote by ${\bf y}=({\bf y}^{A},{\bf y}^{B})=(y^{A}_{1},\cdots,y^{A}_{\convgenus},y^{B}_{1},\cdots,y^{B}_{\convgenus})$ the coordinates of
$\mathbb{R}^{2\convgenus}=\mathbb{R}^{\convgenus}\times \mathbb{R}^{\convgenus}$.


Let $H\in \mathcal{MF}_{L(F)}$.
From Proposition \ref{prop:locally_constant},
\begin{align*}
{\bf a}^{A}_{H}
&
=\left(
\int_{A_1}T_{H,x_0}
\cdots,
\int_{A_{\convgenus}}T_{H,x_0}
\right),
\\
{\bf a}^{B}_{H}
&
=
\left(
\int_{B_1}T_{H,x_0}
\cdots,
\int_{B_{\convgenus}}T_{H,x_0}
\right)
\end{align*}
depend only on $H$
and
$\mathcal{MF}_{L(F)}\ni H\mapsto ({\bf a}^{A}_{H},{\bf a}^{B}_{H})\in \mathbb{R}^{2\convgenus}$
is continuous
(see also \cite[Lemma 4.3]{MR523212}).
We define a convex cone
$$
\mathbb{H}_{F}=
\bigcap_{H\in \mathcal{MF}_{L(F)}}
\{{\bf y}=({\bf y}^{A},{\bf y}^{B})\in \mathbb{R}^{2\convgenus}=\mathbb{R}^{\convgenus}\times \mathbb{R}^{\convgenus}
\mid
{\bf a}^{A}_{H}\veccdot ({\bf y}^{B})^{T}-
{\bf a}^{B}_{H}\veccdot ({\bf y}^{A})^{T}>0\},
$$
where the superscript ``$T$" means the transpose of matrices (vectors).
From 
Theorem \ref{thm:embedding-H1} and \eqref{eq:Twedgegamma},
we have

\begin{proposition}[Coordinates]
\label{prop:chart}
The mapping
$$
\Phi\colon \teich_{g,m}\to \mathbb{R}^{2\convgenus}
$$
is a diffeomorphism onto the image.
The image coincides with $\mathbb{H}_{F}$.
\end{proposition}

\subsection{Differentials of the periods}
\label{subsec:description-complex-structure}
We first notice the following variational formula
obtained in \cite{MR3715450}:
For $v\in T_x\teich_{g,m}$,
let $\lambda\mapsto f(\lambda)\in \teich_{g,m}$ be the holomorphic disk
defined around $\lambda=0$ with $f(0)=x$ and $f_*(\partial/\partial\lambda\mid_{\lambda=0})=v$.
Then
\begin{equation}
\label{eq:differential-formula-period}
\begin{cases}
{\displaystyle
\left(
\frac{\partial}{\partial \lambda} 
\int_C\omega_{q_{F,f(\cdot)}}
\right)_{\lambda=0}}
&
{\displaystyle=
\int_C\overline{\frac{\pi_{q_{F,x_0}}^*(\eta_v)}{\omega_{q_{F,x_0}}}}
},
\\
{\displaystyle
\left(
\frac{\partial}{\partial \overline{\lambda}}
\int_C\omega_{q_{F,f(\cdot)}}
\right)_{\lambda=0}
}
&=
{\displaystyle
-\int_C\frac{\pi_{q_{F,x_0}}^*(\eta_v)}{\omega_{q_{F,x_0}}}
}
\end{cases}
\end{equation}
for $C\in H_1(\tilde{M}_{q_{F,x_0}};\mathbb{R})^-$
(cf. \cite[Lemma 4.1]{MR3715450}).
From Propositions \ref{prop:case-F} and \ref{prop:locally_constant},
$\lambda\mapsto
\int_C{\rm Re}(\omega_{q_{F,f(\cdot)}})$
is a constant function for each $c\in H_{1}(\tilde{M}_{q_{0}})$,
Hence, when $\lambda=s+it$,
we have
\begin{align}
\frac{\partial}{\partial s}
\left({\rm Im}
\int_C\omega_{q_{F,f(\cdot)}}
\right)_{\lambda=0}
\label{eq:period-computation1}
&=
-\sqrt{-1}\left(
\left.\frac{\partial \left(\chi_{q_{F,f(\lambda)}}(c)\right)}{\partial \lambda}\right|_{\lambda=0}
+
\left.\frac{\partial \left(\chi_{q_{F,f(\lambda)}}(c)\right)}{\partial \overline{\lambda}}\right|_{\lambda=0}
\right)
\\
&
=-\sqrt{-1}\left(
\overline{\int_{c}\frac{\pi_{q_{F,x_0}}^{*}\left(\eta_{v}\right)}{\omega_{q_{F,x}}}}
-
\int_{c}\frac{\pi_{q_{F,x_0}}^{*}\left(\eta_{v}\right)}{\omega_{q_{F,x}}}
\right)
\nonumber
\\
&=-2{\rm Im}\int_{c}\frac{\pi_{q_{F,x_0}}^{*}\left(\eta_{v}\right)}{\omega_{q_{F,x}}}.
\nonumber
\end{align}

Let ${\bf y}\in \mathbb{H}_{F}$ and $x=(M,f)\in \teich_{g,m}$
with ${\bf y}=\Phi(x)$.
For $j=1,\cdots,\convgenus$.
let $\varphi^{j}_{F,x}$ be the Abelian differential on $\tilde{M}_{q_{F,x}}$
normalized by
$\displaystyle\int_{A_{k}}\varphi^{j}_{F,x}=\delta_{jk}$
for $k=1,\cdots,\convgenus$.
Set
$$
\pi_{jk}=\pi_{jk}({\bf y})=\int_{B_{k}}\varphi^{j}_{F,x}
\quad (j,k=1,\cdots,\convgenus)
$$
and $\Pi=\Pi({\bf y})=(\boldsymbol{\pi}_{1},\cdots,\boldsymbol{\pi}_{\convgenus})=(\pi_{jk}({\bf y}))$.
From the definition,
\begin{equation}
\label{eq:canonicaldifferential}
\omega_{q_{F,x}}=({\bf a}^{A}_{F}+i{\bf y}^{A})\veccdot \boldsymbol{\varphi}_{F,x}{}^T,
\end{equation}
where ${\bf y}^A=\Phi_A(x)$, ${\bf y}^B=\Phi_B(x)$ and $\boldsymbol{\varphi}_{F,x}=(\varphi^{1}_{F,x},\cdots, \varphi^{\convgenus}_{F,x})$.
Comparing the $B$-periods of both sides of \eqref{eq:canonicaldifferential},
we have the following relation:
\begin{equation}
\label{eq:periodrelation0}
{\bf a}^{B}_{F}+i{\bf y}^{B}=({\bf a}^{A}_{F}+i{\bf y}^{A})\veccdot \Pi.
\end{equation}

\subsection{The complex structure on $\teich_{g,m}$ under the coordinates}
\label{subsec:complex_structure}
For $j=1,\cdots,\convgenus$,
we define a tangent vector $v_{j}=v_{j}(x)\in T_x\teich_{g,m}$ by
\begin{align}
\varphi^{j}_{F,x}&=2\sqrt{-1}\frac{\pi_{q_{F,x_0}}^{*}\left(\eta_{v_{j}(x)}\right)}{\omega_{q_{F,x}}}=-2\frac{\pi_{q_{F,x_0}}^{*}\left(\eta_{v_{\convgenus+j}(x)}\right)}{\omega_{q_{F,x}}}
\label{eq:period-computation}
\end{align}
 where $\eta_{v_{j}(x)}$ is the $q_{F,x}$-realization for the tangent vector $v_{j}(x)$ (cf. \eqref{eq:q0-realization}). 
Since \eqref{eq:qo-realization-1} is an anti-complex isomorphism, $\sqrt{-1}v_{j}=v_{\convgenus+j}$ and $\sqrt{-1}v_{\convgenus+j}=-v_{j}$ for all $j=1,\cdots,\convgenus$. Set $\{{\bf v}^{A},{\bf v}^{B}\}$
$=
\{(v_{k})_{k=1}^{\convgenus},(v_{k})_{k=\convgenus+1}^{2\convgenus}\}$.
From \eqref{eq:period-computation1},
we have
\begin{equation*}
\Phi_{*}(\{{\bf v}^{A},{\bf v}^{B}\})=(\partial^{A},\partial^{B})
\begin{pmatrix}
I_{\convgenus} & 0 \\
{\rm Re}(\Pi) & {\rm Im}(\Pi)
\end{pmatrix}
\end{equation*}
where
$
\partial^{A}
=(\partial/\partial y^{A}_{1},\cdots,\partial/\partial y^{A}_{\convgenus})$
and $\partial^{B}
=(\partial/\partial y^{B}_{1},\cdots,\partial/\partial y^{B}_{\convgenus})
$.
Since the matrix in the right-hand side is non-singular, the differential $\Phi_*$ of $\Phi$ is non-singular. Therefore, so is the differential of $\mathcal{H}_F$.

Set $\{{\bf V}^{A},{\bf V}^{B}\}=\Phi_{*}(\{{\bf v}^{A},{\bf v}^{B}\})$.
We define an almost complex structure $\mathcal{J}$ on $\mathbb{H}_{F}$
by
\begin{align}
\mathcal{J}\left(V_{j}\right)=V_{\convgenus+j},
&\quad
\mathcal{J}\left(V_{\convgenus+j}\right)=-V_{j}
\quad (j=1,\cdots,\convgenus)
\label{eq:ComplexStructure}
\end{align}
(cf. \cite[\S2, Chapter IV]{MR1393941}).
From \eqref{eq:ComplexStructure},
we see that $\mathcal{J}\circ \Phi_*=\Phi_*\circ (\sqrt{-1}\times)$,
where $\sqrt{-1}\times$ is the standard complex structure
on $\teich_{g,m}$ defined by multiplying
by $\sqrt{-1}$.
Therefore, $\Phi$ is holomorphic.
Thus, from Theorem \ref{thm:embedding-H1}, we obtain
\begin{proposition}[Holomorphic chart]
\label{prop:holomorphic-chart}
$\Phi\colon (\teich_{g,m},\sqrt{-1}\times)\to (\mathbb{H}_{F},\mathcal{J})$
is biholomorphic.
\end{proposition}

\section{Complex analysis}
\label{subsec:result-from-pluripotential-theory}
\subsection{PSH exhaustions and Boundary measures}
\label{subsubsec:PSH exhaustions}
Let $\Omega$  be a domain in $\mathbb{C}^n$.
A function $u$ on $\Omega$ is said to be \emph{plurisubharmonic} (PSH)
if for each $a\in \Omega$ and $b\in \mathbb{C}^n$,
the function $\lambda\mapsto u(a+\lambda b)$ is subharmonic
or identically $-\infty$ on every component of the set $\{\lambda\in \mathbb{C}\mid a+\lambda b\in \Omega\}$.
A function $u$ on $\Omega$ is, by definition,
\emph{pluriharmonic} if $u\in C^2(\Omega)$ and the restriction
to any complex line that meets $\Omega$ is harmonic.
The real part of a holomorphic function on $\Omega$ is pluriharmonic
(e.g. \cite{MR1150978}).
A function $u\colon \Omega\to [-\infty,0)$ is said to be an \emph{exhaustion} on $\Omega$
if $u^{-1}([-\infty,r))$ is relatively compact in $\Omega$ for $r<0$.
A domain $\Omega\subset \mathbb{C}^n$ is said to be
\emph{hyperconvex}
if there is a continuous PSH exhaustion
$u\colon \Omega\to [-\infty,0)$
(cf. \cite[D\'efinition 2.1]{MR881709}).

Let $\Omega$ be a bounded hyperconvex domain in $\mathbb{C}^n$.
Let $u\colon \Omega\to [-\infty,0)$ be a continuous PSH-exhaustion
on $\Omega$ and 
set $S_u(r)=\{z\in \Omega\mid u(z)=r\}$
and $B_u(r)=\{z\in \Omega\mid u(z)<r\}$.
For $r<0$,
there is a Borel measure $\mu_{u,r}$ on $\mathbb{C}^n$
supported on $S_u(r)$ 
which satisfies the Lelong-Jensen formula:
\begin{equation}
\label{eq:Demailly-LJ-formula}
\int_{S_u(r)}Vd\mu_{u,r}=\int_{B_u(r)}V(dd^cu)^n+\int_{B_u(r)}(r-u)dd^cV\wedge (dd^cu)^{n-1}
\end{equation}
for any PSH function $V$ on $\Omega$,
where $d=\partial+\overline{\partial}$,
$d^c=\sqrt{-1}(\overline{\partial}-\partial)$
and
$dd^c=2\sqrt{-1}\partial\overline{\partial}$
(cf. \cite[D\'efinition 0.1]{MR881709}).

When $\int_{\Omega}(dd^cu)^n<\infty$,
there is a Borel measure $\mu_u$ on $\mathbb{C}^n$ which is supported on $\partial \Omega$
such that $\mu_{u,r}$ converges to $\mu_u$ weakly on $\mathbb{C}^n$
and $\mu_{u}(\partial \Omega)=\int_{\Omega}(dd^cu)^n$.
The measure $\mu_u$ is called the \emph{boundary measure associated to $u$}
(cf. \cite[Th\'eor\`eme et D\'efinition 3.1]{MR881709}).
We will use the following results due to Demailly later.

\begin{proposition}[Th\'eor\`eme 3.4 in \cite{MR881709}]
\label{prop:thm-BT}
Let $\Omega$ be a bounded hyperconvex domain in $\mathbb{C}^N$. Let $u, v\colon \Omega\to (-\infty, 0]$ be PSH exhaustion. Suppose that $u\le v$ and $\int_\Omega (dd^cu)^N<\infty$. Then
$$
\int_{\Omega}(dd^cv)^N\le \int_{\Omega}(dd^cu)^N.
$$
and $\mu_v\le \mu_u$ on $\partial \Omega$.
\end{proposition}

\begin{proposition}[Th\'eor\`eme 3.8 in \cite{MR881709}]
\label{prop:Demally1}
Let $\Omega$ be a bounded hyperconvex domain in $\mathbb{C}^n$.
Let $u,v\colon \Omega\to [-\infty,0)$ be PSH-continuous exhaustions 
with
$$
\int_\Omega (dd^xu)^n,\ \int_\Omega(dd^c v)^n<\infty.
$$
Suppose that there is an relatively open $N_0\subset \partial \Omega$ and a function $\lambda\ge 0$
on $N_0$ such that for all $z\in N_0$,
$$
\limsup_{w\to z}\frac{u(z)}{v(z)}=\lambda(z)<\infty.
$$
Then $d\mu_u\le \lambda^nd\mu_v$ on $N_0$.
If the lim-sup is the limit,
$d\mu_u=\lambda^nd\mu_v$ on $N_0$.
\end{proposition}


\subsection{Pluricomplex Green function and Pluriharmonic measures}
Demailly observed that
for any bounded hyperconvex domain $\Omega$ in $\mathbb{C}^n$
and $w\in \Omega$,
there is a unique PSH function $g_{\Omega,w}\colon \Omega\to [-\infty,0)$
such that
\begin{enumerate}
\item
$(dd^cg_{\Omega,w})^n=(2\pi)^n\delta_w$,
where $\delta_w$ is the Dirac measure with support at $w$;
and
\item
$g_{\Omega,w}(z)=\sup_v\{v(z)\}$,
where the supremum runs over all non-positive PSH function $v$
on $\Omega$ with $v(z)\le \log\|z-w\|+O(1)$ around $z=w$
\end{enumerate}
(cf. \cite[Th\'eor\`eme 4.3]{MR881709}).
The function $g_\Omega(w,z)=g_{\Omega,w}(z)$ is called 
the \emph{pluricomplex Green function} on $\Omega$.
The pluricomplex Green function $g_{\Omega,w}$ is a continuous exhaustion on $\Omega$
for fixed $w\in \Omega$.
For $w\in \Omega$,
the boundary measure associated with $u_{\Omega,w}=(2\pi)^{-n}g_{\Omega,w}$ is called the \emph{pluriharmonic measure of point $w$}
(cf. \cite[(5.2) D\'efinition]{MR881709}). It is known that the pluriharmonic measure of two distinct points are mutually absolutely continuous (cf. \cite[(5.3) Th\'eor\`eme]{MR881709}).
The pluriharmonic measure of point $w$ provides the following integral formula:
\begin{equation}
\label{eq:integral-formula}
\int_{\partial \Omega}V(\zeta)d\mu_{u_{\Omega,w}}
=
V(w)+\int_{\Omega}dd^cV\wedge |u_{\Omega,w}|(dd^cu_{\Omega,w})^{n-1}
\end{equation}
for any PSH function $V$ which is continuous on $\overline{\Omega}$.

\section{The Monge-Amp\`ere measure of extremal lengths}
\label{sec:measures-on-horosphere-monge-ampere-measure}

\subsection{Extremal length functions on the chart and Frames}
We discuss with the chart $\Phi$ on $\teich_{g,m}$ in Propositions \ref{prop:chart} and \ref{prop:holomorphic-chart}
for an essentially complete measured foliation $F$ as complex coordinates.
For the simplicity
we set
\begin{equation*}
\extsymp_G({\bf y})=\ext_{\Phi^{-1}({\bf y})}(G)
\end{equation*}
for ${\bf y}\in \mathbb{H}_F$ and $G\in \mathcal{MF}$.
From \eqref{eq:periodrelation0} and Riemann's bilinear relation,
we have
\begin{align}
\extsymp_F({\bf y})
&=
\frac{1}{2}
\left(
{\bf a}^{A}_{F}\veccdot ({\bf y}^{B})^{T}-
{\bf a}^{B}_{F}\veccdot ({\bf y}^{A})^{T}
\right)
=\frac{1}{2}
\left(
{\bf a}^{A}_{F}\veccdot {\rm Im}(\Pi)\veccdot ({\bf a}^{A}_{F})^{T}
+{\bf y}^{A}\veccdot {\rm Im}(\Pi)\veccdot ({\bf y}^{A})^{T}
\right)
\label{eq:extremallengthfunction1-1}
\end{align}
for ${\bf y}\in \mathbb{H}_F$.

Under the chart $\Phi$,
we define smooth vector fields
${\bf Z}=(Z_{j})_{1\le j\le \convgenus}$
and $\overline{{\bf Z}}=(\overline{Z}_{j})_{1\le j\le \convgenus}$,
and $1$-forms
$\boldsymbol{\Omega}=(\Omega^{k})_{1\le k\le \convgenus}$
and
$\overline{\boldsymbol{\Omega}}=(\overline{\Omega}^{k})_{1\le k\le \convgenus}$ by
\begin{align}
({\bf Z},\overline{{\bf Z}})
&=\left(\frac{1}{2}({\bf V}^A-\sqrt{-1}{\bf V}^B),\frac{1}{2}({\bf V}^A+\sqrt{-1}{\bf V}^B)\right)=\frac{1}{2}(\partial^{A},\partial^{B})\veccdot
\begin{bmatrix}
I_{\convgenus} & I_{\convgenus} \\
\overline{\Pi} & \Pi
\end{bmatrix},
\nonumber
\\
(\boldsymbol{\Omega},\overline{\boldsymbol{\Omega}})
&=-\sqrt{-1}(d{\bf y}^{A},d{\bf y}^{B})\veccdot
\begin{bmatrix}
\Pi & -\overline{\Pi} \\
-I_{\convgenus} & I_{\convgenus}
\end{bmatrix}
\begin{bmatrix}
{\rm Im}(\Pi) & 0 \\
0 & {\rm Im}(\Pi)
\end{bmatrix}
^{-1}.
\label{eq:Omega-def}
\end{align}
From the observations in \S\ref{subsec:description-complex-structure} and \S\ref{subsec:complex_structure},
$Z_k$ and $\overline{Z}_k$ are $(1,0)$ and $(0,1)$-vector fields,
and $\Omega_k$ and $\overline{\Omega}_k$ are
$(1,0)$ and $(0,1)$-forms on $\mathbb{H}_F$
such that $\Omega_k(Z_l)=\overline{\Omega}_k(\overline{Z}_l)=\delta_{kl}$
and $\Omega_k(\overline{Z}_l)=\overline{\Omega}_k(Z_l)=0$
for $k,l=1,\cdots,\convgenus$.
The systems $\{Z_k\}_{k=1}^{\convgenus}$
and  $\{\Omega_k\}_{k=1}^{\convgenus}$ 
are a smooth frame on the holomorphic tangent bundle
and a smooth coframe of the holomorphic cotangent bundle
on $\mathbb{H}_{F}$ (and hence on $\teich_{g,m}$).
\begin{proposition}[\cite{Towards-complex-miyachi2017}]
\label{prop:derivatives-extremallengthfunction}
The differentials and the Levi-form
of the extremal length function of $F$ satisfies the following:
\begin{align}
\left(\partial \extsymp_F\right)_{{\bf y}}
&=
-\frac{\sqrt{-1}}{4}\sum_{k=1}^{\convgenus}
({\bf a}_{F}^{A}+\sqrt{-1}{\bf y}^{A})\veccdot{\rm Im}(\boldsymbol{\pi}_{k})\,
\Omega^{k},
\label{eq:thm:extremallengthfunction1}
\\
\left(\overline{\partial}\extsymp_F\right)_{{\bf y}}
&=
\frac{\sqrt{-1}}{4}\sum_{k=1}^{\convgenus}
({\bf a}_{F}^{A}-\sqrt{-1}{\bf y}^{A})\veccdot{\rm Im}(\boldsymbol{\pi}_{k})\,
\overline{\Omega}^{k},
\label{eq:thm:extremallengthfunction2}
\\
\partial \overline{\partial}\extsymp_F
&=\frac{1}{4}
\sum_{k,l=1}^{\convgenus}
{\rm Im}(\pi_{kl})\,\Omega^{k}\wedge \overline{\Omega}^{l}
\label{eq:thm:extremallengthfunction3}
\end{align}
for ${\bf y}\in \mathbb{H}_{F}$.
\end{proposition}
The formulas in in Proposition \ref{prop:derivatives-extremallengthfunction} are calculated in \cite{Towards-complex-miyachi2017}.
For the completeness, we shall check the formulas.
We take the tangent vectors $\{v_j\}_{j=1,\cdots,\convgenus}$ as \S\ref{subsec:complex_structure}. 
Let $\mu_j$ be a Beltrami differential on $x=(M,f)\in \teich_{g,m}$ satisfying $v_j(x)=[\mu_j]\in T_x\teich_{g,m}$ for $j=1,\cdots,\convgenus$.
Then,
\begin{align*}
-\int_{M}\mu_j q_{F,x}
&=-\int_{M}\frac{\overline{\eta_{v_j}}}{|q_{F,x}|}q_{F,x}
=-\frac{\sqrt{-1}}{4}\int_{\tilde{M}_{q_{F,x}}}
\omega_{q_{F,x_0}}\wedge \overline{\frac{\pi_{q_{F,x}}^*(\eta_{v_j})}{\omega_{q_{F,x}}}}
\\
&=\frac{1}{8}\int_{\tilde{M}_{q_{F,x}}}
\omega_{q_{F,x}}\wedge \overline{\varphi^j_{F,x}}
\\
&=\frac{1}{8}\sum_{k=1}^{\tilde{g}}
\left(
\int_{A_k}\omega_{q_{F,x}}\overline{\int_{B_k}\varphi^j_{F,x}}
-\int_{B_k}\omega_{q_{F,x}}\overline{\int_{A_k}\varphi^j_{F,x}}
\right) \\
&=\frac{1}{8}\left(
({\bf a}_F^A+\sqrt{-1}{\bf y}^A)\cdot \overline{\boldsymbol{\pi}_{j}}
-
({\bf a}_F^A+\sqrt{-1}{\bf y}^A)\cdot \boldsymbol{\pi}_{j}
\right)
\\
&=-\frac{\sqrt{-1}}{4}({\bf a}_F^A+\sqrt{-1}{\bf y}^A)\cdot {\rm Im}(\boldsymbol{\pi}_{j}).
\end{align*}
Since $\Phi_*(v_j)=Z_j$, from Gardiner's formula (\cite{MR736212}), we have \eqref{eq:thm:extremallengthfunction1} and \eqref{eq:thm:extremallengthfunction2}.

Let $v=\sum_{k=1}^{\convgenus}a_iv_j$. From Theorem 5.1 in \cite{MR3715450}, the Levi form $\mathcal{L}(\extsymp_F)[v,\overline{v}]$ at $x=(M,f)\in \teich_{g,m}$ satisfies
\begin{align*}
\mathcal{L}(\extsymp_F)[v,\overline{v}] &=2\int_{M}\frac{|\eta_{v}|^2}{|q_{F,x}|}
=2\sum_{j,k=1}^{\convgenus}a_j\overline{a_k}\int_{M}\frac{\overline{\eta_{v_j}}\eta_{v_k}}{|q_{F,x}|} \\
&=2\sum_{j,k=1}^{\convgenus}a_j\overline{a_k}\frac{\sqrt{-1}}{4}\int_{\tilde{M}_{q_{F,x}}}\frac{\pi_{q_{F,x}}^*(\eta_{v_k})}{\omega_{q_{F,x}}}\wedge\overline{\frac{\pi_{q_{F,x}}^*(\eta_{v_j})}{\omega_{q_{F,x}}}}\\
&=\frac{1}{4}\sum_{j,k=1}^{\convgenus}{\rm Im}(\pi_{jk})a_j\overline{a_k}
\end{align*}
from \eqref{eq:period-computation} since $\Pi$ is symmetric.
This implies \eqref{eq:thm:extremallengthfunction3}.

\subsection{Vector fields tangent to Teichm\"uller disks}
We define a $(1,0)$-vector field ${\bf X}$ on $\mathbb{H}_F$ ($\cong \teich_{g,m}$) by
\begin{align}
{\bf X}={\bf X}_{{\bf y}}
&=-2\sqrt{-1}({\bf a}^{A}_{F}-\sqrt{-1}{\bf y}^{A})\veccdot {\bf Z}
\label{eq:vector-X}
\\
&=-\sqrt{-1}({\bf a}^{A}_{F}-\sqrt{-1}{\bf y}^{A})\veccdot (I_{\convgenus},\overline{\Pi})
\begin{pmatrix}
(\partial^A)^T
\nonumber
\\
(\partial^B)^T
\end{pmatrix}
\nonumber
\\
&=-
({\bf y}^{A}+\sqrt{-1}{\bf a}^{A}_{F})\veccdot (\partial^A)^T
-({\bf y}^{B}+\sqrt{-1}{\bf a}^{B}_{F})\veccdot (\partial^B)^T
\nonumber
\end{align}
(cf. \eqref{eq:periodrelation0}). The vcctor field ${\bf X}$ is tangent to the Teichm\"uller disk defined by the Hubbard-Masur differentials with vertical foliation $F$.

\begin{proposition}
\label{prop:vectorfield-X}
The tangent vector field ${\bf X}$ corresponds to the
$(1,0)$-vector associated to the infinitesimal Beltrami differential
$\overline{q_{F,x}}/|q_{F,x}|$ at $x\in \teich_{g,m}$.
\end{proposition}

\begin{proof}
Let $x=(M,f)\in \teich_{g,m}$.
Let $v\in T_x\teich_{g,m}$ be
the tangent vector associated to the infinitesimal Beltrami differential
$\overline{q_{F,x}}/|q_{F,x}|$.
By the defintion of the $q_{F,x}$-realization,
$\eta_v=q_{F,x}$ (cf. \eqref{eq:q0-realization}).
Let $\lambda\mapsto f(\lambda)$ be a holomorphic disk
defined around $\lambda=0$ which satisfies
$f(0)=x$ and $f_*(\partial/\partial\lambda\mid_{\lambda=0})=v$.
From \eqref{eq:differential-formula-period},
\begin{align*}
(\Phi^A\circ f)_*
\left(
\left.
\frac{\partial}{\partial \lambda} 
\right|_{\lambda=0}
\right)
&=
\frac{1}{2\sqrt{-1}}
\left(
\int_{A_i}\overline{\frac{\pi_{q_{F,x_0}}^*(q_{F,x})}{\omega_{q_{F,x}}}}
-\overline{
\left(-\int_{A_i}\frac{\pi_{q_{F,x_0}}^*(q_{F,x})}{\omega_{q_{F,x}}}
\right)
}
\right)_{i=1}^{\convgenus}
\\
&=
-\sqrt{-1}({\bf a}^A_F-\sqrt{-1}{\bf y}^A),
\\
(\Phi^B\circ f)_*
\left(
\left.
\frac{\partial}{\partial \lambda} 
\right|_{\lambda=0}
\right)
&=
\frac{1}{2\sqrt{-1}}
\left(
\int_{B_i}\overline{\frac{\pi_{q_{F,x_0}}^*(q_{F,x})}{\omega_{q_{F,x}}}}
-\overline{
\left(-\int_{B_i}\frac{\pi_{q_{F,x_0}}^*(q_{F,x})}{\omega_{q_{F,x}}}
\right)
}
\right)_{i=1}^{\convgenus}
\\
&=
-\sqrt{-1}({\bf a}^B_F-\sqrt{-1}{\bf y}^B).
\end{align*}
This means $\Phi_*({\bf v})={\bf X}_{\Phi(x)}$
from \eqref{eq:vector-X}.
%
\end{proof}
\subsection{Monge-Amp\`ere measures associated with extremal lengths}
From \eqref{eq:Omega-def} and \eqref{eq:thm:extremallengthfunction3},
\begin{align}
(dd^c\extsymp_F)^{\convgenus}
&=
\convgenus!\left(\frac{\sqrt{-1}}{2}\right)^{\convgenus}
\det ({\rm Im}(\Pi))\Omega^1\wedge\overline{\Omega}^{1} 
\wedge \cdots \wedge\Omega^{\convgenus}\wedge\overline{\Omega}^{\convgenus}
\label{ddcNExt}
\\
&=
\convgenus!\,dy^A_1\wedge dy^B_1\wedge
\cdots \wedge dy^A_{\convgenus}\wedge dy^B_{\convgenus}.
\nonumber
\end{align}
Namely,
the Monge-Amp\`ere measure of $\extsymp_F$
coincides with the constant multiple of 
the Euclidean measure under the chart.
Set ${\bf W}=-({\bf X}+\overline{{\bf X}})/4$
and ${\bf W}^c=({\bf X}-\overline{{\bf X}})/4\sqrt{-1}
=\mathcal{J}({\bf W})$.
Then,
\begin{align}
\label{eq:WW}
{\bf W}
=\frac{1}{2}
\left({\bf y}^{A}\veccdot (\partial^A)^T
+
{\bf y}^{B}\veccdot (\partial^B)^T
\right)
\ \ \mbox{and} \ \
{\bf W}^c
=\frac{1}{2}
\left({\bf a}^{A}_{F}\veccdot (\partial^A)^T
+
{\bf a}^{B}_{F}\veccdot (\partial^B)^T
\right).
\end{align}
From Proposition \ref{prop:derivatives-extremallengthfunction},
we have
\begin{align}
d\extsymp_F[{\bf W}]
&=
d^c\extsymp_F[{\bf W}^c]
=
\frac{1}{2}\extsymp_F,
\label{eq:extremal-differential-W}
\\
d\extsymp_F[{\bf W}^c]
&=
d^c\extsymp_F[{\bf W}]
=
0,
\label{eq:extremal-differential-Wc}
\\
{\bf W}\lrcorner dd^c\extsymp_F
&=
d^c\extsymp_F,
\label{eq:W-contract-ddcExt1}
\\
{\bf W}^c\lrcorner dd^c\extsymp_F
&=
-d\extsymp_F,
\label{eq:Wc-contract-ddcExt1}
\\
{\bf W}\lrcorner (dd^c\extsymp_F)^{\convgenus}
&=
\convgenus(dd^c\extsymp_F)^{\convgenus-1}\wedge d^c\extsymp_F
\label{eq:W-contract-ddcExt}
\end{align}
on $\mathbb{H}_F\cong \teich_{g,m}$,
where $\lrcorner$ stands for the contraction
(e.g.
\cite{MR1211412}).
Define a function $\extrecip_G$ on $\teich_{g,m}$ by
$$
\extrecip_G({\bf y})=-\frac{1}{\extsymp_G({\bf y})}
=-\frac{1}{\ext_{\Phi^{-1}({\bf y})}(G)}
$$
for ${\bf y}\in \mathbb{H}_F$ and $G\in \mathcal{MF}$.
From \cite[Theorem 5.3]{MR3715450},
$\extrecip_{G}$ is a continuous PSH-function on $\teich_{g,m}$
for all $G\in \mathcal{MF}$.
Notice
from \eqref{eq:extremal-differential-W},
\eqref{eq:extremal-differential-Wc},
\eqref{eq:W-contract-ddcExt1}
and 
\eqref{eq:Wc-contract-ddcExt1}
that
\begin{align*}
{\bf W}\lrcorner (dd^c\extrecip_F)
&=\frac{{\bf W}\lrcorner (
\extsymp_F dd^c\extsymp_F-2d\extsymp_F\wedge d^c\extsymp_F)}
{\extsymp_F^3}=0,
\\
{\bf W}^c\lrcorner (dd^c\extrecip_F)
&=\frac{{\bf W}^c\lrcorner (
\extsymp_F dd^c\extsymp_F-2d\extsymp_F\wedge d^c\extsymp_F)}
{\extsymp_F^3}=0.
\end{align*}
Hence,
the $(1,0)$-vector field ${\bf X}$ on $\mathbb{H}_F$ is in the null-space
of the complex Hessian $dd^c\extrecip_F$ of $\extrecip_F$,
and $\extrecip_F$ satisfies
the \emph{homogeneous Monge-Amp\`ere equation}
$$
(dd^c\extrecip_F)^{\convgenus}=0
$$
on $\mathbb{H}_F\cong \teich_{g,m}$
(cf. \cite[\S3.1]{MR1150978}).
From \eqref{eq:W-contract-ddcExt},
we obtain
\begin{align}
\label{ddcNresExt1}
\left(dd^c\extrecip_F\right)^{\convgenus-1}\wedge d^c\extrecip_F
&=
\frac{(dd^c\extsymp_F)^{\convgenus-1}\wedge d^c\extsymp_F}
{(\extsymp_F)^{2\convgenus}}
=
\frac{1}{\convgenus}
{\bf W}\lrcorner
\left(
\frac{(dd^c\extsymp_F)^{\convgenus}}
{(\extsymp_F)^{2\convgenus}}
\right).
\end{align}
In particular
$(dd^c\extrecip_F)^{\convgenus-1}\ne 0$.

%
%
\subsection{Measures on the horospheres}
For $G\in \mathcal{MF}$ and $R>0$,
we define the \emph{horosphere} for $G$ by
\begin{align*}
{\rm HS}(G;R)=\{x\in\teich_{g,m}\mid \ext_{x}(G)=R^2\}.
\end{align*}
From \eqref{eq:extremallengthfunction1-1},
under the coordinates in Proposition \ref{prop:chart},
the horosphere ${\rm HS}(F;R)$ is represented as
the affine subspace
\begin{equation}
\label{eq:HS-chart}
\{
({\bf y}^{A},{\bf y}^{B})\in
\mathbb{H}_F
\mid
{\bf a}^{A}_{H}\veccdot ({\bf y}^{B})^{T}-
{\bf a}^{B}_{H}\veccdot ({\bf y}^{A})^{T}=2R^2\}.
\end{equation}
Henceforth,
we also denote by ${\rm HS}(F;R)$ the set \eqref{eq:HS-chart}
under the coordinates in Proposition \ref{prop:chart}.
From Proposition \ref{prop:vectorfield-X},
the tangent vector field ${\bf W}$ is the gradient vector field 
of the extremal length function for $F$.
From \eqref{ddcNresExt1},
the contraction
\begin{equation}
\label{eq:measure-on-horosphere}
d\HSmeas_{F,R}
=
\frac{1}{\convgenus}
{\bf W}\lrcorner
\left(
\frac{(dd^c\extsymp_F)^{\convgenus}}
{(\extsymp_F)^{2\convgenus}}
\right)
=
\left(dd^c\extrecip_F\right)^{\convgenus-1}\wedge
d^c\extrecip_F
\end{equation}
is a non-trivial Borel measure on the horosphere ${\rm HS}(F;R)$.
Let $t\in\mathbb{R}$,
we set $T_t({\bf y})=e^{2t}{\bf y}$.
From 
\eqref{eq:extremallengthfunction1-1},
\eqref{eq:WW} and
\eqref{eq:measure-on-horosphere},
\begin{equation}
\label{eq:isomeasure-mu}
\HSmeas_{F,e^{2t}R}(T_t(E))=\HSmeas_{F,R}(E)
\end{equation}
for all Borel set $E\subset {\rm HS}(F,R)$.
From \eqref{eq:measure-on-horosphere} and \eqref{eq:isomeasure-mu},
\begin{equation}
\label{compactizon}
\HSmeas_{F,R}(E)
\asymp
\int_{\hat{E}}
(dd^c\extsymp_F)^{\convgenus}
\end{equation}
holds for any Borel set $E\subset {\rm HS}(F,R)$ and $R>0$,
where $\hat{E}=\{T_{t-\log (\sqrt{2}R)}({\bf y})\in \mathbb{H}_F\mid {\bf y}\in E, 0\le t\le (\log 2)/2\}$
and the constants for the comparison depend only on the topology of $\Sigma_{g,m}$.

\subsection{Comparison between $\HSmeas_{F,R}$ and $\PThursM^{x_0}$}
\label{subsec:Remark-embedding}
Consider a mapping
\begin{equation}
\label{eq:HF-MF}
\Psi\colon
\mathbb{H}_F\ni {\bf y}\mapsto h(q_{F,\Phi^{-1}({\bf y})})\in \mathcal{MF}.
\end{equation}
Since the piecewise linear structure on $\mathcal{MF}$
is determined by the intersection number function
associated with some finite system of simple closed curves,
the mapping \eqref{eq:HF-MF} is a piecewise linear homeomorphism onto its image
(cf.  \cite{MR1810534},
\cite[Expos\'e 6 and Appendice]{MR568308}
and \cite{MR662738}).
Hence the pushforward measure
$\Psi_*(\left(dd^c\extsymp_F\right)^{\convgenus})$
via the mapping \eqref{eq:HF-MF} is locally comparable with the Thurston measure
$\ThursM$.
Compare another treatment of $\ThursM$ due to Masur in \cite[\S4]{MR644018}.

Let $\mathcal{PMF}^F\subset \mathcal{PMF}$ be the projection of the image $\Psi(\mathbb{H}_F)\subset \mathcal{MF}-\{0\}$. $\mathcal{PMF}^F$ consists of the projective classes of measured foliations transverse to $F$.
For $x\in \teich_{g,m}$,
let $\mathcal{SMF}_x^F\subset \mathcal{SMF}_x$ be the corresponding 
subset via the identification $\mathcal{SMF}_x\cong \mathcal{PMF}$ discussed in \S\ref{subsec:Thurstonmeasure}.
The set $\mathcal{SMF}^F_x$ is an open subset of $\mathcal{SMF}_x$.
For $R>0$,
we define a homeomorphism
\begin{equation}
\label{eq:TRHS}
{\bf T}_{R,x}\colon \mathcal{SMF}_x^F\to {\rm HS}(F,R)
\end{equation}
in such a way that for $G\in \mathcal{SMF}_x^F$,
$\Psi({\bf T}_{R,x}(G))$ is projectively equivalent to $G$
in $\mathcal{MF}$.
%
\begin{proposition}[Comparison between $\HSmeas_{F,R}$ and $\PThursM^{x}$]
\label{prop:comparison}
Let $x\in \teich_{g,m}$.
Let $G_0\in \mathcal{SMF}_{x}$
such that $G_0$ and $F$ are transverse.
Let $U$ be a neighborhood of $G_0$ with $\overline{U}\subset \mathcal{SMF}^F_{x}$.
Then,
$$
\HSmeas_{F,R}({\bf T}_{R,x}(E))\asymp \PThursM^{x}(E)
$$
for all $R>0$ and all Borel set $E$ on $U$,
where $\asymp$ means that the measures are comparable
with constants independent of the choice of the set $E$ in $U$,
but may depend on $U$.
\end{proposition}

\begin{proof}
Since the image ${\bf T}_{R,x}(U)$ is relatively compact in $\mathbb{H}_F$,
from the above discussion and \eqref{compactizon},
the measure
$\HSmeas_{F,R}$ is comparable with the pushforward
measure $({\bf T}_{R,x})_*(\PThursM^{x})$ on ${\bf T}_{R,x}(U)$.
From \eqref{eq:isomeasure-mu} (or \eqref{compactizon} again),
the constants for the comparison are independent of $R>0$.
\end{proof}

\subsection{Monge-Amp\`ere mass of reciprocal of extremal length}
\label{subsec:Pluri-Green-Extremal-length}
Krushkal \cite{MR1142683}
observed that
the pluricomplex Green function $g_{\teich_{g,m}}$ on $\teich_{g,m}$
is represented as
\begin{equation}
\label{eq:Green-function}
g_{\teich_{g,m}}(x,y)=\log\tanh d_T(x,y)
\quad
(x,y\in \teich_{g,m}).
\end{equation}
See also \cite{miyachi-pluripotentialtheory1} for another proof.

For $G,H\in \mathcal{MF}$,
we define
a continuous PSH-function on $\teich_{g,m}$ by
$$
\extrecip_{G,H}(x)=\max\{\extrecip_G(x),\extrecip_H(x)\}.
$$

\begin{proposition}
\label{prop:exhausion-F-G}
When $G,H\in \mathcal{MF}$ are transverse,
the function $\extrecip_{G,H}$ is a continuous
and negative PSH-exhaustion on $\teich_{g,m}$,
and satisfies
\begin{equation}
\label{eq:comparison-uFG-AND-dT}
u_{G,H}(x)\asymp -e^{-2d_T(x_0,x)}
\quad
(x\in \teich_{g,m}),
\end{equation}
where the constants for the comparison depend only on $G$, $H$ and $x_0$.
\end{proposition}

\begin{proof}
Since $G$ and $H$ are transverse,
there is an $\epsilon_0>0$ such that $i(G,J)^2+i(H,J)^2\ge \epsilon_0$
for all $J\in \mathcal{SMF}_{x_0}$.
Minsky's inequality \eqref{eq:minsky-inequality}
implies
$$
\ext_x(G)+\ext_x(H)\ge (i(G,J)^2+i(H,J)^2)/\ext_x(J)\ge \epsilon_0/\ext_x(J).
$$
From Teichm\"uller's theorem,
for any $x\in \teich_{g,m}-\{x_0\}$,
there is a unique $J\in \mathcal{SMF}_{x_0}$ satisfying
$\ext_x(J)=e^{-2d_T(x_0,x)}$
(cf. \cite[\S5.2.3]{MR1215481}).
From \eqref{eq:comparison-K}
\begin{equation}
\label{eq:comparison-exhaustion-ext1}
\epsilon_0e^{2d_T(x_0,x)}\le 
\ext_x(G)+\ext_x(H)
\le (\ext_{x_0}(G)+\ext_{x_0}(H))e^{2d_T(x_0,x)}.
\end{equation}
Since extremal length functions are positive functions,
\begin{equation}
\label{eq:comparison-exhaustion-ext2}
-\frac{2}{\ext_x(G)+\ext_x(H)}
\le u_{G,H}(x)
\le 
-\frac{1}{\ext_x(G)+\ext_x(H)}
\end{equation}
for all $x\in \teich_{g,m}$.
The comparison \eqref{eq:comparison-uFG-AND-dT}
follows from \eqref{eq:comparison-exhaustion-ext1}
and \eqref{eq:comparison-exhaustion-ext2}.
\end{proof}

Let $K$ be a compact set in $\teich_{g,m}$ containing $x_0$
in the interior.
From the Krushkal formula \eqref{eq:Green-function}
and Proposition \ref{prop:exhausion-F-G},
we have
\begin{equation}
\label{eq:comparison-green-dT}
g_{\teich_{g,m}}(x_0,x)\asymp -e^{-2d_T(x_0,x)}\asymp \extrecip_{G,H}(x)
\quad
(x\in \teich_{g,m}-K),
\end{equation}
where the constants for the first comparison depend only on $K$.
%
%


\begin{proposition}[Finiteness of MA-mass and Boundary measure for $\extrecip_{G,H}$]
\label{prop:ddcuFG-volume-finite}
When $G,H\in \mathcal{MF}$ are transverse,
$$
\int_{\teich_{g,m}}(dd^c\extrecip_{G,H})^{\convgenus}<\infty,
$$
and the boundary measure $\mu_{\extrecip_{G,H}}$ of $\extrecip_{G,H}$ and the pluriharmonic measure are comparable  on the Bers boundary in the sense that they are mutually absolutely continuous and the Radon-Nikodym derivatives are bounded.
\end{proposition}

\begin{proof}
We identify $\Bers{x_0}$ with $\teich_{g,m}$ via the Bers embedding
(cf. \S\ref{subsec:Bers-slice}).
Fix $r<0$.
We set
$$
g_{\teich_{g,m},x_0;r}(x)=\max\{r,g_{\teich_{g,m}}(x_0,x)\}.
$$
Then,
$g_{\teich_{g,m},x_0;r}$ is a continuous PSH-exhaustion on $\teich_{g,m}$.
Since $g_{\teich_{g,m},x_0;r}$ coincides with $g_{\teich_{g,m}}(x_0,\,\cdot\,)$
in the outside of a compact set containing $x_0$,
$$
\int_{\teich_{g,m}}(dd^cg_{\teich_{g,m},x_0;r})^{\convgenus}<\infty.
$$
Since both $g_{\teich_{g,m},x_0;r}(x)$
and $\extrecip_{G,H}(x)$ are negative bounded continuous exhaustions,
from \eqref{eq:comparison-green-dT},
there is a constant $A=A(G,H,x_0,r)$, $B=B(G,H,x_0,r)>0$ such that
$$
Ag_{\teich_{g,m},x_0;r}(x)<\extrecip_{G,H}(x)<Bg_{\teich_{g,m},x_0;r}(x)\quad
(x\in \teich_{g,m}).
$$
From Proposition \ref{prop:thm-BT}, the Monge-Amp\`ere mass of $\extrecip_{G,H}$ is finite. Since $g_{\teich_{g,m},x_0;r}$ coincides with the pluricomplex Green function outside a compact set, the boundary measure of $\extrecip_{G,H}$ and the pluriharmonic measure are comparable on the Bers boundary.
\end{proof}

%

\subsection{Behavior of $\extrecip_{F,G}$ and $\extrecip_F$ around $\partial \teich_{g,m}$}
We continue to identify $\Bers{x_0}$ with $\teich_{g,m}$ via the Bers embedding.
Let $\Gamma_0$ be the marked Fuchsian group representing $x_0$
as \S\ref{subsec:Bers-slice}.

\begin{proposition}
\label{prop:F-G-N}
For any $\varphi\in \partial \Bers{x_0}$,
there are $F,G\in \mathcal{MF}$ and a neighborhood $\hat{N}$
of $\varphi$ in $A_{2}(\mathbb{H}^{*},\Gamma_{0})$
such that
\begin{enumerate}
\item[{\rm (1)}]
$F$ is essentially complete
and $F$ and $G$ are transverse; 
\item[{\rm (2)}]
for any $x\in \teich_{g,m}$,
$(\Xi_{x_0}\circ \boldsymbol{\psi}_x)^{-1}(\hat{N})$
is contained in an open set $\mathcal{SMF}_x$
which is relatively compact in $\mathcal{SMF}^{F}_x$;
and
\item[{\rm (3)}]
$\extrecip_{F,G}(x)=\extrecip_F(x)$
for $x\in \hat{N}\cap \Bers{x_0}$.
\end{enumerate}
\end{proposition}

\begin{proof}
Take an essentially complete $F\in \mathcal{MF}$
whose support is realizable in the Kleinian manifold associated to $\varphi$,
that is, ${\rm leng}_{\rho_\varphi}(F)>0$.
Since the length function is continuous on the Bers compactification,
we can take a small neighborhood $\hat{N}$ of $\varphi$ in $A_{2}(\mathbb{H}^{*},\Gamma_{0})$
such that the ending laminations of the Kleinian surface groups in 
the closure of $\hat{N}\cap \Bers{x_0}$ 
does not coincide with $L(F)$
(see also \cite[Lemma 30]{MR1413855}).
One can check that such an $\hat{N}$ satisfies the condition (2).

We claim 
\begin{claim}
\label{claim:1-1}
For any $G'\in\mathcal{MF}$,
the quotient $\ext_{x}(G')/\ext_x(F)$ is bounded on $\hat{N}\cap \Bers{x_0}$.
\end{claim}

\begin{proof}[Proof of Claim \ref{claim:1-1}]
Take a sequence $\{x_n\}_{n}$ in $\hat{N}\cap \Bers{x_0}$
such that
$$
\ext_{x_n}(G')/\ext_{x_n}(F)\to \sup _{x\in \hat{N}\cap \Bers{x_0}}\left(\ext_{x}(G')/\ext_{x}(F)\right).
$$
For our purpose, we may assume that $\{x_n\}_n$ is a divergent sequence in $\Bers{x_0}$.
Let $G_{x_n}\in \mathcal{SMF}_{x_0}$ be the vertical foliation
of the quadratic differential (of unit norm) associated to the Teichm\"uller geodesic connecting
from $x_0$ to $x_n$.
By taking a subsequence,
we may assume that $\{x_n\}_n$ converges to $\varphi_0$ in 
the closure of $\hat{N}\cap \partial \Bers{x_0}$
and
to $[H_\infty]\in \mathcal{PMF}$ in the Thurston compactification,
and $\{G_{x_n}\}_n$ converges to $G_\infty\in \mathcal{SMF}_{x_0}$.
From \cite[Proposition 5.1]{MR2449148},
$i(G_\infty,H_\infty)=0$.

We claim  $i(F,G_\infty)\ne 0$.
Otherwise, $L(F)=L(G_\infty)=L(H_\infty)$ since $F$ is essentially complete
(cf. \cite{MR662738}).
On the other hand,
from \cite[Theorem 5.2]{MR3330544},
the limit $H_\infty$ is disjoint from the parabolic loci of the Kleinian manifold asssociated to $\varphi_0$
and satisfies $i(H_\infty,H')=0$ for any $H'\in \mathcal{MF}$ such that $L(H')$
coincides with the ending lamination of a geometrically infinite end of the Kleinian manifold
associated to $\varphi_0$. In particular, $L(H_\infty)$ is not realizable in the Kleinian manifold associated to $\varphi_0$. This contradicts the realizability of $F$.

Since $i(F,G_\infty)\ne 0$,
there is a constant $\epsilon_0>0$ such that $i(F,G_{x_n})\ge \epsilon_0$ for sufficiently large $n$.
Hence,
$$
\epsilon_{0}^2\le i(F,G_{x_n})^2\le \ext_{x_n}(F)\ext_{x_n}(G_{x_n})
=e^{-2d_T(x_0,x_n)}\ext_{x_n}(F)
$$
and
\begin{align*}
\frac{\ext_{x_n}(G')}{\ext_{x_n}(F)}
&\le
\frac{\ext_{x_n}(G')}{\epsilon_0^2e^{2d_T(x_0,x_n)}}
\le \frac{\ext_{x_0}(G')}{\epsilon_0^2} 
\end{align*}
for sufficiently large $n$.
\end{proof}

Let us complete the proof of Proposition \ref{prop:F-G-N}. Take $G'\in \mathcal{MF}$ which is transverse to $F$.
Let $M=\sup_x \ext_x(G')/\ext_x(F)$ for $x\in \hat{N}\cap \Bers{x_0}$ and set $G=G'/M^{1/2}$.
Then,
$F$ and $G$ are transverse and satisfy
$\extrecip_{F,G}(x)=\extrecip_F(x)$
for
$x\in \hat{N}\cap \Bers{x_0}$.
\end{proof}

\section{Pluriharmonic measure and Thurston measure}
\label{sec:pluriharmonic-measure-Thurston-measure}
%
For $x\in \teich_{g,m}$,
we denote by $\harmonicmeasure^{x_0}_x$
the pluriharmonic measure of $x$
on $\Bers{x_0}$
(cf. \S\ref{subsec:result-from-pluripotential-theory}).
The superscript ``$x_0$" of $\harmonicmeasure^{x_0}_x$
indicates the base point of the Bers slice $\Bers{x_0}$.
Since $\partial \Bers{x_0}$ is a compact metrizable space,
$\harmonicmeasure^{x_0}_x$ and the pushforward measure
$\pushThursMBers_x$ defined in \eqref{eq:defintion-pushforward-Th}
are inner and outer regular
(cf. \cite[Theorem 1.1]{MR1700749}).

The aim of this section is to prove the following.

\begin{theorem}[PH measure and Thurston measure]
\label{thm:support-PH-measure}
For any $x\in \teich_{g,m}$,
the pluriharmonic measure $\harmonicmeasure^{x_0}_x$
is absolutely continuous 
with respect to 
$\pushThursMBers_x$
on $\partial\Bers{x_0}$.
\end{theorem}

\subsection{Cusps are negligible}
We first check the following.

\begin{proposition}[Cusps are negligible]
\label{prop:Bers-APT}
$\harmonicmeasure^{x_0}_x
(\partial^{cusp} \Bers{x_0})=0$.
Namely,
the pluriharmonic measure $\harmonicmeasure^{x_0}_x$ is supported on 
$\partial^{mf}\Bers{x_0}$.
\end{proposition}

\begin{proof}
Let $\gamma\in \pi_1(\Sigma_{g,m})$.
Let $\APT^{x_{0}}_{\gamma}\subset
\partial \Bers{x_0}$
be the boundary groups which admit $\gamma$ as an APT
(possibly, 
$\APT^{x_{0}}_{\gamma}=\emptyset$ for some $\gamma\in \pi_1(\Sigma_{g,m})$).
Since
$\partial^{cusp} \Bers{x_0}=
\cup_{\gamma\in \pi_1(\Sigma_{g,m})}
\APT^{x_{0}}_{\gamma}$,
it suffices to show that 
$\harmonicmeasure^{x_0}_x
(\APT^{x_{0}}_{\gamma})=0$
for each $\gamma\in \pi_1(\Sigma_{g,m})$.

Suppose $\APT^{x_{0}}_{\gamma}\ne \emptyset$.
Consider
a holomorphic function on $\hat{\mathbb{C}}-[0,4]$ defined by
$$
H(w)=\frac{w-2-\sqrt{w^{2}-4w}}{2}
$$
and $H(\infty)=0$.
$H$ maps $\hat{\mathbb{C}}-[0,4]$
conformally onto $\mathbb{D}$
and is continuous on  $\hat{\mathbb{C}}-[0,4)$
with $H(4)=1$.
Set
$$
F_{\gamma}(\varphi)=H({\rm tr}^{2}\rho_{\varphi}(\gamma))
$$
for $\varphi\in \overline{\Bers{x_{0}}}$.
Since every monodromy $\rho_{\varphi}$
is faithful and discrete for $\varphi\in \overline{\Bers{x_{0}}}$,
${\rm tr}^{2}\rho_{\varphi}(\gamma)\in \mathbb{C}-[0,4)$
for all $\varphi\in \overline{\Bers{x_{0}}}$.
Therefore,
$F_{\gamma}$ is holomorphic on $\Bers{x_{0}}$
and continuous on the Bers closure $\overline{\Bers{x_{0}}}$
such that
$F_{\gamma}(\Bers{x_{0}})\subset\mathbb{D}$
and
$F_{\gamma}(\overline{\Bers{x_{0}}})\subset\mathbb{D}\cup\{1\}$.
Furthermore,
for $\varphi\in \overline{\Bers{x_{0}}}$,
$\varphi\in \APT^{x_{0}}_{\gamma}$
if and only if $F_{\gamma}(\varphi)=1$.

By Demailly's Poisson integral formula in \cite[Th\'eor\`eme 5.1]{MR881709},
for all $n\in \mathbb{N}$,
the $n$-th power
$(F_{\gamma})^{n}$
of $F_{\gamma}$ is represented by
\begin{equation}
\label{eq:integral-APT}
(F_{\gamma})^{n}(x)=\int_{\partial\Bers{x_0}}
(F_{\gamma})^{n}(\varphi)\harmonicmeasure^{x_0}_x(\varphi)
\quad (x\in \Bers{x_0}).
\end{equation}
The $n$-th power $(F_{\gamma})^{n}$ converges pointwise
to
the characteristic function $\chi_{\APT^{x_{0}}_{\gamma}}$
of $\APT^{x_{0}}_{\gamma}$
on $\overline{\Bers{x_{0}}}$
as $n\to \infty$.
Since $\harmonicmeasure^{x_0}_x(\partial\Bers{x_0})=1$
and all $(F_{\gamma})^{n}$ is uniformly bounded on the Bers closure,
by Lebesgue's dominated convergence theorem,
from \eqref{eq:integral-APT}
$$
0=\int_{\partial\Bers{x_0}}
\chi_{\APT^{x_{0}}_{\gamma}}(\varphi)
d\harmonicmeasure^{x_0}_x(\varphi)
=\harmonicmeasure^{x_0}_x(\APT^{x_{0}}_{\gamma}),
$$
and we are done.
\end{proof}

\subsection{Local comparison and Proof of Theorem \ref{thm:support-PH-measure}}
\label{subsec:A-covering-partial-Bers}
Let $x\in \teich_{g,m}$.
%
For $H\in \mathcal{SMF}^{mf}_{x}\cong \mathcal{PMF}^{mf}$,
$\varphi_H=\varphi_{H,x_0}\in \partial\Bers{x_0}$
be the totally degenerate group whose ending lamination is equal to $L(H)$.
By Proposition \ref{prop:F-G-N},
there are $F_H$,
$G_H\in \mathcal{MF}$
and a neighborhood $\hat{N}_H$ of $\varphi_H$
such that
$F_H$ is essentially complete,
$F_H$ and $G_H$ are transverse and satisfy
$$
\extrecip_{F_H,G_H}=\extrecip_{F_H}
$$
on $\hat{N}_H\cap \Bers{x}$.
Since $\partial\Bers{x_0}$ is compact,
we can choose a finite system $\{\hat{N}_{H_i}\}_{i}$
which covers $\partial\Bers{x_0}$.
%
For the simplicity,
set 
$F_i=F_{H_i}$,
$G_i=G_{H_i}$,
and $\hat{N}_i=\hat{N}_{H_i}$.

Theorem \ref{thm:support-PH-measure} follows from the following proposition.

\begin{proposition}[Local comparison]
\label{prop:local-comparison}
For each $i$,
the pluriharmonic measure
$\harmonicmeasure^{x_0}_{x}$
is absolutely continuous with respect to
$\pushThursMBers_x$
on $\partial\Bers{x_0}\cap \hat{N}_i$.
\end{proposition}

\begin{proof}
Since $\harmonicmeasure^{x_0}_{x}$ and $\pushThursMBers_x$
are outer regular on $\partial\Bers{x_0}$,
it suffices to show that
\begin{equation*}
\harmonicmeasure^{x_0}_{x}(U\cap \hat{N}_i)
\lesssim \pushThursMBers_x(U\cap \hat{N}_i)
\end{equation*}
for each relative open set $U\subset \partial\Bers{x_0}$,
where the constant for the comparison is independent of $U$.

For $G\in \mathcal{SMF}_x$ which is transverse to $F_i$,
let $R_{G}\colon \mathbb{R}\to \teich_{g,m}\cong \Bers{x_0}$
be the Teichm\"uller ray defined by $-q_{F_i,{\bf T}_{1,x}(G)}$
emanating from ${\bf T}_{1,x}(G)\in {\rm HS}(F_i,1)\subset \mathbb{H}_{F_i}\cong \teich_{g,m}$,
where ${\bf T}_{1, x}\colon \mathcal{SMF}^{F_i}_{x}\to \mathbb{H}_{F_i}$ is defined for $F_i$ and $R=1$ as \eqref{eq:TRHS}.
Let $G^{F_i}=h(q_{F_i,{\bf T}_{1,x}(G)})$ for the simplicity.
Then,
$G^{F_i}$ is projectively equivalent to $G$,
and satisfies $\ext_{{\bf T}_{1,x}(G)}(G^{F_i})=\ext_{{\bf T}_{1,x}(G)}(F_i)=i(G^{F_i},F_i)=1$.

We define a function
$\tau_i\colon \partial^{mf}\Bers{x_0}\cap\hat{N}_i\to \mathbb{R}$
by
$$
\tau_i(\varphi_G)=\inf\{\tau>0\mid 
\mbox{$R_{G'}(t)\in \hat{N}_i$ for $t\ge \tau$,
$G'\in \mathcal{SMF}_x^{mf}$ with $L(G')=L(G)$}\}.
$$
From Proposition \ref{prop:Teichmuller-limit},
$\tau_i(\varphi_G)<\infty$ for any $\varphi_G\in \partial^{mf}\Bers{x_0}\cap\hat{N}_i$.
By the definition of $\tau_i$,
$R_{G'}(t)\in \hat{N}_i$ for all $t>\tau_i(\varphi_G)$ and $G'\in \mathcal{SMF}_x^{mf}$ with $L(G')=L(G)$.

We claim
\begin{claim}
\label{claim:3-3}
$\tau_i$ is upper semicontinuous.
\end{claim}

\begin{proof}[Proof of Claim \ref{claim:3-3}]
Suppose to the contrary that $\tau_i$ is not upper semicontinuous
at $\varphi_G\in \hat{N}_i\cap \partial^{mf}\Bers{x_0}$.
There are $\epsilon>0$
and $\{H_n\}_n\subset \mathcal{SMF}_x^{mf}$
($\cong \mathcal{PMF}^{mf}$)
such that $\tau(\varphi_{H_n})\ge \tau(\varphi_G)+\epsilon$
and $\varphi_{H_n}\to \varphi_G$.
Since the Hausdorff limit (in the space of geodesic laminations)
of any subsequences of $\{L(H_n)\}_n$ 
%
contains $L(G)$,
any accumulation point of $\{H_n\}_n$ in $\mathcal{SMF}_x$
is topologically equivalent to $G$
since $G$ is minimal and filling
(see the discussion in the last second paragraph in \cite[\S1]{MR2258749}).
Therefore,
we may assume that 
there is a sequence $\{t_n\}_n$ in $\mathbb{R}$
such that $t_n\ge \tau(\varphi_G)+\epsilon/2$ such that $R_{H_n}(t_n)\not\in\hat{N}_i$
and $H_n\to G'\in  \mathcal{SMF}_x^{mf}$ with $L(G')=L(G)$.

When $\{t_n\}_n$ is bounded from above,
we may also assume that $t_n\to t_\infty$.
Since ${\bf T}_{1,x}(H_n)\to {\bf T}_{1,x}(G)$,
$R_{H_n}(t_n)\to R_{G'}(t_\infty)$
(cf. \cite[\S1.1, Theorem]{MR0445013}).
Since $t_\infty\ge \tau(\varphi_G)+\epsilon/2$,
$R_{G'}(t_\infty)\in \hat{N}_i$.
This is a contradiction.

Suppose $t_\infty\to \infty$.
We may assume that $\{R_{H_n}(t_n)\}_n$
converges to $\varphi_\infty$ in $\overline{\Bers{x_0}}$.
By the same discussion as the proof of Proposition \ref{prop:Teichmuller-limit},
we have
$$
{\rm leng}_{\rho_{\varphi_n}}(H_n)\lesssim
e^{-t_n}\ext_{{\bf T}_{1,x}(H_n)}(H_n)^{1/2}
=e^{-t_n}i((H_n)^{F_i},F_i)=e^{-t_n},
$$
where $\varphi_n\in \Bers{x_0}$ is the corresponding point
to $R_{H_n}(t_n)$.
From the continuity of the length function,
we have ${\rm leng}_{\rho_{\varphi_\infty}}(G')=0$.
Therefore we obtain $\varphi_\infty=\varphi_{G'}=\varphi_G$.
This is also a contradiction since $\hat{N}_i$ is a neighborhood of
$\varphi_G$.
\end{proof}

Let $U^{mf}=U\cap \partial^{mf}\Bers{x_0}$.
For $s>0$,
we define
$$
U^{mf}_s=\{\varphi_G\in \hat{N}_i\cap \partial^{mf}\Bers{x_0}\mid
\tau_i(\varphi_G)<s/2\}.
$$
Then,
$\{U^{mf}_s\}_s$ and $U^{mf}$ satisfies
\begin{enumerate}
\item
$U^{mf}_s\subset U^{mf}_{s'}$ for $s<s'$;
\item
$U^{mf}\cap \hat{N}_i=\cup_{s>0}U^{mf}_s$;
\item
$U^{mf}_s$ is open in $\partial^{mf}\Bers{x_0}$ in the sense that for any $\varphi_G\in U^{mf}_s$, there is an open neighborhood $V$ of $\varphi_G$ with $V\cap\partial^{mf}\Bers{x_0} \subset U^{mf}_s$; and
\item
$\mathcal{U}^{mf}_s=(\Xi_{x_0}\circ \proje_x)^{-1}(U^{mf}_s)$ is open in $\mathcal{SMF}^{mf}_x$ in the sense that any $G\in \mathcal{U}^{mf}_s$ admits a neighborhood $V'$ with $V'\cap \mathcal{SMF}^{mf}_x\subset \mathcal{U}^{mf}_s$.
\end{enumerate}
The third condition follows from
Claim \ref{claim:3-3},
and the fourth condition is deduced from
the continuity of the mapping $\Xi_{x_0}\circ \proje_x$ (cf. \S\ref{subsec:BoundarygroupswithoutAPT} and \ref{subsec:Thurstonmeasure}).
From (2) of Proposition \ref{prop:F-G-N}, each $\mathcal{U}^{mf}_s$ is contained in an open set in $\mathcal{SMF}_x$ which is relatively compact in $\mathcal{SMF}^{F_i}_x$.
Such an open set is defined from $(\Xi_{x_0}\circ \boldsymbol{\psi}_x)^{-1}(\hat{N}_i)$, and taken independently of $U$.

Take $s_0>0$ satisfying $U^{mf}_s\ne \emptyset$ for $s>s_0$.
Fix $s>s_0$,
we define subsets $\mathcal{R}$
and $\mathcal{R}_s$ in $\Bers{x_0}\cong \teich_{g,m}$ by
\begin{align*}
\mathcal{R}
&=\{R_{G'}(t)\in \Bers{x_0}\mid \mbox{$G'\in \mathcal{SMF}^{mf}_x\cap \mathcal{SMF}_x^{F_i}$, $t>0$}\},
\\
\mathcal{R}_s
&=\{R_{G'}(t)\in \Bers{x_0}\mid G'\in \mathcal{U}^{mf}_s, t>0\}.
\end{align*}
From the above discussion,
$\mathcal{R}_s$ satisfies an open condition in the sense that
any $z\in \mathcal{R}_s$ admits a neighborhood $V_z$ in $\Bers{x_0}$
with $V_z\cap \mathcal{R}\subset \mathcal{R}_s$.
Next we claim
\begin{claim}
\label{claim:3-4}
For $\varphi_G\in U^{mf}_s\cap \hat{N}_i$,
there is a neighborhood $V_{\varphi_G}$ in $A_{2}(\mathbb{H}^{*},\Gamma_{0})$
such that
$V_{\varphi_G}\cap \partial^{mf}\Bers{x_0}\subset U^{mf}_s\cap \hat{N}_i$
and $V_{\varphi_G}\cap \mathcal{R}\subset \mathcal{R}_s$.
\end{claim}

\begin{proof}[Proof of Claim \ref{claim:3-4}]
Since $U^{mf}_s$ is open,
we can take a neighborhood of $\varphi_G$ with the first condition.
We need to show the existence of a neighborhood of $\varphi_G$ with the second condition.

Otherwise,
there is a sequence $\{z_n\}_n\subset \mathcal{R}\cap \hat{N}_i$
with $z_n\notin \mathcal{R}_s$
and $z_n\to \varphi_G$ in $\overline{\Bers{x_0}}$.
Take $t_n>0$ and $H_n\in \mathcal{SMF}^{mf}_x$
with $R_{H_n}(t_n)=z_n$.
By taking a subsequence,
we may assume that $H_n$ converges to some $G'\in \mathcal{SMF}_x$
as $n\to \infty$.
Then,
$$
{\rm leng}_{\rho_{\varphi_n}}((H_n)^{F_i})
\lesssim \ext_{z_n}((H_n)^{F_i})^{1/2}=e^{-t_n},
$$
where 
$\varphi_n\in \Bers{x_0}$ is the corresponding point
to $z_n\in \teich_{g,m}$. 
Since $z_n\to \varphi_G$ and $G\in \mathcal{SMF}_x^{mf}$,
from the continuity of the length function,
$G'$ is topologically equivalent to $G$.
Therefore,
$\varphi_{H_n}\to \varphi_{G'}=\varphi_{G}$.
This implies that
$\varphi_{H_n}\in U^{mf}_s\cap \hat{N}_i$
and $H_n\in \mathcal{U}^{mf}_s$ for sufficiently large $n$.
This is a contradiction.
\end{proof}

Let us proceed the proof of Proposition \ref{prop:local-comparison}.
We define an open set in the ambient space $A_{2}(\mathbb{H}^{*},\Gamma_{0})$
by
$$
\mathcal{V}_s=
\left(\cup_{z\in \mathcal{R}_s}V_z
\right)\cup 
\left(
\cup_{\varphi_G\in U^{mf}_s}V_{\varphi_G}
\right).
$$
From the definition,
\begin{align*}
\mathcal{V}_s\cap \partial^{mf}\Bers{x_0}
&=U^{mf}_s \\
\mathcal{V}_s\cap \mathcal{R}
&=\mathcal{R}_s.
\end{align*}
Since $\mathcal{SMF}^{mf}_x$ is a subset of full-measure on $\mathcal{SMF}_x$,
from Proposition \ref{prop:comparison},
\begin{align}
\HSmeas_{F,e^{2u}}(\mathcal{V}_s\cap {\rm HS}(F_i,e^{2u}))
&=\HSmeas_{F,e^{2u}}({\bf T}_{e^{2u},x}(\mathcal{U}^{mf}_s))
\label{eq:comparison-on-Bi1}
\\
&\asymp 
\PThursM^{x}(\mathcal{U}^{mf}_s)
=
\pushThursMBers_x(U^{mf}_s)
\nonumber
\\
&\le \pushThursMBers_x(U\cap \hat{N}_i)
\nonumber
\end{align}
for all $u>s$,
where the constant for the comparison is independent of $s>s_0$
and $U$.
From the definition of the function $\tau_i$, 
$$
(\mathcal{V}_s\cap  \mathcal{R})\cap {\rm HS}(F_i,e^{2u})\subset \hat{N}_i
$$
for $u>s/2$.
Since $\extrecip_{F_i,G_i}=\extrecip_{F_i}$ on $\hat{N}_i$ (Proposition \ref{prop:F-G-N}) and $\mu_{i,u}=\mu_{\extrecip_{F_i,G_i},-e^{-2u}}$ is supported on the level set $S_{\extrecip_{F_i,G_i}}(-e^{-2u})$ of $\extrecip_{F_i,G_i}$ (cf. \S\ref{subsubsec:PSH exhaustions}),
\begin{equation}
\label{eq:comparison-on-Bi2}
\mu_{i,u}(\mathcal{V}_s)
=
\HSmeas_{F,e^{2u}}(\mathcal{V}_s\cap {\rm HS}(F_i,e^{2u}))
\end{equation}
when $u$ is sufficiently large.
From 
\cite[Th\'eor\`eme et D\'efinition 3.1]{MR881709}
and
 Proposition \ref{prop:ddcuFG-volume-finite},
$\mu_{i,u}$ converges weakly 
to the boundary measure $\mu_i$ of $\extrecip_{F_i,G_i}$.
Hence,
from \eqref{eq:comparison-on-Bi1}
and \eqref{eq:comparison-on-Bi2}
we conclude
$$
\mu_i(U^{mf}_s)
=
\mu_{i}(\mathcal{V}_s)\le \liminf_{u\to \infty}\mu_{i,u}(\mathcal{V}_s)
\lesssim
\pushThursMBers_x(U\cap \hat{N}_i)
$$
since $\mathcal{V}_s$ is an open set in the ambient space (cf. \cite[(iv) of Theorem 2.1]{MR1700749}).
Since $\{U^{mf}_s\}_s$ is an increasing sequence of measurable sets and $\cup_{s>0}U^{mf}_s=U^{mf}\cap \hat{N_i}$, from Proposition \ref{prop:Bers-APT}, from Proposition \ref{prop:ddcuFG-volume-finite}, we deduce
\begin{align*}
\harmonicmeasure^{x_0}_{x}(U\cap \hat{N}_i)
&=\harmonicmeasure^{x_0}_{x}(U^{mf}\cap \hat{N}_i)
\\
&\asymp \mu_i(U^{mf}\cap \hat{N}_i)
=
\lim_{s\to \infty}\mu_i(U^{mf}_s)
\lesssim
\pushThursMBers_x(U\cap \hat{N}_i),
\end{align*}
where the constants for the comparisons are independent of $U$.
\end{proof}


\subsection{Corollary of Theorem \ref{thm:support-PH-measure}}
The pushforward measure
$\pushThursMBers_x$ is supported on $\partial^{ue}\teich_{g,m}$ (cf. \S\ref{subsec:pushforwardmeasure-Bers-boundary}).
The Thurston measure $\PThursM^x$ on $\mathcal{SMF}^{mf}_x$
is defined from the Euclidean measure on the train track coordinates.
Hence,
we can see that $\pushThursMBers_x$ has no atom
on $\partial\Bers{x_0}$
since the inverse image $(\Xi_{x_0}\circ \proje_{x_0})^{-1}(\varphi_G)$
for $\varphi_G\in \partial^{mf}\Bers{x_0}$
is a proper (linear) subspace in any train track coordinates
around $G$.
Thus, from Theorem \ref{thm:support-PH-measure}, we deduce

\begin{corollary}
\label{coro:support-PH-measure}
For any $x\in \teich_{g,m}$,
the pluriharmonic measure $\harmonicmeasure^{x_0}_{x}$
is supported on $\partial^{ue}\Bers{x_0}$
and has no atom on $\partial\Bers{x_0}$.
\end{corollary}


\section{Pluriharmonic Poisson kernel}
\label{sec:Poisson-kernel}
The aim of this section
is to determine the Poisson kernel for Teichm\"uller space.
%

\begin{theorem}[Poisson kernel]
\label{thm:Poisson-kernel}
The function \eqref{eq:Poisson-kernel}
is the Poisson kernel.
Namely,
for $x,y\in \teich_{g,m}$,
$$
d\harmonicmeasure^{x_0}_y=\mathbb{P}(x,y,\,\cdot\,)d\harmonicmeasure^{x_0}_x
$$
on $\partial\Bers{x_0}$.
\end{theorem}

\begin{proof}
Since the function \eqref{eq:Poisson-kernel}
is reciprocal in the sense that 
$\mathbb{P}(x,y,\varphi)=\mathbb{P}(y,x,\varphi)^{-1}$
for $x,y\in \teich_{g,m}$ and $\varphi\in \partial \Bers{x_0}$,
from Demailly's theorem (Proposition \ref{prop:Demally1})
and Corollary \ref{coro:support-PH-measure},
the assertion of the theorem follows from
\begin{align}
\lim_{\teich_{g,m}\ni z\to \varphi}
\frac{g_{\teich_{g,m}}(y,z)}{g_{\teich_{g,m}}(x,z)}
&=
\frac{\ext_{x}(F_{\varphi})}{\ext_{y}(F_{\varphi})},
\label{eq:limit-Poisson-kernel}
\\
\limsup_{\teich_{g,m}\ni z\to \varphi'}
\frac{g_{\teich_{g,m}}(y,z)}{g_{\teich_{g,m}}(x,z)}
&\le e^{2d_T(x,y)}
\label{eq:limit-Poisson-kernel2}
\end{align}
for $x,y\in \teich_{g,m}$,
$\varphi\in \partial^{ue}\Bers{x_0}$
and $\varphi'\in \partial\Bers{x_0}$,
where $F_\varphi\in \mathcal{MF}$ is defined as \S\ref{subsubsec:results}.
Indeed,
\eqref{eq:limit-Poisson-kernel} and \eqref{eq:limit-Poisson-kernel2} implies that the left-hand side of \eqref{eq:limit-Poisson-kernel2} is measurable and integrable on $\partial \Bers{x_0}$ with respect to the harmonic measure $\harmonicmeasure^{x_0}_x$ ($x\in \teich_{g,m}$) and coincides with our function $\mathbb{P}(x,y,\cdot )$ a.e. on $\partial \Bers{x_0}$ from Corollary \ref{coro:support-PH-measure}.

\eqref{eq:limit-Poisson-kernel2} follows from \eqref{eq:Green-function} and
$$
\frac{g_{\teich_{g,m}}(x,z)}{g_{\teich_{g,m}}(y,z)}
=e^{-2(d_{T}(x,z)-d_{T}(y,z))}(1+o(1))
\le e^{2d_{T}(x,y)}(1+o(1))
$$
as $z\to \varphi'\in \partial\Bers{x_0}$.
We show \eqref{eq:limit-Poisson-kernel}.
%
We claim

\begin{claim}
\label{claim:3-1}
Let $\{x_{n}\}_{n}\subset \teich_{g,m}\cong \Bers{x_0}$ be a sequence converging to $\varphi\in \partial^{ue}\Bers{x_0}$.
Then,
$\{x_{n}\}_{n}$ converges
to the projective class $[F_{\varphi}]$
in the Gardinar-Masur compactification.
\end{claim}

\begin{proof}[Proof of Claim \ref{claim:3-1}]
This claim follows by applying the discussion in \cite[\S3]{MR3336619}.
We give a proof for confirmation.

Take $\alpha_{n}\in \mathcal{S}$ with $\ext_{x_{n}}(\alpha_{n})\le M$
for some constant $M$ depending only on $(g,m)$
(cf. \cite[Theorem 1]{MR780038}).
By taking a subsequence, we may assume that
$t_{n}\alpha_{n}\to F\in \mathcal{MF}-\{0\}$
with some $t_n>0$.
Since $x_{n}$ converges to a totally degenerate group without APT,
$\ext_{x_0}(\alpha_{n})\to \infty$ and hence 
$t_{n}\to 0$
(cf. \cite{MR0396937}).
By 
the Bers inequality
and \eqref{eq:comparison-hyp-ext},
the hyperbolic length of the geodesic representation of $t_{n}\alpha_{n}$
in the quasifuchsian manifold associated with $x_{n}$ tends to $0$.
From the continuity of the length function,
any sublamination of the support of $F$
is non-realizable in the Kleinian manifold
associated with $\varphi$.
Hence,
the support of $F$ is contained in $F_{\varphi}$ and $i(F,F_{\varphi})=0$.
Thus we have $[F]=[F_{\varphi}]$
in $\mathcal{PMF}$
since $F_{\varphi}$ is uniquely ergodic.

By taking a subsequence if necessary, we may assume that $x_{n}\to \mathfrak{p}\in \partial_{GM}\teich_{g,m}$.
From \eqref{eq:ex-geo2}, we obtain
\begin{align*}
i_{x_0}(\mathfrak{p},[F_{\varphi}])
&=i_{x_0}(\mathfrak{p},[F])=
\lim_{n\to \infty}
i_{x_0}(x_{n},[t_{n}\alpha_{n}]) \\
&=
\lim_{n\to \infty}e^{-d_{T}(x_0,x_{n})}
\frac{\ext_{x_{n}}(t_{n}\alpha_{n})^{1/2}}
{\ext_{x_0}(t_{n}\alpha_{n})^{1/2}} \\
&\le \lim_{n\to \infty}
\frac{M^{1/2}t_{n}e^{-d_{T}(x_0,x_{n})}}{{\ext_{x_0}(t_{n}\alpha_{n})^{1/2}}}
=0.
\end{align*}
From the characterization of uniquely ergodic measured foliations in the Gardiner-Masur
compactification,
we conclude that $\mathfrak{p}=[F_{\varphi}]$
(cf. \cite[Theorem 3]{MR3009545}).
\end{proof}

Let us finish the proof.
From
\eqref{eq:ex-geo1},
\eqref{eq:ex-geo2}
and
the Krushkal formula \eqref{eq:Green-function},
\begin{align*}
\frac{g_{\teich_{g,m}}(y,z)}{g_{\teich_{g,m}}(x,z)}
&=\exp(2(d_T(x,z)-d_T(y,z)))(1+o(1)) \\
&=\exp(2d_T(y,x))
\exp(-4\gromov{x}{z}{y})(1+o(1)) 
\\
&\to\exp(2d_T(y,x))
 \left(
 \exp(-d_{T}(y,x))
 \frac{\ext_{x}(F_\varphi)^{1/2}}{\ext_{y}(F_\varphi)^{1/2}}\right)^2
 \\
& =
\frac{\ext_{x}(F_\varphi)}{\ext_{y}(F_\varphi)}
\end{align*}
as $z\to \varphi\in \partial^{ue}\Bers{x_0}$.
This  implies \eqref{eq:limit-Poisson-kernel}.
\end{proof}

\begin{remark}
\label{remark:Levi-Poisson}
The Poisson kernel $\mathbb{P}(x,y,\,\cdot\,)$ is
not pluriharmonic in the variable $y$
when $\convgenus\ge 2$.
Indeed,
when $F\in \mathcal{MF}$ is uniquely ergodic and essentially complete,
for
$y=(M_1,f_1)\in \teich_{g,m}$ and
$v\in T_y\teich_{g,m}$,
\begin{align*}
\Levi{\mathbb{P}(x,\,\cdot\,,F)}{v}{\overline{v}}
&=
-\convgenus\frac{\ext_{x}(F)^{\convgenus}}{\ext_{y}(F)^{\convgenus+1}}
\left(
2\int_{M_1}\frac{|\eta_v|^2}{|q_{F,y}|}
-\frac{\convgenus+1}{\ext_{y}(F)}
\left|
\langle v,q_{F,y}\rangle
\right|^2
\right),
\end{align*}
where $\mathcal{L}$ stands for the Levi form
and $\eta_v\in \mathcal{Q}_y$ is the $q_{F,y}$-realization of $v$ (cf. \eqref{eq:q0-realization} and \cite[Theorem 5.1]{MR3715450}).
When $v$ is represented by the infiniteismal Beltrami differential $\overline{q_{F,y}}/|q_{F,y}|$,
$\eta_v=q_{F,y}$.
Hence,
the Levi form of $\mathbb{P}(x,\,\cdot\,,F)$ at $y$ is positive in the direction $v$.
However,
when $v$ satisfies $\langle v,q_{F,y}\rangle=0$,
the Levi form at $y$ is negative in this direction $v$.

On the other hand,
the Poisson kernel $\mathbb{P}(x,y,\,\cdot\,)$
is plurisubharmonic in the variable $x$
(cf. \cite[Corollary 1.1]{MR3715450}).
\end{remark}

\section{The Green formula}
\label{sec:intrinsic-representation}
The aim of this section is to complete the proof of 
the Poisson integral formula (Theorem \ref{thm:Poisson-integral-formula}).
Indeed,
Theorem \ref{thm:Poisson-integral-formula}
is derived from the following theorem.

\begin{theorem}[Green formula]
\label{thm:Poisson-integral-formula-sub}
Let $V$ be a continuous function on the Bers compactification
$\overline{\Bers{x_0}}$
which is plurisubharmonic on $\Bers{x_0}\cong \teich_{g,m}$.
Then
\begin{align*}
V(x)
&=\int_{\partial \Bers{x_0}}V(\varphi)
\mathbb{P}(x_0,x,\varphi)
d\pushThursMBers_{x_0}(\varphi)
-
\int_{\Omega}dd^cV\wedge |g_{x}|(dd^cg_{x})^{\convgenus-1},
\end{align*}
where 
$g_{x}(y)=(2\pi)^{-1}\log\tanh d_T(x,y)$.
Furthermore,
when $\convgenus\ge 2$,
$$
V(x)=\int_{\partial \Bers{x_0}}V(\varphi)
\mathbb{P}(x_0,x,\varphi)
d\pushThursMBers_{x_0}(\varphi)
-
\int_{\Omega}dd^cV\wedge (dd^cg_{x})^{\convgenus-2}\wedge
dg_x\wedge d^cg_x.
$$
\end{theorem}
From the definitions of the function $\mathbb{P}$ and the probability measure $\pushThursMBers_{x_0}$,
the first terms of the above Green formulas
are dealt with from Thurston theory and Extremal length geometry.
It is also possible to discuss the second terms
from the topological aspect in Teichm\"uller theory.
Indeed,
the Levi form of the pluricomplex Green function has a topological interpretation
in terms of the Thurston symplectic form on $\mathcal{MF}$
via Dumas' K\"ahler (symplectic) structure
on the space of holomorphic quadratic differentials
(cf.  \cite{miyachi-pluripotentialtheory1}.
See also \cite[Theorem 5.8]{MR3413977}).

Theorem \ref{thm:Poisson-integral-formula-sub}
follows from
the following theorem,
Theorem \ref{thm:Poisson-kernel},
and the Jensen-Lelong formula \eqref{eq:integral-formula}
(cf. \cite[Th\'eor\`eme 5.1]{MR881709}).

\begin{theorem}[PH measure is Thurston measure]
\label{thm:PHmeasure-is-Thurston-measure}
For any $x\in \teich_{g,m}$,
$$
\harmonicmeasure^{x_0}_{x}=\pushThursMBers_x
$$
on $\partial\Bers{x_0}$.
\end{theorem}

\subsection{Measures and the action of $\mcg_{g,m}$}
Since the action of $\mcg_{g,m}$
extends continuously to $\Bers{x_0}\cup \partial^{mf}\Bers{x_0}$,
the pushforward measure $[\omega]_*\harmonicmeasure^{x_0}_{x}$
is well-defined for $[\omega]\in \mcg_{g,m}$ and $x\in \teich_{g,m}$
from Corollary \ref{coro:support-PH-measure}.
We first check the following
(see the discussion after \cite[D\'efinition 5.2]{MR881709} and \cite[(5.8) in Th\'eor\`eme 5.4]{MR881709}).

\begin{lemma}[$\mcg_{g,m}$ and PH measure]
\label{lem:mcg-PH}
For $[\omega]\in \mcg_{g,m}$ and $x\in \teich_{g,m}$
\begin{equation*}
[\omega]_*\harmonicmeasure^{x_0}_{x}=
\harmonicmeasure^{x_0}_{[\omega](x)}
\end{equation*}
on $\partial\Bers{x_0}$.
\end{lemma}

\begin{proof}
We need to show that for any bounded continuous function
$f$ on 
$A_{2}(\mathbb{H}^{*},\Gamma_{0})$,
$x\in \teich_{g,m}$ and $[\omega]\in \mcg_{g,m}$,
\begin{equation*}
\int_{\partial \Bers{x_0}}f\circ [\omega]d\harmonicmeasure^{x_0}_{x}
=\int_{\partial\Bers{x_0}}f\,d\harmonicmeasure^{x_0}_{[\omega](x)}
\end{equation*}
(cf. \cite[Theorem 1.2]{MR1700749}).
We may assume that $f$ is a strictly PSH function of class $C^2$
on a neighborhood of the Bers compactification
(see the proof of \cite[Th\'eor\`eme and D\'efinition 3.1]{MR881709}).
For simplicity,
we set
$g_{x}(y)=(2\pi)^{-1}g_{\teich_{g,m}}(x,y)$
as Theorem \ref{thm:Poisson-integral-formula-sub}.

Since $g_{[\omega](x)}([\omega](y))=g_{x}(y)$,
from the Lelong-Jensen formula \eqref{eq:Demailly-LJ-formula},
\begin{equation}
\int_{S_{g_{[\omega](x)}}(r)}fd\mu_{g_{[\omega](x)},r}
=\int_{S_{g_{x}}(r)}f\circ [\omega]d\mu_{g_{x},r}
\label{eq:base-change-1}
\end{equation}
for $r<0$.
We define a function $f^*$ on $\overline{\Bers{x_0}}$ by
$$
f^*(\varphi)=\lim_{\delta\to 0}
\sup
\{f\circ [\omega](\varphi')\mid \|\varphi'-\varphi\|_\infty<\delta, \varphi'\in \Bers{x_0}\},
$$
where $\|\cdot \|_\infty$ is the hyperbolic supremum norm
on $A_{2}(\mathbb{H}^{*},\Gamma_{0})$.
Then $f^*$ is bounded and upper semicontinuous on $\overline{\Bers{x_0}}$
and satisfies $f^*=f\circ [\omega]$
on $\Bers{x_0}\cup \partial^{mf}\Bers{x_0}$ by virtue of the continuity of $[\omega]$ on $\Bers{x_0}\cup \partial^{mf}\Bers{x_0}$.
Since $\mu_{g_{x},r}$ converges to $\harmonicmeasure^{x_0}_x$
weakly as $r\to 0$ on $A_{2}(\mathbb{H}^{*},\Gamma_{0})$,
from \eqref{eq:base-change-1} and Proposition \ref{prop:Bers-APT},
\begin{align*}
\int_{\partial\Bers{x_0}}f\,d\harmonicmeasure^{x_0}_{[\omega](x)}
&=\limsup_{r\to 0}\int_{S_{g_{[\omega](x)}}(r)}fd\mu_{g_{[\omega](x)},r}
\\
&=\limsup_{r\to 0}\int_{S_{g_{x}}(r)}f\circ [\omega]d\mu_{g_{x},r}
\\
&=\limsup_{r\to 0}\int_{S_{g_{x}}(r)}f^*d\mu_{g_{x},r}
\\
&\le \int_{\partial \Bers{x_0}}f^*d\harmonicmeasure^{x_0}_{x}
=\int_{\partial \Bers{x_0}}f\circ [\omega]\,d\harmonicmeasure^{x_0}_{x}
\end{align*}
(cf. \cite[Theorem 2.1, Problem 2.6 in Chapter 1]{MR1700749}).
Applying the similar argument 
to a bounded lower semicontinuous function
$$
f_*(\varphi)=\lim_{\delta\to 0}
\inf
\{f\circ [\omega](\varphi')\mid \|\varphi'-\varphi\|_\infty<\delta, \varphi'\in \Bers{x_0}\}
$$
on $\overline{\Bers{x_0}}$,
we obtain the reverse inequality.
%
%
%
%
\end{proof}

Next,
we show the following.
\begin{lemma}
\label{lem:Th-points}
For $x\in \teich_{g,m}$ and $[\omega]\in \mcg_{g,m}$,
$$
d\pushThursMBers_{[\omega](x)}=\mathbb{P}(x,[\omega](x),\,\cdot\,)d\pushThursMBers_x
\ \mbox{and} \ \
[\omega]_*(\pushThursMBers_x)=\pushThursMBers_{[\omega](x)}.
$$
\end{lemma}

\begin{proof}
Since ${\rm Vol}_{Th}([\omega](x))={\rm Vol}_{Th}(x)$,
for any bounded continuous function $f$ on
$A_{2}(\mathbb{H}^{*},\Gamma_{0})$, 
\begin{align*}
\int_{\partial\Bers{x_0}}f\,d\pushThursMBers_{[\omega](x)}
&=
\int_{\mathcal{SMF}^{mf}_{[\omega](x)}}
f\circ (\Xi_{x_0}\circ \proje_{[\omega](x)})d
\PThursM^{[\omega](x)}
\\
&=\int_{\mathcal{SMF}^{mf}_x}f\circ (\Xi_{x_0}\circ \proje_{[\omega](x)}
\circ \proje_{x,[\omega](x)})d
\left((\proje_{x,[\omega](x)}^{-1})_*
\PThursM^{[\omega](x)}
\right)
\\
&=\int_{\mathcal{SMF}^{mf}_x}
f\circ (\Xi_{x_0}\circ \proje_x(G))\frac{1}{\ext_{[\omega](x)}(G)^{\convgenus}}
d\PThursM^x(G)\\
&=
\int_{\partial\Bers{x_0}}f(\varphi)\mathbb{P}(x,[\omega](x),\varphi)\,
d\pushThursMBers_x(\varphi)
\end{align*}
from \eqref{eq:push-forward-marking_change}, where $\proje_x$ and $\proje_{x,[\omega](x)}$ are homeomorphisms defined in \S\ref{subsec:Thurstonmeasure}. This implies the first equation.

Let us prove the second equation.
Any element $[\omega]\in \mcg_{g,m}$
induces a homeomorphism
$$
[\omega]\colon \mathcal{SMF}_x
\ni
G
\to [\omega] (G)
\in \mathcal{SMF}_{[\omega](x)}.
$$
Since the Thurston measure $\ThursM$ is an invariant measure on $\mathcal{MF}$
with respect to the action of $\mcg_{g,m}$,
for a measurable set $E\subset \mathcal{SMF}_{[\omega](x)}$,
\begin{align*}
[\omega]_*\PThursM^x(E) &=\PThursM^x([\omega]^{-1}(E)) \\
&=
\ThursM(\{t G\mid G\in [\omega]^{-1}(E),\ 0\le t\le 1\})/{\rm Vol}_{Th}(x)
\\
&=
\ThursM([\omega]^{-1}(\{t G\mid G\in E,\ 0\le t\le 1\}))/{\rm Vol}_{Th}(x)
\\
&=
\ThursM(\{t G\mid G\in E,\ 0\le t\le 1\})/{\rm Vol}_{Th}([\omega](x))
=\PThursM^{[\omega](x)}(E).
\end{align*}
Therefore we obtain
\begin{align*}
\int_{\partial\Bers{x_0}}f\circ [\omega]\,d\pushThursMBers_x
&=
\int_{\mathcal{SMF}_x}
f\circ [\omega]\circ (\Xi_{x_0}\circ \proje_x)d
\PThursM^x
\\
&=
\int_{\mathcal{SMF}_x}
f\circ (\Xi_{x_0}\circ \proje_{[\omega](x)}\circ [\omega])d
\PThursM^x
\\
&=
\int_{\mathcal{SMF}_{[\omega](x)}}
f\circ (\Xi_{x_0}\circ \proje_{[\omega](x)})
\left(
d([\omega]_*\PThursM^x)
\right)
\\
&=
\int_{\mathcal{SMF}_{[\omega](x)}}
f\circ (\Xi_{x_0}\circ \proje_{[\omega](x)})
d\PThursM^{[\omega](x)}
=\int_{\partial\Bers{x_0}}f\,d\pushThursMBers_{[\omega](x)},
\end{align*}
which implies what we wanted.
%
\end{proof}

\subsection{Proof of Theorem \ref{thm:PHmeasure-is-Thurston-measure}}
Let $x\in \teich_{g,m}$.
From Theorem \ref{thm:support-PH-measure},
there is an integrable function $\Lambda_x$ on $\partial\Bers{x_0}$
such that
$$
d\harmonicmeasure^{x_0}_{x}
=\Lambda_xd\pushThursMBers_{x}
$$
on $\partial\Bers{x_0}$.
For $[\omega]\in \mcg_{g,m}$,
from Theorem \ref{thm:Poisson-kernel},
Lemmas \ref{lem:mcg-PH} and \ref{lem:Th-points},
\begin{align*}
\mathbb{P}(x,[\omega](x),\,\cdot\,)
\Lambda_x d\pushThursMBers_{x}.
&=\mathbb{P}(x,[\omega](x),\,\cdot\,)d\harmonicmeasure^{x_0}_x
=d\harmonicmeasure^{x_0}_{[\omega](x)}
=[\omega]_*d\harmonicmeasure^{x_0}_{x}
\\
&=[\omega]_*(\Lambda_xd\pushThursMBers_{x})
=\Lambda_x\circ[\omega]^{-1}
d\pushThursMBers_{[\omega](x)}
\\
&
=(\Lambda_x\circ[\omega]^{-1})\mathbb{P}(x,[\omega](x),\,\cdot\,)
d\pushThursMBers_{x}.
\end{align*}
Therefore,
we obtain $\Lambda_x\circ[\omega]^{-1}=\Lambda_x$
a.e. on $\partial\Bers{x_0}$ with respect to $\pushThursMBers_{x}$.
Hence, the pullback $\Lambda_x\circ \Xi_x$
is an invariant integrable function on $\mathcal{PMF}$ under the action of $\mcg_{g,m}$.
Since the action of $\mcg_{g,m}$ on $\mathcal{PMF}$
is ergodic with respect to the measure class of $(\proje_{x_0})_*(\PThursM^{x_0})$
(cf. \cite[Corollary 2]{MR644018}),
$\Lambda_x\circ \Xi_x$ is a constant function,
and so is $\Lambda_x$ as a measurable function on $\partial\Bers{x_0}$.
Since both measures $\harmonicmeasure^{x_0}_x$ and
$\pushThursMBers_{x}$
are probability measures on $\partial\Bers{x_0}$,
$\Lambda_x\equiv 1$ a.e. on $\partial\Bers{x_0}$.
\qed
%
%
%
%

\section{Boundary behavior of Poisson integral}
\label{sec:boundary-behaviorPI}
The purpose of this section to prove Theorem \ref{thm:boundary-behavior}.
%

\subsection{Two lemmas}
As \S\ref{subsec:A-covering-partial-Bers},
for $H\in \mathcal{SMF}^{mf}_{x_0}$,
we denote by $\varphi_H\in \partial\Bers{x_0}$
the totally degenerate group whose ending lamination is equal to $L(H)$.

\begin{lemma}
\label{lem:neighborhood-H}
Let $\varphi_H\in \partial^{mf}\Bers{x_0}$.
For $\delta>0$,
we define
$$
N(\varphi_H;\delta)=\{\varphi_G\in \partial^{mf}\Bers{x_0}\mid
i(H,G)<\delta\}.
$$
Then,
$N(\varphi_H;\delta)$ is an open neighborhood of $\varphi_H$ in $\partial^{mf}\Bers{x_0}$
and satisfies
$$
\bigcap_{\delta>0}N(\varphi_H;\delta)=\{\varphi_H\}.
$$
\end{lemma}

\begin{proof}
The mapping $\Xi_{x_0}\colon \mathcal{PMF}^{mf}\to \partial^{mf}\Bers{x_0}$
is factorized as the composition
of the homeomorphism from the Gromov-boundary of the complex of curves
to $\partial^{mf}\Bers{x_0}$
and
the measure-forgetting mapping from $\mathcal{PMF}^{mf}$ to the Gromov boundary
of the complex of curves
(cf. \cite[Theorem 6.6]{MR2582104}).
The measure-forgetting mapping
is the quotient mapping
(cf. \cite{MR2258749} and \cite{Klaereich-boundary}).
Since the intersection number function is continuous,
$(\Xi_{x_0}\circ \proje_{x_0})^{-1}(N(\varphi_H;\delta))=\{G\in \mathcal{SMF}^{mf}_{x_0}\mid i(H,G)<\delta\}$
is open in $\mathcal{SMF}^{mf}_{x_0}$.
Hence, $N(\varphi_H;\delta)$ is an open neighborhood of $\varphi_H$
in $\partial^{mf}\Bers{x_0}$.

Let $\varphi_G\in \bigcap_{\delta>0}N(\varphi_H;\delta)\subset\partial^{mf}\Bers{x_0}$.
Since $i(G,H)<\delta$ for all $\delta>0$, $i(G,H)=0$,
and hence $L(G)=L(H)$ since $H$ is minimal and filling.
Therefore $\varphi_G=\varphi_H$.
\end{proof}

\begin{lemma}
\label{lem:Poisson-kernel}
Let $\delta>0$ and $\varphi_H\in \partial^{ue}\Bers{x_0}$,
there is a neigborhood $U$ of $\varphi_H$ in $A_{2}(\mathbb{H}^{*},\Gamma_{0})$
such that
$$
\sup\{
\mathbb{P}(x_0,x,\varphi)
\mid
\varphi\in \partial^{ue}\Bers{x_0}-
N(\varphi_H;\delta)
\}\le \left(\frac{2}{\delta^{2}}\right)^{\convgenus}
e^{-2\convgenus d_T(x_0,x)}
$$
for $x\in \Bers{x_0}\cap U$.
\end{lemma}

\begin{proof}
We first claim that
$$
U'=\{x\in \teich_{g,m}\mid 
\mbox{$e^{-2d_T(x_0,x)}\ext_{x}(G)\ge \delta^2/2$ for $G\in \mathcal{SMF}^{ue}_{x_0}$ with $i(H,G)\ge \delta$}\}
$$
satisfies that $U\cap \Bers{x_0}\subset U'$ for some neighborhood $U$ of $\varphi_H$ in 
is a neighborhood of $\varphi_H$ in the sense that there is a neighborhood $U$ of $A_{2}(\mathbb{H}^{*},\Gamma_{0})$.

Otherwise, there is a sequence $\{x_n\}_{n=1}^\infty$ in $\teich_{g,m}$ converging to $\varphi_H$ and $\{G_n\}_n\subset\mathcal{SMF}^{ue}_{x_0}$ such that $e^{-2d_T(x_0,x_n)}\ext_{x_n}(G_n)<\delta^2/2$ for some $G_n\in \mathcal{SMF}^{ue}_{x_0}$ with $i(G_n,H)\ge \delta$. We may assume that $G_n$ converges to some $G_0\in \mathcal{SMF}_{x_0}$. Since the intersection number is continuous, $i(G_0,H)\ge \delta$.
From Claim \ref{claim:3-1} in Theorem \ref{thm:Poisson-kernel},
\eqref{eq:ex-geo2} and \eqref{eq:ex-geo3},
\begin{align*}
\label{eq:extremal_length_ue}
\delta^2/2 &> e^{-2d_T(x_0,x_n)}\ext_{x_n}(G_n)
=i_{x_0}(x_n,[G_n]) \\
&\to i_{x_0}([H],[G_0])=
\frac{i(G_0,H)^2}{\ext_{x_0}(G_0)\ext_{x_0}(H)}
=i(G_0,H)^2\ge \delta^2
\nonumber
\end{align*}
as $n\to \infty$, which is a contradiction.

%
%
%

We show that the open neighborhood $U$ which is taken above satisfies the desired condition.
Indeed, for $x\in U'$, we deduce
$$
\mathbb{P}(x_0,x,\varphi_G)=
\left(\frac{\ext_{x_0}(G)}{\ext_{x}(G)}
\right)^{\convgenus}
\le
\left(\frac{2}{\delta^{2}}\right)^{\convgenus}
e^{-2\convgenus d_T(x_0,x)}.
$$
Since the right-hand side is independent of $\varphi_G\in \partial^{ue}\Bers{x_0}
-
N(\varphi_H;\delta)$,
we have the assertion.
\end{proof}

\subsection{Proof of Theorem \ref{thm:boundary-behavior}}
We prove Theorem \ref{thm:boundary-behavior} with a weaker assumption.
Suppose $V$ is integrable on $\partial\Bers{x_0}$ and the restriction of $V$ to $\partial^{mf}\Bers{x_0}$ is continuous at $\varphi_0\in \partial^{ue}\Bers{x_0}$.

Fix $\epsilon>0$.
From Lemma \ref{lem:neighborhood-H},
there is $\delta>0$
such that $|V(\varphi)-V(\varphi_0)|<\epsilon$
for $\varphi\in N(\varphi_H;\delta)$.
Since
$d\pushThursMBers_{x}=\mathbb{P}(x_0,x,\cdot)
d\pushThursMBers_{x_0}$ is a probability measure on $\partial\Bers{x_0}$
for $x\in \Bers{x_0}$,
\begin{equation}
\label{eq:bdy-beh-1}
\int_{N(\varphi_H;\delta)}
|V(\varphi)-V(\varphi_0)|
\mathbb{P}(x_0,x,\varphi)
d\pushThursMBers_{x_0}(\varphi)<\epsilon.
\end{equation}
Since $V$ is integrable on $\partial \Bers{x_0}$
and $\partial^{ue}\Bers{x_0}$ is of full measure in $\partial\Bers{x_0}$
with respect to $\pushThursMBers_{x_0}$ (Corollary \ref{coro:support-PH-measure}),
from Lemma \ref{lem:Poisson-kernel},
there is a neighborhood $U$ of $\varphi_H$ in $A_{2}(\mathbb{H}^{*},\Gamma_{0})$ such that 
\begin{equation}
\label{eq:bdy-beh-2}
\int_{\partial \Bers{x_0}-N(\varphi_H;\delta)}
|V(\varphi)-V(\varphi_0)|
\mathbb{P}(x_0,x,\varphi)
d\pushThursMBers_{x_0}(\varphi)
\le Me^{-2\convgenus d_T(x_0,x)}
\end{equation}
for $x\in \Bers{x_0}\cap U$,
where $M>0$ depends only on $V$, $\varphi_0$ and $\delta$.
From \eqref{eq:bdy-beh-1} and \eqref{eq:bdy-beh-2},
we conclude
\begin{align*}
&\left|\int_{\partial \Bers{x_0}}V(\varphi)
\mathbb{P}(x_0,x,\varphi)
d\pushThursMBers_{x_0}(\varphi)-
V(\varphi_0)
\right|
\\
&\le\int_{\partial \Bers{x_0}}
|V(\varphi)-V(\varphi_0)|
\mathbb{P}(x_0,x,\varphi)
d\pushThursMBers_{x_0}(\varphi)
\le \epsilon+Me^{-2\convgenus d_T(x_0,x)}
\end{align*}
for $x\in\Bers{x_0}\cap U$.
\qed


\section{Averaging on $\mathcal{PMF}$}
\label{sec:averaging-pmf}
We discuss
on the integral representation from the topological point of view.

\subsection{Integral representation with $\mathcal{PMF}$}
\label{subsec:integral-representation-with-pmf}
We identify $\mathcal{SMF}_{x_0}$ with $\mathcal{PMF}$  as \S\ref{subsec:Thurstonmeasure}.
We think of $\PThursM^{x_0}$
as a Borel measure on $\mathcal{PMF}$ under the identification.
We define a linear operator (isometry)
$$
L^1(\partial \Bers{x_0},\pushThursMBers_{x_0})
\ni V\mapsto \hat{V}=
V\circ \Xi_{x_0}\in L^1(\mathcal{PMF},\PThursM^{x_0}).
$$
The following is an immediate consequence from
Theorem \ref{thm:Poisson-integral-formula}.

\begin{corollary}[Integral representation with $\mathcal{PMF}$]
\label{coro:integral-representation-PMF}
Let $V$ be a pluriharmonic function on $\teich_{g,m}$ which is continuous on the Bers closure.
Then,
\begin{equation}
\label{eq:integral-representation-PMF}
V(x)=\int_{\mathcal{PMF}}\hat{V}([F])
\left(\sqrt{\frac{\ext_{x_0}(F)}{\ext_x(F)}}\right)^{2\convgenus}
d\PThursM^{x_0}([F])
\end{equation}
for $x\in \teich_{g,m}$.
\end{corollary}

\begin{remark}
\label{remark:PS-measure}
The family of measures
$$
\left\{
\left(\sqrt{\frac{\ext_{x_0}(\cdot )}{\ext_x(\cdot )}}\right)^{2\convgenus}d\PThursM^{x_0}
\right\}_{x\in \teich_{g,m}}
$$
on $\mathcal{PMF}$ which are
appeared in the right-hand side \eqref{eq:integral-representation-PMF}
is already discussed in \cite[\S2.3.1]{MR2913101},
and recognized as the conformal density of dimension $2\convgenus=6g-6+2m$
(the Patterson-Sullivan measures) on $\teich_{g,m}$
from the dynamical point of view
(see also \cite{MR2495764}).
\end{remark}

We prove Mirzakhani and Dumas' observation in \cite{MR3413977}
by using the formulation as in Corollary \ref{coro:integral-representation-PMF}
as follows.
\begin{corollary}[Mirzakhani and Dumas \cite{MR3413977}]
\label{coro:Hubbard-Masur-constant}
The Hubbard-Masur function \eqref{eq:HM-fns} is constant.
\end{corollary}

\begin{proof}
Fix $x_0\in \teich_{g,m}$. Let $x\in \teich_{g,m}$.
By applying $V\equiv 1$ on $A_{2}(\mathbb{H}^{*},\Gamma_{0})$ to Corollary \ref{coro:integral-representation-PMF}, we obtain
\begin{align}
1
&=
\int_{\mathcal{PMF}}\left(\frac{\ext_{x_0}(F)}{\ext_x(F)}\right)^{\convgenus}d\PThursM^{x_0}([F])
=\int_{\mathcal{SMF}_{x_0}}\left(\frac{\ext_{x_0}(F)}{\ext_x(F)}\right)^{\convgenus}d\PThursM^{x_0}(F)
\label{eq:Hubbard-Masur-constant1}
\\
&
=\frac{1}{{\rm Vol}_{Th}(x_0)}
\int_{\mathcal{BMF}_{x_0}}\left(\frac{\ext_{x_0}(F)}{\ext_x(F)}\right)^{\convgenus}d\ThursM(F)
\nonumber
\end{align}
for $x\in \teich_{g,m}$.
On the other hand,
the mapping
$$
\mathcal{BMF}_{x_0}\ni F\mapsto 
\sqrt{\frac{\ext_{x_0}(F)}{\ext_x(F)}}
F\in \mathcal{BMF}_x
$$
is homeomorphic (the origin is sent to the origin).
Therefore,
the last term of \eqref{eq:Hubbard-Masur-constant1}
coincides with ${\rm Vol}_{Th}(x)/{\rm Vol}_{Th}(x_0)$.
\end{proof}

From \eqref{eq:push-forward-marking_change} and Corollary \ref{coro:Hubbard-Masur-constant}, we also obtain the following.

\begin{corollary}
\label{coro:pull-back-thurston-measure}
For $x,y\in \teich_{g,m}$, after identifying $\mathcal{SMF}_x$ and $\mathcal{SMF}_y$ with $\mathcal{PMF}$ as \S\ref{subsec:Thurstonmeasure}, we have
\begin{equation*}
\PThursM^{y}(E)=\int_E\left(\frac{\ext_{x}(F)}{\ext_{y}(F)}\right)^{\convgenus}d\PThursM^{x}([F])
\end{equation*}
for any measurable set $E\subset \mathcal{PMF}$.
\end{corollary}

\subsection{Phenomena by averaging}
\label{subsec:phenomena-by-averaging}


 In this section, we discuss the averaging procedure from the Poisson integral formula \eqref{eq:main-Poisson-integral-formula}. 
 
\subsubsection{Presentation of differentials by averaging}
\label{subsubsec:presentation-differential}
Let $V$ be a pluriharmonic function on $\Bers{x_0}$ which is continuous on the Bers closure.
We identify the holomorphic cotangent bundle over $\teich_{g,m}$ with the space of holomorphic quadratic differentials as \S\ref{subsec:infinitesimal-Teich}. The following formula is deduced by the differentiating the both sides of \eqref{eq:integral-representation-PMF}:
\begin{align}
(\partial V)_x&=\convgenus \int_{\mathcal{PMF}}\hat{V}([F])
\left(\frac{\ext_{x_0}(F)}{\ext_x(F)}\right)^{\convgenus}
\frac{q_{F,x}}{\|q_{F,x}\|}
\,
d\PThursM^{x_0}([F])
\label{eq:Diff-holo-functions}
\\
&=\convgenus \int_{\mathcal{PMF}}\hat{V}([F])\frac{q_{F,x}}{\|q_{F,x}\|}\,d\PThursM^{x}([F]) 
\nonumber\\
(\overline{\partial}V)_x&=\convgenus \int_{\mathcal{PMF}}\hat{V}([F])
\left(\frac{\ext_{x_0}(F)}{\ext_x(F)}\right)^{\convgenus}
\frac{\overline{q_{F,x}}}{\|q_{F,x}\|}\,d\PThursM^{x_0}([F])
\label{eq:Diff-holo-functions2}
\\
&=\convgenus \int_{\mathcal{PMF}}\hat{V}([F])\frac{\overline{q_{F,x}}}{\|q_{F,x}\|}\,d\PThursM^{x}([F])
\nonumber
\end{align}
for $x\in \teich_{g,m}$ from Gardiner's formula (\cite{MR736212}) and Corollary \ref{coro:pull-back-thurston-measure}
since $\|q_{F,x}\|=\ext_x(F)$. 
Equation \eqref{eq:Diff-holo-functions2} is deduced from the equation $\overline{\partial}F=\overline{\partial \overline{F}}$ for a $C^1$-function $F$.
Equations \eqref{eq:Diff-holo-functions} and \eqref{eq:Diff-holo-functions2} mean that the $\partial$ and $\overline{\partial}$-differentials are obtained by averaging the boundary value with the vector-valued (quadratic differential-valued) measures
$$
\left\{\convgenus\frac{q_{\,\cdot\,,x}}{\|q_{\,\cdot\,,x}\|}\,d\PThursM^{x},\
\convgenus\frac{\overline{q_{\,\cdot\,,x}}}{\|q_{\,\cdot\,,x}\|}\,d\PThursM^{x}\right\}_{x\in \teich_{g,m}}.
$$
Thus, for $\hat{V}\in L^1(\mathcal{PMF},\PThursM^{x_0})$,
the homogeneous tangential Cauchy-Riemann equation \eqref{eq:CR}
is rephrased as 
\begin{equation}
\label{eq:remark-CR}
\int_{\mathcal{PMF}}\hat{V}([F])
\frac{\overline{
q_{F,x}}}{\|q_{F,x}\|}
\,
d\PThursM^{x}([F])=0\quad (x\in \teich_{g,m}).
\end{equation}
%

\subsubsection{Differentials for lengths of hyperbolic geodesics}
\label{subsubsec:Wolpert_differentials}
We give an application of \eqref{eq:Diff-holo-functions}.
We use the notation defined in \S\ref{subsec:Bers-slice} and \S\ref{subsec:A-covering-partial-Bers} frequently.

For $\gamma\in \pi_1(\Sigma_{g,m})$, denote by $\ell_\gamma(x)={\rm leng}_x(\gamma)$ the hyperbolic length of the hyperbolic geodesic on a marked Riemann surface $x$ in the class $\gamma$ (cf. \S\ref{subsec:extremal-length}). Wolpert discussed a Petersson series which defines a holomorphic quadratic differential $\Theta_{\gamma,x}\in \mathcal{Q}_x$ satisfying that 
\begin{equation}
\label{eq:Wolpert}
d\ell_\gamma[v]=
{\rm Re}\langle v,\Theta_{\gamma,x}\rangle
\end{equation}
for $v\in T_x\teich_{g,m}$ which follows from the Gardiner variational formula. The quadratic differential $\Theta_{\gamma,x}$ is a fundamental object in the Weil-Petersson geometry (e.g. \cite[\S7, \S8]{MR1215481} and \cite[Chapter 3]{MR2641916}). 
\begin{proof}[Proof of Theorem \ref{thm:Wolpert-Hubbard-Masur}]
For $x\in \teich_{g,m}$ and $y\in \teich_{g,m}\cup \mathcal{PMF}^{mf}$, we denote by $\rho_{y,x}=\rho_{\varphi_{y}}$ where $\varphi_{y}\in \overline{\Bers{x}}$ is the corresponding differential to $y$ (cf. \S\ref{subsec:BoundarygroupswithoutAPT}). 
%

Let $v\in T_{x}\teich_{g,m}$. From \eqref{eq:Diff-holo-functions}, we have
$$
\left(\partial {\rm tr}^2\rho_{\varphi_{\cdot,x_0}}(\gamma)\right)[v]=\convgenus\int_{\mathcal{PMF}^{mf}}{\rm tr}^2\rho_{\varphi_{[F],x_0}}(\gamma)\,\frac{\langle v,q_{F,x}\rangle}{\|q_{F,x}\|} d\PThursM^{x}([F])
$$
for $x_0\in \teich_{g,m}$.
Since $\rho_{x,x}$ is the Fuchsian representation of $x$, from the argument in the proof of the variation of the hyperbolic length, we can check
\begin{equation}
\label{eq:sinh-Theta}
\left(\partial {\rm tr}^2\rho_{\varphi_{\cdot,x}}(\gamma)\right)[v]=4\sinh(\ell_\gamma(x))\cdot \frac{1}{2}\langle v,\Theta_{\gamma,x}\rangle
\end{equation}
(e.g. \cite[Theorem 8.3]{MR1215481}). This implies what we wanted.
\end{proof}

\subsubsection{The case of $(1,1)$}
In the case of $(g,m)=(1,1)$, we give a concrete explanation of Theorem \ref{thm:Wolpert-Hubbard-Masur}. 
Since the identification $\boldsymbol{\beta}\colon \mathbb{H}\cong \teich_{1,1} \to \Bers{\tau_0}$ is the Riemann mapping which sends $\tau_0$ to $0$,
$$
\mathbb{H}\ni \tau\mapsto {\rm Tr}_\gamma(\tau):={\rm tr}^2\left(\rho_{\varphi_{\boldsymbol{\beta}(\tau),\tau_0}}(\gamma)\right)
$$
is a holomorphic function which extends continuously to $\mathbb{H}\cup \hat{\mathbb{R}}$ and satisfies ${\rm Tr}_\gamma(\xi)={\rm tr}^2\left(\rho_{\varphi_{F_{[a,b]},\tau_0}}(\gamma)\right)$ when $u=-b/a$ (cf. \S\ref{subsec:Case_once_tori}). Notice that the representation $\rho_{\varphi_{F_{[a,b]},\tau_0}}$ is well-defined in this case even when $u\in \hat{\mathbb{Q}}$, since the complement of a simple closed curve in a once punctured torus is a three hold sphere. From the residue theorem, the right-hand side of \eqref{eq:Wolpert-Hubbard-Masur} is equal to
\begin{equation}
\label{eq:OPT_2}
\frac{1}{2\sinh(\ell_\gamma(\tau_0))}\left(\int_{\hat{\mathbb{R}}}{\rm Tr}_\gamma(u)\frac{-1}{\pi}\frac{du}{(u-\tau_0)^2}\right)dz^2
=\frac{-\sqrt{-1}}{\sinh(\ell_\gamma(\tau_0))}\frac{d{\rm Tr}_\gamma}{d\tau}(\tau_0)dz^2
\end{equation}
for $\gamma\in \pi_1(\Sigma_{1,1})$. 

Notice that the differential $\partial/\partial \tau$ on $\mathbb{H}$ at $\tau=\tau_0$ is induced by the Beltrami differential $\mu=(\sqrt{-1}/(2{\rm Im}(\tau_0))(d\overline{z}/dz)$ on $M_{\tau_0}$ (e.g. \cite[\S7.2]{MR3715450}). This means that for a smooth function $F$ around $\tau=\tau_0$, if $\partial F$ is associated to $Adz^2\in \mathcal{Q}_{M_{\tau_0}}$, 
$$
\frac{\partial F}{\partial \tau}(\tau_0)=\langle\mu,Adz^2\rangle=\frac{\sqrt{-1}}{2}A.
$$
Thus, from \eqref{eq:sinh-Theta}, \eqref{eq:OPT_2} implies \eqref{eq:Wolpert-Hubbard-Masur} in the case when $(g,m)=(1,1)$.

\section{Questions}
\subsection{}

Theorem \ref{thm:Poisson-integral-formula}, Corollary \ref{thm:boundary-value} and Corollary \ref{coro:integral-representation-PMF} give an interaction between holomorphic functions on Teichm\"uller space and measurable functions on the Bers boundary and $\mathcal{PMF}$. A natural problem from our integral formula is:
\begin{question}
Determine the classes of holomorphic or pluriharmonic functions on $\teich_{g,m}$ to which the Poisson integral formula \eqref{eq:main-Poisson-integral-formula} apply. 
\end{question}
We have already stated the homogeneous tangential Cauchy-Riemann equation (in the distribution sense) in \eqref{eq:CR} and \eqref{eq:remark-CR}. Since they are given by integration, it is hard to derive infinitesimal properties of the boundary functions.

\subsection{}
Our Poisson integral formula is for pluriharmomonic functions which are continuous on the Bers compactifications. Since the Bers slices depend on the choice of the base point, the class of pluriharmonic functions continuous up to a Bers boundary possibly looks like the wrong object of study (the author thanks referees for pointing it out). On the other hand, as noticed in \S\ref{subsec:History}, any holomorphic function on the Teichm\"ulller space is approximated by holomorphic functions which are continuous up to the Bers boundary. Hence, holomorphic functions which are continuous up to the Bers boundary would be worth to study in some sense. 

\begin{question}
Fix a base point $x_0\in \teich_{g,m}$.
When measurable functions on $\mathcal{PMF}^{mf}$ or $\partial^{mf} \Bers{x_0}$ extend as pluriharmonic (or holomorphic) functions on $\Bers{x_0}$ which are continuous on the Bers boundary?
\end{question}
This question will be related to a problem which asks how the Bers slices depend the base points. Namely, even if some measurable function on $\mathcal{PMF}^{mf}$ extends continuously on $\overline{\Bers{x_0}}$ and pluriharmonically on $\Bers{x_0}$, it will not do for $\Bers{x_1}$ for some $x_1\ne x_0$, This problem originates from Kerckhoff and Thurston's observation \cite{MR1037141}. 

\subsection{}
In this paper, we settle the Poisson integral formula on the Bers compactification. As noticed in \S\ref{subsec:History}, there are many embeddings (slices) which realize the Teichm\"uller space. For instance, the Maskit slice \cite{MR346149} is a version of the upper-half space model of the Teichm\"uller space.
\begin{question}
Study the Poisson integral formula for various slices of Teichm\"uller spaces.
\end{question}


\bibliographystyle{plain}
\bibliography{References-1}

\end{document}